\documentclass[a4paper,twoside]{article}  % 12pt,,12pt

\usepackage{amssymb,amsmath,amsthm,dsfont,amsfonts,color,latexsym,CJK}

\usepackage{hyperref}

\usepackage{mathrsfs} % hyperref,showkeys

\definecolor{blue}{rgb}{0.00,0.00,1.00}
\definecolor{red}{rgb}{1.00,0.00,0.00}

\renewcommand{\baselinestretch}{1.2}
\def\bq{\begin{equation}}
\def\eq{\end{equation}}
\def\ba{\begin{array}{ccc}}
\def\bal{\begin{array}{lll}}
\def\ea{\end{array}}

 \def\lt#1{\left#1}\def\rt#1{\right#1}
\def\({\left(}\def\){\right)}
\def\[{\left[}\def\]{\right]}
\def\<{\langle}\def\>{\rangle}

    \def \O   {\mathbb{O}}

    \def \R   {\mathbb{R}}

    \def\S    {\mathbb{S}}
    \def\eps  {\epsilon}
    \def\intr {\int_{\R^3}}
    \def\intrr {\int_{\R^6}}
    \def\ints {\int_{\S^2}}
    \def\intt {\int^t_0}

    \def \Q    {\mathcal{Q}}
    
    \def \N    {\mathbb{N}}

    \def \pt   {\partial}
    
    \def \Dt   {\frac{\rm d}{{\rm d}t}}
    
    \def \dt    {\partial_t}
    
    \def \da    {\pt^\alpha}

    \def \dx    {\partial_x}

    \def \dxa   {\partial^{\alpha}_x}
    \def \dxb   {\partial^{\beta}_x}
    \def \dv    {\partial_v}
    
    \def \dvb   {\partial^{\beta}_v}

    \def \divx  {{\rm div}_x}

    \def\Tdx   {\nabla_x}
    \def\Tdv   {\nabla_v}

       \def\bq{\begin{equation}}
       \def\eq{\end{equation}}
       \def\be{\begin{equation}}
       \def\ee{\end{equation}}
       \def\bma#1\ema{{\allowdisplaybreaks\begin{align}#1\end{align}}}
       \def\bmas#1\emas{{\allowdisplaybreaks\begin{align*}#1\end{align*}}}
       \def\bln#1\eln{{\allowdisplaybreaks\begin{aligned}#1\end{aligned}}}
       \def\nnm{\notag}
       \def\bgr#1\egr{\allowdisplaybreaks\begin{gather}#1\end{gather}}
       \def\bgrs#1\egrs{\allowdisplaybreaks\begin{gather*}#1\end{gather*}}

       \theoremstyle{plain}
       \newtheorem{lem}{\bf Lemma}[section]
       \newtheorem{thm}[lem]{\textbf{Theorem}}

       \newtheorem{remark}[lem]{\bf Remark}

\begin{document}
\title{ Diffusion Limit of  the Vlasov-Poisson-Boltzmann System }
\author{ Hai-Liang Li$^1$,\, Tong Yang$^2$,\, Mingying Zhong$^3$\\[2mm]
 \emph{\small\it  $^1$School of  Mathematical Science,
    Capital Normal University,  Beijing, China.}\\
    {\small\it E-mail:\ hailiang.li.math@gmail.com}\\
    {\small\it $^2$Department of Mathematics, City University of Hong Kong,  Hong
    Kong, China}\\
{\small \it \& School of Mathematics and Statistics,   Chongqing University,  Chongqing, China}\\
    {\small\it E-mail: matyang@cityu.edu.hk} \\
    {\small\it  $^3$College of  Mathematics and Information Sciences,
    Guangxi University, Nanning, China.}\\
    {\small\it E-mail:\ zhongmingying@sina.com}\\[2mm]
    }
\date{ }
%\date{\textbf{Draft} \quad  \today}

\pagestyle{myheadings}
\markboth{Vlasov-Poisson-Boltzmann system }%
{H.-L. Li, T. Yang, M.-Y. Zhong}

 \maketitle

 \thispagestyle{empty}
\begin{abstract}\noindent
In the present paper, we study the diffusion limit of the classical solution to the unipolar Vlasov-Poisson-Boltzmann (VPB) system with initial data near a global Maxwellian. We prove the convergence and establish the convergence rate of the global strong solution to the unipolar VPB system towards the solution to an incompressible Navier-Stokes-Poisson-Fourier  system based on the spectral analysis with precise estimation on the initial layer.


\end{abstract}
%\footnotetext{2011 Mathematics Subject Classification. 82C40, 82C10, 35Q20  }

\tableofcontents

\section{Introduction}
The Vlasov-Poisson-Boltzmann system (VPB) can be used to model the motion of
the dilute charged particles in plasma or semiconductor devices under the influence of the
self-consistent electric field \cite{Markowich}. In general, the rescaled VPB system for one-species takes the form
\bgr
\dt F_{\eps}+\frac1{\eps} v\cdot\Tdx F_{\eps}+\frac1{\eps}\Tdx \Phi_{\eps}\cdot\Tdv F_{\eps}
=\frac1{\eps^2}\Q(F_{\eps},F_{\eps}), \label{VPB3}\\
\Delta_x\Phi_{\eps}=\intr F_{\eps}dv-1,\label{VPB4}
%\\
%F_{\eps}(0,x,v)=F_0(x,v),\label{VPB4i}
\egr
where $\eps>0$ is a small parameter related to the mean free path, $F_{\eps}=F_{\eps}(t,x,v)$ is the distribution
function with $(t,x,v)\in\R^+\times\R^3\times\R^3$, and  $\Phi_{\eps}(t,x)$ denotes the electric potential, respectively. Throughout this paper, we assume $\eps\in(0,1)$. The collision between particles is given by the
standard Boltzmann collision operator $\Q(f,g)$ as below
\bq \Q(f,g)=\frac12\intr\ints B(|v-v_*|,\omega)(f'_*g'+f'g'_*-f_*g-fg_*)dv_*d\omega, \label{Q}\eq
where
\bgrs
f'_*=f(t,x,v'_*),\quad f'=f(t,x,v'),\quad f_*=f(t,x,v_*),\quad f=f(t,x,v),\\
 v'=v-[(v-v_*)\cdot\omega]\omega,\quad v'_*=v_*+[(v-v_*)\cdot\omega]\omega,\quad \omega\in\S^2.
 \egrs
For monatomic gas, the collision kernel $B(|v-v_*|,\omega)$ is  a
non-negative function of $|v-v_*|$ and $|(v-v_*)\cdot\omega|$:
$$B(|v-v_*|,\omega)=B(|v-v_*|,\cos\theta),\quad \cos\theta=\frac{|(v-v_*)\cdot\omega|}{|v-v_*|},\quad \theta\in\[0,\frac{\pi}2\].$$
%\fi
In the following, we consider both the hard sphere model and hard potential
 with angular cutoff.
Precisely, for the hard
sphere model,
\bq B(|v-v_*|,\omega)=|(v-v_*)\cdot\omega|=|v-v_*|\cos\theta;\eq
and for the models of the hard potentials with Grad angular cutoff assumption,
\bq B(|v-v_*|,\omega)= b(\cos\theta)|v-v_*|^\gamma,\quad 0\le \gamma<1,\eq
where we assume for simplicity
$$0\le b(\cos\theta)\le C|\cos\theta|.$$

There has been significant  progress made  on the well-posedness  and long time behavior of solutions to the Vlasov-Poisson-Boltzmann system for fixed $\eps$.
In particular, the global existence of renormalized weak solution for general large initial data was shown in \cite{Mischler}. The global existence of unique strong solution with the initial data near the normalized global Maxwellian was obtained in a spatially periodic domain in \cite{Guo2} and  in the  whole space in \cite{Duan2,Yang1,Yang3} for hard sphere model, and further in \cite{Duan5,Duan7} for hard potential or soft potential.   As for  the long-time behavior of the VPB system, we refer to the works  \cite{Duan1,Li1,Li2,Wang2,Yang4}. In addition, the spectrum structure and the optimal decay rate of the classical solution were investigated in \cite{Li1,Li2}.

The fluid dynamical limit of the  solution to the VPB system near Maxwellian was also studied in  \cite{Guo4,Wang1}. In
\cite{Guo4}, the authors  proved the convergence
of the solutions to the unipolar VPB system towards a solution to the compressible
Euler-Poisson system in the whole space.  In \cite{Wang1}, the author established a global convergence result of the solution to the bipolar VPB system towards a solution to the incompressible Vlasov-Navier-Stokes-Fourier system.

On the other hand, the diffusion limit to the Boltzmann equation is a classical problem with pioneer work by Bardos-Golse-Levermore in \cite{FL-1}, and significant  progress on the limit of renormalized solutions to Leray solution to Navier-Stokes system in \cite{Golse-SR}.
In contrast to the works on Boltzmann equation~\cite{FL-1,FL-2,FL-3,FL-4,FL-5,Guo5} and the VPB system \cite{Guo4,Wang1}, the diffusion limit  of the classical solution to the unipolar VPB system \eqref{VPB1}--\eqref{VPB2i} has not been given despite of its importance.
Moreover, the convergence rate of the classical solution to the VPB system towards its fluid dynamical limits was not given  in \cite{Guo4,Wang1}.
% Therefore, it is natural to establish a explicit convergence rate of the solution towards the diffusion limit.

In this paper, we study the diffusion limit
of the strong solution to the rescaled VPB system~\eqref{VPB3}--\eqref{VPB4}  with initial data near the normalized Maxwellian $M(v)$, where
$$
 M=M(v)=\frac1{(2\pi)^{3/2}}e^{-\frac{|v|^2}2},\quad v\in\R^3.
$$
%Our motivation  is to prove the VPB system~\eqref{VPB3}--\eqref{VPB4} gives the incompressible Navier-Stokes-Poisson-type system if the solution $F_{\eps}$ remains near the absolute Maxwellian $M=M(v)$ with the distance of order $\eps$.

Hence, we denote the perturbation of $F_{\eps}$ and $\Phi_{\eps}$ as
$$F_{\eps}=M+\eps \sqrt M f_{\eps},\quad \Phi_{\eps}=\eps \phi_{\eps}.$$
Then the VPB system \eqref{VPB3}--\eqref{VPB4} for $f_{\eps}$ and $\phi_{\eps}$ is
\bgr
\dt f_{\eps}+\frac1{\eps} v\cdot\Tdx f_{\eps}-\frac1{\eps} v\sqrt{M}\cdot\Tdx \phi_{\eps}-\frac1{\eps^2} Lf_{\eps}
= G_1(f_{\eps})+\frac1{\eps} G_2(f_{\eps}),\label{VPB1}\\
\Delta_x\phi_{\eps}= \intr f_{\eps}\sqrt{M}dv,\label{VPB2}
\egr
with the initial condition
\be f_{\eps}(0,x,v)=f_0(x,v), \label{VPB2i}\ee
by assuming that the initial data $f_0$ is independent of $\eps$.
As usual, the linear operator $L$ and the nonlinear operators $G_1,G_2$ are given by
\be \label{gamma}
\left\{\bln
Lf_{\eps}&=\frac1{\sqrt M}[\Q(M,\sqrt{M}f_{\eps})+\Q(\sqrt{M}f_{\eps},M)],\\
G_1(f_{\eps})&=\frac12 (v\cdot\Tdx\phi_{\eps})f_{\eps}-\Tdx\phi_{\eps}\cdot\Tdv f_{\eps},\\
G_2(f_{\eps})&=\Gamma(f_{\eps},f_{\eps})=\frac1{\sqrt M}\Q(\sqrt{M}f_{\eps},\sqrt{M}f_{\eps}).
\eln\right.
\ee

The linearized collision operator $L$ can be written as
\be\label{L_1}
\left\{\bln
(Lf)(v)&=(Kf)(v)-\nu(v) f(v),\\
\nu(v)&=\intr\ints B(|v-v_*|,\omega)M_*d\omega dv_*,\\
(Kf)(v)&=\intr\ints
B(|v-v_*|,\omega)(\sqrt{M'_*}f'+\sqrt{M'}f'_*-\sqrt{M}f_*)\sqrt{M_*}d\omega
dv_*\\
&=\intr k(v,v_*)f(v_*)dv_*,
\eln\right.
\ee
where $\nu(v)$,  the collision frequency, is a real  function, and $K$ is a self-adjoint compact operator
on $L^2(\R^3_v)$ with a real symmetric integral kernel $k(v,v_*)$.
In addition, $\nu(v)$ satisfies
\be
\nu_0(1+|v|)^\gamma\leq\nu(v)\leq \nu_1(1+|v|)^\gamma, \label{nu}
\ee
with $\gamma=1$ for hard sphere and $0\le \gamma<1$ for hard potential.

The nullspace of the operator $L$, denoted by $N_0$, is a subspace
spanned by the orthonormal basis $\{\chi_j,\ j=0,1,\cdots,4\}$  with
\bq \chi_0=\sqrt{M},\quad \chi_j=v_j\sqrt{M} \,\, (j=1,2,3), \quad
\chi_4=\frac{(|v|^2-3)\sqrt{M}}{\sqrt{6}}.\label{basis}\eq

Let   $L^2(\R^3)$ be a Hilbert space of complex-value functions $f(v)$
on $\R^3$ with the inner product and the norm
$$
(f,g)=\intr f(v)\overline{g(v)}dv,\quad \|f\|=\(\intr |f(v)|^2dv\)^{1/2}.
$$
Denote the standard   macro-micro decomposition as follows
\be \label{P10}
\left\{\bal
f=P_0f+P_1f,\\
 P_0f=\sum_{k=0}^4(f,\chi_k)\chi_k,\,\,\,  P_1f=f-P_0f,
 \ea\right.
 \ee
From the Boltzmann's H-theorem, $L$ is non-positive and
moreover, $L$ is locally coercive in the sense that there is a constant $\mu>0$ such that \bma
 (Lf,f)&\leq -\mu \| P_1f\|^2, \quad  \ f\in D(L),\label{L_3}
 \ema
where $D(L)$ is the domains of $L$ given by
$$ D(L)=\left\{f\in L^2(\R^3)\,|\,\nu(v)f\in L^2(\R^3)\right\}.$$
Without the loss of generality, we assume in this paper that $\nu(0)\ge \nu_0\ge \mu>0$.

This paper  aims to prove the
convergence and establish the convergence rate of strong solutions $(f_{\eps}, \phi_{\eps})$ to \eqref{VPB1}--\eqref{VPB2i} towards  $(u , \phi)$, where $u=n \chi_0+m\cdot v\chi_0+q\chi_4$ and $(n,m,q,\phi)(t,x)$ is the solution of the following incompressible Navier-Stokes-Poisson-Fourier (NSPF) system:
\bgr
\Tdx\cdot m=0,\label{NSP_1}\\
n+\sqrt{\frac23}q-\phi=0,\label{NSP_2}\\
\dt m-\kappa_0\Delta_x m +\Tdx p=n\Tdx\phi-\Tdx\cdot (m\otimes m),\label{NSP_3}\\
\dt \bigg(q-\sqrt{\frac23}n\bigg)-\kappa_1\Delta_x q=\sqrt{\frac23} m\cdot\Tdx\phi-\frac53\Tdx\cdot (qm),\label{NSP_4}\\
\Delta_x\phi=n,\label{NSP_5}
\egr
where $p$ is the pressure, and the initial data $(n,m,q)(0)$ satisfies
\bgr
 m(0)=Pm_0,\quad q(0)-\sqrt{\frac23}n(0)= q_0-\sqrt{\frac23}n_0,\label{NSP_5i}\\
n(0) -\Delta^{-1}_xn(0)+\sqrt{\frac23} q(0)=0. \label{compatible}
\egr
Here, $P$ is the projection to the divergence-free subspace and $n_0$, $m_0$ and $q_0$ are given by
\be n_0=(f_0,\chi_0), \quad m_0=(f_0,v\chi_0),\quad q_0=(f_0,\chi_4), \ee
and the viscosity coefficients $\kappa_0$, $\kappa_1>0$ are defined by
\be
\kappa_0=-(L^{-1}P_1(v_1\chi_2),v_1\chi_2),\quad \kappa_1=-(L^{-1}P_1(v_1\chi_4),v_1\chi_4). \label{coe}
\ee

In general, the convergence is not uniform near $t=0$ because of the appearance of  an initial layer.
However, we can show that if the initial data $f_0$ satisfies
\be \label{initial}
\left\{\bln
&f_0(x,v)=n_0(x)\chi_0+m_0(x)\cdot v\chi_0+q_0(x)\chi_4,\\
&\Tdx\cdot m_0=0,\quad n_0-\Delta^{-1}_xn_0+\sqrt{\frac23} q_0=0,
\eln\right.
\ee
then the uniform convergence is up to $t=0$.
\medskip

%\noindent\textbf{Motivation:}
We give a formally derivation of the macroscopic equations \eqref{NSP_1}--\eqref{NSP_5} as follows. Set
\bmas
f_{\eps}&=f_0+\eps f_1+\eps^2 f_2+\cdots,\\
\phi_{\eps}&=\phi_0+\eps \phi_1+\eps^2 \phi_2+\cdots.
\emas
Substituting the above expansion into \eqref{VPB1}, we have the following Hilbert expansion:
\bma
\frac1{\eps^2}: &\,\ Lf_0=0,  \nnm\\
\frac1{\eps}: &\,\ v\cdot\Tdx f_0-v\sqrt M\cdot\Tdx\phi_0-Lf_1=\Gamma(f_0,f_0),\nnm\\
&\,\ \Delta_x\phi_0=\intr f_0\sqrt Mdv, \label{e1}\\
\eps^0: &\,\ \dt f_0+v\cdot\Tdx f_1-v\sqrt M\cdot\Tdx\phi_1-Lf_2\nnm\\
&\,\ =\frac12v\cdot\Tdx\phi_0f_0-\Tdx\phi_0\cdot\Tdv f_0+\Gamma(f_1,f_0)+\Gamma(f_0,f_1),\nnm\\
&\,\ \Delta_x\phi_1=\intr f_1\sqrt Mdv.\nnm
\ema
From the $\eps^{-2}$ step in \eqref{e1}, we deduce that $f_0\in N_0$, i.e.,
$$ f_0=n_0\chi_0+m_0\cdot v\chi_0+q_0\chi_4.$$
Taking the inner product between the $\eps^{-1}$ step in \eqref{e1} and $\{\chi_0,v\chi_0,\chi_4\}$, we have
\be
\Tdx\cdot m_0=0, \quad \Tdx n_0-\Tdx\phi_0+\sqrt{\frac{2}{3}}\Tdx q_0=0. \label{e4}
\ee
From the $\eps^{-1}$ step in \eqref{e1}, we can represent $f_1$ by
\be
P_1f_1=L^{-1}P_1(v\cdot\Tdx f_0)-L^{-1}\Gamma(f_0,f_0).\label{f1}
\ee
By taking the inner product between the $\eps^0$ step in \eqref{e1} and $\{\chi_0,v\chi_0,\chi_4\}$, and using \eqref{f1}, we have
\bgr
\dt n_0+ \Tdx\cdot m_1 =0,\label{e2}\\
\dt m_0-\kappa_0\Delta_x m_0 +\Tdx p_0=n_0\Tdx\phi_0-\Tdx\cdot (m_0\otimes m_0), \label{e6}\\
\dt q_0+\sqrt{\frac23}\Tdx\cdot m_1-\kappa_1\Delta_x q_0=\sqrt{\frac23} m_0\cdot\Tdx\phi_0-\frac53\Tdx\cdot  (q_0m_0),\label{e3}
\egr
where $(n_1,m_1,q_1)$ are the macroscopic density, momentum and energy of $f_1$ given by
$$n_1=(f_1,\chi_0),\quad m_1=(f_1,v\chi_0),\quad q_1=(f_1,\chi_4),$$
and $p_0$ is the pressure determined by $f_0$ and $f_1$:
$$ p_0=n_1-\phi_1+\sqrt{\frac23}q_1-\frac13 m_0^2.
$$

By taking the summation $-\sqrt{\frac23}\eqref{e2}+\eqref{e3}$, we have
\be
\dt \bigg(q_0-\sqrt{\frac23}n_0\bigg)-\kappa_1\Delta_x q_0=\sqrt{\frac23} m_0\cdot\Tdx\phi_0-\frac53\Tdx\cdot (q_0m_0).\label{e5}
\ee
By combining \eqref{e4}, \eqref{e6} and \eqref{e5}, we obtain \eqref{NSP_1}--\eqref{NSP_5}.

\bigskip

\noindent\textbf{Notations:} \ \ Before state the main results in this paper, we list some notations. Throughout this paper, $C$ denotes a generic positive constant. For any $\alpha=(\alpha_1,\alpha_2,\alpha_3)\in \mathbb{N}^3$ and $\beta=(\beta_1,\beta_2,\beta_3)\in \mathbb{N}^3$, we denote
$$\dxa=\pt^{\alpha_1}_{x_1}\pt^{\alpha_2}_{x_2}\pt^{\alpha_3}_{x_3},\quad \dvb=\pt^{\beta_1}_{v_1}\pt^{\beta_2}_{v_2}\pt^{\beta_3}_{v_3}.$$
The Fourier transform of $f=f(x,v)$ is defined by
$$\hat{f}(\xi,v)=\mathcal{F}f(\xi,v)=\frac1{(2\pi)^{3/2}}\intr e^{-i x\cdot\xi}f(x,v)dx,$$
where and throughout this paper we denote $i=\sqrt{-1}$.

Denote the Sobolev space $ H^N_{k}=\{\,f\in L^2(\R^3_x\times \R^3_v)\,|\,\|f\|_{H^N_{k}}<\infty\}$
equipped with the norm
$$
 \|f\|_{H^N_{k}}=\sum_{|\alpha|+|\beta|\le N}\|\nu^k\dxa\dvb f\|_{L^2(\R^3_x\times \R^3_v)}.
$$
 For $q\ge1$, denote
\bmas
L^2_v(L^q_x)=L^2(\R^3_v,L^q(\R^3_x)),\quad
\|f\|_{L^2_v(L^q_x)}=\bigg(\intr\bigg(\intr|f(x,v)|^q dx\bigg)^{2/q}dv\bigg)^{1/2},
\\
L^q_x(L^2_v)=L^q(\R^3_x,L^2(\R^3_v)),\quad
\|f\|_{L^q_x(L^2_v)}=\bigg(\intr\bigg(\intr|f(x,v)|^2 dv\bigg)^{q/2}dx\bigg)^{1/q}.
\emas
In the following, we denote by $\|\cdot\|_{L^2_{x,v}}$  the norm of the function space $L^2(\R^3_x\times \R^3_v)$, and denote by $\|\cdot\|_{L^q_x}$ and $\|\cdot\|_{L^q_v}$  the norms of the function spaces $L^q(\R^3_x)$ and $L^q(\R^3_v)$ respectively.
For any integer $k\ge1$, we denote by $\|\cdot\|_{H^k_x}$ and $\|\cdot\|_{L^2_v(H^k_x)}$ the norms in the spaces $H^k(\R^3_x)$ and $L^2(\R^3_v,H^k(\R^3_x))$ respectively.

\bigskip

Now we are ready to state  main results in this paper.

\begin{thm}\label{thm1.1}
Let $N\ge 4$. For any $\eps\in (0,1)$, there exists a small constant $\delta_0>0$ such that if $\|f_{0}\|_{H^N_1} +\|f_{0}\|_{L^{2}_v(L^1_x)}\le \delta_0$, then the VPB system \eqref{VPB1}--\eqref{VPB2i} admits a unique global solution  $f_{\eps}(t)= f_{\eps}(t,x,v)$ satisfying the following  time-decay estimate:
\be
   \|f_{\eps}(t)\|_{H^N_1}+\|\Tdx\phi_{\eps}(t)\|_{H^N_x } \le C\delta_0 (1+t)^{-\frac14}, %\label{time8}
\ee
where $\phi_{\eps}(t,x)=\Delta^{-1}_x(f_{\eps}(t,x),\chi_0)$ and $C>0$ is a constant independent of $\eps$.

 There exists a small constant $\delta_0>0$ such that if $\|f_{0}\|_{L^2_v(H^N_x)}+\|f_0\|_{L^{2}_v(L^1_x)}\le \delta_0$, then
the NSPF system \eqref{NSP_1}--\eqref{compatible} admits a unique global solution $(n,m,q)(t,x)\in L^{\infty}_t(H^N_x)$.  Moreover,  $u(t,x,v)=n(t,x)\chi_0+m(t,x)\cdot v\chi_0+q(t,x)\chi_4$ has the following time-decay rate:
\be
   \|u(t)\|_{L^2_v(H^N_x)}+\|\Tdx\phi(t)\|_{H^N_x } \le C\delta_0 (1+t)^{-\frac34}, %\label{time8}
\ee
where $\phi(t,x)=\Delta^{-1}_xn(t,x)$ and $C>0$ is a constant.
\end{thm}

\begin{thm}\label{thm1.2}
Let $(f_{\eps},\phi_{\eps})=(f_{\eps}(t,x,v),\phi_{\eps}(t,x)) $ be the global solution to the VPB system \eqref{VPB1}--\eqref{VPB2i}, and let $(n,m,q,\phi)=(n,m,q,\phi)(t,x)$ be the global solution to the NSPF system \eqref{NSP_1}--\eqref{compatible}. Then, there exists a small constant $\delta_0>0$ such that if $\|f_{0}\|_{H^6_1} +\|f_{0}\|_{L^{2}_v(L^1_x)} +\|\Tdx\Delta^{-1}_x(f_0,\chi_0)\|_{L^p_x}\le \delta_0$ with $p\in(1,2)$, we have
\be
\|f_{\eps}(t)-u(t)\|_{L^{\infty}_x(L^{2}_v)}+\|\Tdx\phi_{\eps}(t)-\Tdx\phi(t)\|_{L^{\infty}_x}\le C\delta_0\(\eps^{a} (1+t)^{-\frac12} +\(1+\eps^{-1}t\)^{-b}\), \label{limit_1}
\ee
where $u(t,x,v)=n(t,x)\chi_0+m(t,x)\cdot v\chi_0+q(t,x)\chi_4$, $b=\min\{1,p'\}$ with $p'=3/p-3/2\in (0,3/2)$, and $a=b$ when $b<1$;
$a=1+2\log_{\eps}|\ln \eps| $  when $b= 1$.

Moreover, if the initial data $f_0$ satisfies \eqref{initial} and   $\|f_{0}\|_{L^{2}_v(H^6_x\cap L^1_x)}\le \delta_0$, then we have
\be
\|f_{\eps}(t)-u(t)\|_{L^{\infty}_x(L^{2}_v)}+\|\Tdx\phi_{\eps}(t)-\Tdx\phi(t)\|_{L^{\infty}_x}\le C\delta_0\eps (1+t)^{-\frac34} . \label{limit_1a}
\ee
\end{thm}

\begin{remark}
By taking the Fourier transform to \eqref{NSP_5i} and \eqref{compatible}, we have
\be \hat{n}(0)=-\frac{\sqrt6|\xi|^2}{3+5|\xi|^2}\bigg(\hat{q}_0-\sqrt{\frac23}\hat{n}_0\bigg),\quad \hat{q}(0)=\frac{3+3|\xi|^2}{3+5|\xi|^2}\bigg(\hat{q}_0-\sqrt{\frac23}\hat{n}_0\bigg).\ee
This implies that the initial data $f_0$ of the VPB system has the same regularity in $x$ as the initial data $(n,m,q)(0)$ of the NSPF system.
\end{remark}

\begin{remark}
Define the oscalation part of $f_{\eps}$ as
\be u^{osc}_{\eps}(t,x,v)=\sum_{j=-1,1} \mathcal{F }^{-1}\(e^{ \frac{\eta_j(|\xi|)}{\eps}t-b_j(|\xi|)t}\(P_0\hat{f}_0,h_j(\xi)\)_{\xi}h_j(\xi)\) ,\label{osc}\ee
where $(\eta_j(|\xi|), h_j(\xi))$ and $b_j(|\xi|)$, $j=-1,1$ are defined by \eqref{eigen} and \eqref{eigen4}respectively. Then, under the first assumption of Theorem \ref{thm1.2} we have
\be
\|f_{\eps}(t)-u(t)-u^{osc}_{\eps}(t)-e^{\frac{t}{\eps^2}B_{\eps}}P_1f_0\|_{L^{\infty}_P} \le C\delta_0\eps (1+t)^{-\frac12},
\ee
where the norm $\|\cdot\|_{L^{\infty}_P}$ is defined by
\be \|f\|_{L^{\infty}_P}=\|f\|_{L^{\infty}_x(L^{2}_v)}+\|\Tdx\Delta^{-1}_x(f,\chi_0)\|_{L^{\infty}_x}. \label{norm2}\ee
\end{remark}

The rest of this paper will be organized as follows.
In the next section, we present the results about the spectrum analysis on  the linearized VPB system. In Section~\ref{sect3}, we  establish  the first and second order fluid approximations of the solution to the  linearized VPB system.
In Section~\ref{sect4}, we prove the convergence and establish the convergence rate of the global solution to the  nonlinear VPB system towards the solution to the nonlinear NSP system.

\section{Spectrum analysis}\label{sect2}
\setcounter{equation}{0}

In this section, we are concerned with the spectral analysis of the linear VPB operator $B_{\eps}(\xi)$ defined by
\eqref{B3}, which will be applied to study diffusion limit of the solution to the VPB system \eqref{VPB1}--\eqref{VPB2i}.

From the system \eqref{VPB1}--\eqref{VPB2i}, we have the following linearized VPB system:
\bq \label{VPB}
\left\{\bln
&\eps^2\dt f_{\eps}=B_{\eps}f_{\eps},\quad t>0,\\
 &f_{\eps}(0,x,v)=f_0(x,v),
 \eln\right.
 \eq
where
\bq B_{\eps}=L-\eps v\cdot\Tdx +\eps v\cdot\Tdx\Delta_x^{-1}P_d,\eq
with
\be P_dg=(g,\sqrt{M})\sqrt{M},\quad \forall g\in L^2(\R^3_v).\ee

Take Fourier transform to \eqref{VPB} to get
\be
\left\{\bln
 &\eps^2\dt \hat{f}_{\eps}= B_{\eps}(\xi)\hat f_{\eps},\quad t>0,\\
  &\hat{f}_{\eps}(0,\xi,v)=\hat{f}_0(\xi,v),
   \eln\right.
\ee
where
\bq B_{\eps}(\xi)=L-i\eps(v\cdot\xi)-i\eps\frac{v\cdot\xi }{|\xi|^2}P_d,\quad \xi\ne 0. \label{B3}\eq

For $\xi\ne 0$, define a weight Hilbert space $L^2_{\xi}(\R^3)$
 by
$$L^2_{\xi}(\R^3)=\{f\in
L^2(\R^3)\,|\,\|f\|_{\xi}=\sqrt{(f,f)_{\xi}}<\infty\}$$
 with the inner product
\bq (f,g)_{\xi}=(f,g)+\frac1{|\xi|^2}(P_d f,P_dg).\label{C_1}\eq

Since $P_{d}$ is a self-adjoint projection operator, it follows that
 $(P_{d} f,P_{d} g)=(P_{d} f, g)=( f,P_{d} g)$ and hence
 \bq (f,g)_{\xi}=(f,g+\frac1{|\xi|^2}P_{d}g)=(f+\frac1{|\xi|^2}P_{d}f,g).\label{C_1a}\eq
By
\eqref{C_1a}, we have for any $f,g\in L^2_{\xi}(\R^3_v)\cap D(B_{\eps}(\xi))$,
 \be
 (B_{\eps}(\xi)f,g)_{\xi}=(B_{\eps}(\xi) f,g+\frac1{|\xi|^2}P_{d} g)=(f,B_{\eps}(-\xi)g)_{\xi}. \label{L_7}
\ee
Moreover,  $B_{\eps}(\xi)$ is a dissipate operator in $L^2_{\xi}(\R^3)$:
 \be
 {\rm Re}(B_{\eps}(\xi)f,f)_{\xi}=(Lf,f)\le 0,\quad \forall f\in L^2_{\xi}(\R^3). \label{L_8}
\ee
We can regard $B_{\eps}(\xi)$ as a linear operator from the space $L^2_{\xi}(\R^3)$ to itself because
$$
 \|f\|^2\le \|f\|^2_{\xi}\le(1+|\xi|^{-2})\|f\|^2,\quad \xi\ne 0.
$$

We denote $\rho(A)$ and $\sigma(A)$ to be the resolvent set and spectrum set of the operator $A$ respectively.  The discrete spectrum
of $A$, denoted by $\sigma_d(A)$, is the set of all isolated eigenvalues with finite multiplicity.
The essential spectrum of $A$, $\sigma_{ess}(A)$, is the set $\sigma(A)\setminus \sigma_d
(A)$. By \eqref{L_7},  \eqref{L_8}, \eqref{nu} and \eqref{L_3}, we have the following lemmas.

\begin{lem}\label{SG_1}
The operator $B_{\eps}(\xi)$ generates a strongly continuous contraction semigroup on
$L^2_{\xi}(\R^3)$, which satisfies \bq
\|e^{tB_{\eps}(\xi)}f\|_{\xi}\le\|f\|_{\xi}, \quad\mbox{for any}\ t>0,\,f\in
L^2_{\xi}(\R^3_v). \eq
\end{lem}

\begin{proof}
Since  the operator $B_{\eps}(\xi)$ is a densely defined closed operator on  $L^2_{\xi}(\R^3)$, and both $B_{\eps}(\xi)$ and $B_{\eps}(\xi)^*=B_{\eps}(-\xi)$ are dissipative on $L^2_{\xi}(\R^3)$, it follows from Corollary 4.4 on p.15 of \cite{Pazy}  that $B_{\eps}(\xi)$ generates a strongly continuous contraction
semigroup on $L^2_{\xi}(\R^3)$.
\end{proof}

\begin{lem}\label{Egn}
The following conditions hold for all $\xi\ne 0$ and $\eps\in (0,1)$.
 \begin{enumerate}
\item[(1)]
$\sigma_{ess}(B_{\eps}(\xi))\subset \{\lambda\in \mathbb{C}\,|\, {\rm Re}\le -\nu_0\}$ and $\sigma(B_{\eps}(\xi))\cap \{\lambda\in \mathbb{C}\,|\, -\nu_0<{\rm Re}\le 0\}\subset \sigma_{d}(B_{\eps}(\xi))$.
\item[(2)]
 If $\lambda$ is an eigenvalue of $B_{\eps}(\xi)$, then ${\rm Re}\lambda<0$ for any $\xi\ne 0$ and $\eps\in (0,1)$.
 \end{enumerate}
\end{lem}

\begin{proof}
Define
 \be
  A_{\eps}(\xi)=-\nu(v)- i\eps (v\cdot\xi).  \label{Cxi}
 \ee
By \eqref{nu}, $\lambda- A_{\eps}(\xi)$ is invertible for ${\rm
Re}\lambda>-\nu_0$ and hence $\sigma (A_{\eps}(\xi))\subset \{\lambda\in \mathbb{C}\,|\, {\rm Re}\le -\nu_0\}$. Since $ B_{\eps}(\xi)$ is a compact perturbation of $ A_{\eps}(\xi)$, it follows from Theorem 5.35 in p.244 of \cite{Kato} that $ \sigma_{ess}(B_{\eps}(\xi))=\sigma_{ess}(A_{\eps}(\xi))$   and $\sigma (B_{\eps}(\xi))$ in the domain ${\rm Re}\lambda>-\nu_0$ consists of discrete eigenvalues  with possible accumulation
points only on the line ${\rm Re}\lambda= -\nu_0$. This proves (1). By a similar argument as  Proposition 2.2.8 in \cite{Ukai3}, we can prove (2) and  the details are omitted for brevity.
\end{proof}

Now denote by $T$ a linear operator on $L^2(\R^3_v)$ or
$L^2_{\xi}(\R^3_v)$, and we define the corresponding norms of $T$ by
$$
 \|T\|=\sup_{\|f\|=1}\|Tf\|,\quad
 \|T\|_{\xi}=\sup_{\|f\|_{\xi}=1}\|Tf\|_{\xi}.
$$
Obviously, it holds that
 \be
(1+|\xi|^{-2})^{-1/2}\|T\|\le \|T\|_{\xi}\le (1+|\xi|^{-2})^{1/2}\|T\|.\label{eee}
 \ee

First, we consider the spectrum and resolvent sets of $B_{\eps}(\xi)$ for $\eps|\xi|>r_0$ with $r_0>0$ being
a constant.  Write
\bma
\lambda-B_{\eps}(\xi)&=\lambda-A_{\eps}(\xi)-K+i\eps\frac{ v\cdot\xi}{|\xi|^2}P_{d}\nnm\\
&=\(I-K(\lambda-A_{\eps}(\xi))^{-1}+i\eps\frac{v\cdot\xi}{|\xi|^2}P_{d}(\lambda-A_{\eps}(\xi))^{-1}\)(\lambda-A_{\eps}(\xi)).\label{B_d}
\ema
Then, we estimate the terms on the right hand side of \eqref{B_d} as follows.

\begin{lem}[\cite{Ukai3,Li1}]\label{LP_3} There is a constant $C>0$ such that the following holds.
 \begin{enumerate}
\item For any $\delta>0$, if ${\rm Re}\lambda\ge -\nu_0+\delta$, we have
\be
\|K(\lambda-A_{\eps}(\xi))^{-1}\|\leq C\delta^{-11/13}(1+\eps|\xi|)^{-2/13}.\label{T_7}
\ee

\item  For any $\delta>0$,  if ${\rm Re}\lambda\ge -\nu_0+\delta$ and  $|{\rm Im}\lambda|\geq (2\eps|\xi|)^{5/3}\delta^{-2/3}$, we have
\be
\|K(\lambda-A_{\eps}(\xi))^{-1}\|\leq C\delta^{-3/5}(1+|{\rm Im}\lambda|)^{-2/5}.\label{T_8}
\ee

\item  For any $\delta>0,\, r_0>0$,  if ${\rm Re}\lambda\ge -\nu_0+\delta$ and $|\xi|\ge r_0$, we have
\be
\|(v\cdot\xi)|\xi|^{-2}P_d(\lambda-A_{\eps}(\xi))^{-1}\|\leq
C(\delta^{-1}+1)(r_0^{-1}+1) (|\xi|+|\lambda|)^{-1}. \label{L_9}
\ee
 \end{enumerate}
\end{lem}

By \eqref{B_d} and Lemma \ref{LP_3}, we have the spectral gap
of the operator $B_{\eps}(\xi)$ for $\eps|\xi|>r_0$.

\begin{lem}[Spectral gap]\label{LP_1}Fixed $\eps\in (0,1)$. For any $r_0>0$, there
exists $\alpha =\alpha(r_0)>0$ such that for  $ \eps|\xi|\ge r_0$,
\bq \sigma(B_{\eps}(\xi))\subset\{\lambda\in\mathbb{C}\,|\, \mathrm{Re}\lambda \leq-\alpha\} .\label{gap}\eq
\end{lem}

\begin{proof}
Let $\lambda\in \sigma(B_{\eps}(\xi))\cap\{\lambda\in\mathbb{C}\,|\, \mathrm{Re}\lambda\geq -\nu_0+\delta\}$ with  $\delta>0$. We first show that  $\sup_{\eps|\xi|\ge r_0}|{\rm
Im}\lambda|<+\infty$. By \eqref{eee}, \eqref{T_7} and \eqref{L_9}, there exists $r_1=r_1(\delta)>0$ large enough such that for
$\mathrm{Re}\lambda\geq -\nu_0+\delta$ and $\eps|\xi|\geq r_1$,
\bq
\|K(\lambda-A_{\eps}(\xi))^{-1}\|_{\xi}\leq1/4,\quad
\|\eps(v\cdot\xi)|\xi|^{-2}P_d(\lambda-A_{\eps}(\xi))^{-1}\|_{\xi}\leq1/4.\label{bound}
\eq
This implies that
$I+K(\lambda-A_{\eps}(\xi))^{-1}+i\eps(v\cdot\xi)|\xi|^{-2}P_d(\lambda-A_{\eps}(\xi))^{-1}$
is invertible on $L^2_{\xi}(\R^3_v)$, which together with
\eqref{B_d} yields that $\lambda-B_{\eps}(\xi)$ is also invertible on $L^2_{\xi}(\R^3_v)$ and
satisfies
\bq
(\lambda-B_{\eps}(\xi))^{-1}=(\lambda-A_{\eps}(\xi))^{-1}\(I-K(\lambda-A_{\eps}(\xi))^{-1}+i\eps\frac{v\cdot\xi}{|\xi|^2}P_d(\lambda-A_{\eps}(\xi))^{-1}\)^{-1} . \label{E_6}
\eq
Therefore,
$\{\lambda\in\mathbb{C}\,|\,\mathrm{Re}\lambda\ge
-\nu_0+\delta\}\subset \rho(B_{\eps}(\xi))$
for $\eps|\xi|\ge r_1$.

As for $r_0\le\eps|\xi|\le r_1$, by \eqref{T_8} and \eqref{L_9},  there exists
$\beta=\beta(r_0,r_1,\delta)>0$ such that if $\mathrm{Re}\lambda\geq -\nu_0+\delta$, $|\mathrm{Im}\lambda|>\beta$ and $\eps|\xi|\in [r_0,r_1]$, then
\eqref{bound} still holds and thus $\lambda-B_{\eps}(\xi)$ is invertible.
This implies that $\{\lambda\in\mathbb{C}\,|\,\mathrm{Re}\lambda\ge
-\nu_0+\delta, |\mathrm{Im}\lambda|>\beta\}\subset \rho(B_{\eps}(\xi))$ for
$r_0\le\eps|\xi|\le r_1$. Hence, we conclude that for $\eps|\xi|\ge r_0$,
 \bq
 \sigma(B_{\eps}(\xi)) \cap\{\lambda\in\mathbb{C}\,|\,\mathrm{Re}\lambda\ge-\nu_0+\delta\}
\subset \{\lambda\in\mathbb{C}\,|\,\mathrm{Re}\lambda\ge -\nu_0+\delta,\,|\mathrm{Im}\lambda|\le \beta \} .                            \label{SpH}
 \eq

Next, we prove that $\sup_{ \eps |\xi|\ge r_0}{\rm Re}\lambda<0$. Based on the above argument and Lemma \ref{Egn}, it is sufficient to prove
that   ${\rm Re}\lambda<0$ for $\eps|\xi|\in[r_0,r_1]$ and $|{\rm Im}\lambda|\le \beta$. This can
be proved by contradiction. Suppose that  there exists
$\eps|\xi_n|\in[r_0,r_1]$, $f_n\in L^2(\R^3)$, $\|f_n\|=1$,
$\lambda_n\in \sigma(B_{\eps}(\xi_n))$ such that
$$Lf_n-i\eps(v\cdot\xi_n)f_n-i\eps\frac{v\cdot\xi_n}{|\xi_n|^2}P_d f_n=\lambda_nf_n,\quad
\text{Re}\lambda_n\rightarrow0.$$
Write
$$(\lambda_n+\nu-i\eps(v\cdot\xi_n))f_n=Kf_n-i\eps\frac{v\cdot\xi_n}{|\xi_n|^2}P_d f_n.$$
Since $K$ and $P_d$ are compact, there exists a subsequence $f_{n_j}$ of $f_n$
and $g_1\in L^2(\R^3)$, $|C_0|\le 1$ such that $Kf_{n_j}\rightarrow g_1$ and $P_df_{n_j}\rightarrow C_0\sqrt{M}$ as $j\rightarrow\infty.$ Due to the fact
that $\eps|\xi_n|\in[r_0,r_1]$, $|\text{Im}\lambda_n|\leq \beta $ and ${\rm Re}\lambda_n\to 0$,  there exists a subsequence of
 (still denoted by) $(\xi_{n_j},\lambda_{n_j})$, and $(\xi_0,\lambda_0)$ with $\eps|\xi_0|\in[r_0,r_1]$, ${\rm
Re}\lambda_0=0$
such that $(\xi_{n_j},\lambda_{n_j})\to (\xi_0,\lambda_0)$ and $i(v\cdot\xi_{n_j})|\xi_{n_j}|^{-2}P_d f_{n_j}\rightarrow
g_2=i(v\cdot\xi_0)|\xi_0|^{-2}C_0\sqrt{M}$ as $j\rightarrow\infty.$
 Noting that $|\lambda_{n}+\nu+ i\eps(v\cdot\xi_{n})|\ge
\delta$, we have
$$\lim_{j\rightarrow\infty}f_{n_j}=\lim_{j\rightarrow\infty}\frac{g_1-\eps g_2}{\lambda_{n_j}+\nu-i\eps(v\cdot\xi_{n_j})}
=\frac{g_1-\eps g_2}{\lambda_0+\nu-i\eps(v\cdot\xi_0)}:=f,$$
and hence $Kf=g_1$ and   $i(v\cdot\xi_0)|\xi_0|^{-2}P_{d}
f=g_2.$
It follows that
$B_{\eps}(\xi_0) f=\lambda_0 f$ and thus $\lambda_0$ is an eigenvalue of
$B_{\eps}(\xi_0)$ with ${\rm Re}\lambda_0=0$, which contradicts ${\rm
Re}\lambda<0$ for $\xi\ne 0$ established by Lemma~\ref{Egn}. This proves the lemma.
\end{proof}

Then, we study the spectrum and resolvent sets of
$B_{\eps}(\xi)$ for $\eps|\xi|\le r_0$.  To this end, we decompose
\be
\lambda-B_{\eps}(\xi)=\lambda P_0- D_{\eps}(\xi)+\lambda P_1-Q_{\eps}(\xi)+i\eps P_0(v\cdot\xi)P_1+i\eps P_1(v\cdot\xi)P_0, \label{Bd3}
\ee
where
\be
D_{\eps}(\xi)= -i\eps P_0(v\cdot\xi)P_0-i\eps \frac{v\cdot\xi}{|\xi|^2}P_d,\quad
Q_{\eps}(\xi)=L-i\eps P_1(v\cdot\xi)P_1.\label{Qxi}
\ee
Here $D_{\eps}(\xi)=\eps D(\xi)$ is a
linear operator from $N_0$ to $N_0$,  where $D(\xi)=-i P_0(v\cdot\xi)P_0-i \frac{v\cdot\xi}{|\xi|^2}P_d$ is a linear operator represented by a matrix in the basis of $N_0$ as
\bq -D(\xi)=\left( \ba 0, & i\xi^T, & 0\\
i\xi(1+\frac1{|\xi|^2}), & 0 & i\sqrt{\frac23}\xi\\ 0, &
i\sqrt{\frac23}\xi^T, & 0 \ea\right).
\eq
It is easy to verify that $ \eta_{j}(|\xi|)$ and $h_j(\xi)$, $j=-1,0,1,2,3$ are the eigenvalues and eigenfunctions of $D(\xi)$ defined by
\be \label{eigen}
\left\{\bln
&\eta_{\pm1}(|\xi|)=\pm i\sqrt{1+\frac53|\xi|^2},\quad
 \eta_{k}(|\xi|)=0,\,\,\, k=0,2,3,\\
&h_0(\xi)=\frac{\sqrt{2}|\xi|^2}{\sqrt{3+5|\xi|^2}\sqrt{1+|\xi|^2}} \chi_0-\frac{\sqrt{3+3|\xi|^2}}{\sqrt{3+5|\xi|^2}}\chi_4,\\
&h_{\pm1}(\xi)=\frac{\sqrt{3/2}|\xi|}{\sqrt{3+5|\xi|^2}} \chi_0\pm \sqrt{\frac12}v\cdot \frac{\xi}{|\xi|}  \chi_0+\frac{ |\xi|}{\sqrt{3+5|\xi|^2}}\chi_4,\\
&h_k(\xi)=v\cdot W^k \chi_0,\quad k=2,3,\\
&\(h_j(\xi), h_k(\xi)\)_{\xi}=\delta_{jk},\quad j,k=-1,0,1,2,3,
\eln\right.
\ee
where $W^k$, $k=2,3$, are orthonormal vectors satisfying $W^k\cdot\xi=0$.

\begin{lem}[\cite{Li1}]\label{LP}
Let $\xi\neq0$, the following holds for $D_{\eps}(\xi)$ and $Q_{\eps}(\xi)$ defined by \eqref{Qxi}.
 \begin{enumerate}
\item  If $\lambda\neq \eps\eta_j(|\xi|)$, then  the operator $\lambda P_0-D_{\eps}(\xi)$ is
invertible on $N_0$ and satisfies
\bgr
  \|(\lambda P_0-D_{\eps}(\xi))^{-1}\|_{\xi}
  =\max_{-1\leq j \leq 3}\(|\lambda-\eps\eta_j(|\xi|)|^{-1}\),\label{S_2a}
\\
  \|P_1(v\cdot\xi)P_0(\lambda P_0-D_{\eps}(\xi))^{-1}\|_{\xi}
 \le C|\xi|\max_{-1\leq j \leq 3}\(|\lambda-\eps\eta_j(|\xi|)|^{-1}\),\label{S_2b}
 \egr
where $\eta_j(|\xi|)$, $j=-1,0,1,2,3$, are the eigenvalues of $D(\xi)$
defined by \eqref{eigen}.

\item  If $\mathrm{Re}\lambda>-\mu $, then the operator $\lambda P_1-Q_{\eps}(\xi)$ is
invertible on $N_0^\bot$ and satisfies
 \bgr
 \|(\lambda P_1-Q_{\eps}(\xi))^{-1}\|\leq(\mathrm{Re}\lambda+\mu )^{-1},  \label{S_3c}
\\
 \|P_0(v\cdot\xi)P_1(\lambda P_1-Q_{\eps}(\xi))^{-1}\|_{\xi}
 \leq
 C(1+|\lambda|)^{-1}[(\mathrm{Re}\lambda+\mu )^{-1}+1](|\xi|+\eps|\xi|^2). \label{S_5d}
 \egr
\end{enumerate}
\end{lem}

By \eqref{Bd3} and Lemmas~\ref{Egn}--\ref{LP}, we are able to analyze  the spectral
and resolvent sets of the operator $B_{\eps}(\xi)$ as follows.

\begin{lem}\label{spectrum}For fixed $\eps\in (0,1)$, the following facts hold.
 \begin{enumerate}
\item  For all $\xi\ne 0$, there exists $y_0>0$ such that %the resolvent set of $B_{\eps}(\xi)$ contains the following domain
\bq
 \{\lambda\in\mathbb{C}\,|\,
     \mathrm{Re}\lambda\ge-\frac{\mu}{2},\,|\mathrm{Im}\lambda|\geq y_0\}
 \cup\{\lambda\in\mathbb{C}\,|\,\mathrm{Re}\lambda>0\}
 \subset\rho(B_{\eps}(\xi)).                           \label{rb1}
\eq

\item For any $\delta>0$, there exists $r_0=r_0(\delta)>0$ such that for $\eps|\xi|\leq r_0$, %the spectrum set of $B_{\eps}(\xi)$ is located in the following domain
 \bq
 \sigma(B_{\eps}(\xi))\cap\{\lambda\in\mathbb{C}\,|\,\mathrm{Re}\lambda\ge-\frac{\mu}2\}
 \subset
 \{\lambda\in\mathbb{C}\,|\,|\lambda|\le\delta\}.   \label{sg4}
 \eq
 \end{enumerate}
\end{lem}

\begin{proof}
By Lemmas \ref{LP}, we have for $\rm{Re}\lambda>-\mu $ and
$\lambda\neq \eps\eta_j(|\xi|)$  that the operator
$\lambda  P_0-D_{\eps}(\xi)+\lambda  P_1-Q_{\eps}(\xi)$ is invertible on
$L^2_{\xi}(\R^3_v)$ and satisfies
 \be
 (\lambda P_0-D_{\eps}(\xi)+\lambda P_1-Q_{\eps}(\xi))^{-1}
=(\lambda P_0-D_{\eps}(\xi))^{-1}+(\lambda P_1-Q_{\eps}(\xi))^{-1},
 \ee
because the operator $\lambda P_0-D_{\eps}(\xi)$ is orthogonal to $\lambda
 P_1-Q_{\eps}(\xi)$. Therefore, we can re-write \eqref{Bd3} as
\bmas
 \lambda-B_{\eps}(\xi)
=&(I+Y_{\eps}(\lambda,\xi))(\lambda  P_0-D_{\eps}(\xi)+\lambda  P_1-Q_{\eps}(\xi)),
 \\
Y_{\eps}(\lambda,\xi)= &i\eps P_1(v\cdot\xi)P_0(\lambda P_0-D_{\eps}(\xi))^{-1}
    +i\eps P_0(v\cdot\xi) P_1(\lambda P_1-Q_{\eps}(\xi))^{-1}.
 \emas

By Lemma \ref{LP_3}, there exists $R_0>0$ large enough so that for $\mathrm{Re}\lambda\geq
-\nu_0/2$ and $\eps |\xi|\geq R_0$,  $\lambda-B_{\eps}(\xi)$ is  invertible on $L^2_{\xi}(\R^3_v)$ and
satisfies \eqref{E_6}. Thus, $\{\lambda\in\mathbb{C}\,|\,\mathrm{Re}\lambda\ge
-\nu_0/2\}\subset \rho(B_{\eps}(\xi))$ for $\eps |\xi|\ge R_0$.

For  $\eps|\xi|\leq R_0$, by \eqref{S_2b} and
\eqref{S_5d} we can choose $y_0>0$ sufficiently large such that for $\mathrm{Re}\lambda\ge-\mu/2$ and
$|\mathrm{Im}\lambda|\geq y_0$,
 \be
 \|\eps P_1(v\cdot\xi)P_0(\lambda P_0-A_{\eps}(\xi))^{-1}\|_{\xi}\leq \frac14,
 \quad
\|\eps P_{d}(v\cdot\xi)P_1(\lambda P_1-Q_{\eps}(\xi))^{-1}\|_{\xi}\leq\frac14.\label{bound_1}
 \ee
This implies that the operator $I+Y_{\eps}(\lambda,\xi)$ is invertible on
$L^{2}_{\xi}(\R^3_v)$ and thus  $\lambda-B_{\eps}(\xi)$ is invertible on $L^{2}_{\xi}(\R^3_v)$ and satisfies
 \be
 (\lambda-B_{\eps}(\xi))^{-1}
 =\((\lambda P_0-A_{\eps}(\xi))^{-1} P_0+(\lambda P_1-Q_{\eps}(\xi))^{-1} P_1\)(I+Y_{\eps}(\lambda,\xi))^{-1}.\label{S_8}
 \ee
Therefore, $\rho(B_{\eps}(\xi))\supset \{\lambda\in\mathbb{C}\,|\,{\rm
Re}\lambda\ge-\mu/2, |{\rm Im}\lambda|\ge y_0\}$ for $\eps |\xi|\leq R_0$. This and Lemma~\ref{Egn} lead to \eqref{rb1}.

Assume that $ |\lambda|>\delta$ and $\mathrm{Re}\lambda\ge-\mu/2$. Then, by \eqref{S_2b}  and
\eqref{S_5d} we can choose $r_0=r_0(\delta)>0$ so that estimates in \eqref{bound_1}
still hold for $\eps|\xi|\leq r_0$, and then
$\lambda-B_{\eps}(\xi)$ is invertible on $L^{2}_{\xi}(\R^3)$.
Therefore, we have
 $\rho(B_{\eps}(\xi))\supset\{\lambda\in\mathbb{C}\,|\, |\lambda|>\delta,\mathrm{Re}\lambda\ge-\mu/2\}$
for $\eps|\xi|\leq r_0$, which gives \eqref{sg4}.
\end{proof}

%\section{Eigenvalues of $B_{\eps}(\xi)$ near zero}\setcounter{equation}{0}
Now we study the asymptotic expansions of the eigenvalues and  eigenfunctions of $B_{\eps}(\xi)$ for $\eps|\xi|$ sufficiently small. Firstly, we consider a 1-D eigenvalue problem:
\bq B_{\eps}(s)e=:\(L- i\eps v_1s-i\eps\frac{ v_1}{s}P_{d}\)e=\beta e,\quad s\in \R.\label{L_2}\eq

Let $e$ be the eigenfunction of \eqref{L_2}, we rewrite $e$ in the
form $e=g_0+g_1$, where $g_0=P_0e$ and $g_1=(I-P_0)e=P_1e$.  The
eigenvalue problem \eqref{L_2} can be decomposed into
\bma \beta
g_0&=-i\eps sP_0[v_1(g_0+g_1)]-i\eps \frac{v_1}{s}P_dg_0,\label{A_2}\\
\beta g_1&=Lg_1-i\eps sP_1[v_1(g_0+g_1)].\label{A_3}
\ema
From Lemma \ref{LP} and \eqref{A_3}, we obtain that for any $\text{Re}\beta>-\mu $,
\bq
g_1=i\eps s(L-\beta P_1-i \eps sP_1v_1P_1)^{-1}(P_1v_1g_0).\label{A_4}
\eq
Substituting \eqref{A_4} into \eqref{A_2}, we have
\bq
\beta g_0=-i\eps sP_0v_1g_0-i\eps \frac{v_1}{s}P_dg_0+\eps^2s^2 P_0[v_1R(\beta,\eps s)P_1v_1g_0],\label{A_5}\eq
where
\bq R(\beta,s)=(L-\beta P_1-i sP_1v_1P_1)^{-1}.\eq

We will now reduce \eqref{A_5} to a problem of 5-dimension linear system. Since $g_0\in N_0$, we  have
$$ g_0=\sum_{j=0}^4W_j\chi_j\quad
\text{with}\quad W_j=\intr g_0\chi_jdv.$$
Taking the inner product of \eqref{A_5} with $\{\chi_j,\ j=0,1,2,3,4\}$ gives
\bma
\beta W_0&=-i\eps s W_1,\label{A_6}\\
\beta W_1&=-i\eps W_0\(s+\frac{1}{s}\)-i\eps s\sqrt{\frac23}W_4 + \eps^2s^2 W_1R_{11}(\beta,\eps s)\nnm\\
&\quad+\eps^2s^2 W_4R_{41}(\beta,\eps s),\label{A_7}\\
\beta W_j&=  \eps^2s^2 W_jR_{22}(\beta,\eps s), \quad j=2,3,\label{A_7a}\\
\beta
W_4&=-i\eps s \sqrt{\frac23}W_1+\eps^2s^2 W_1R_{14}(\beta,\eps s) +\eps^2s^2 W_4R_{44}(\beta,\eps s),\label{A_8}
\ema
where
\be R_{jk}(\beta, s)=(R(\beta, s)P_1(v_1\chi_j),v_1 \chi_k),\quad j,k=1,2,4. \label{rij}\ee

 Denote
\bma
&D_0(z,s)=z- s^2R_{22}(z, s), \label{ddd}
\\
&D_1(z,s,\eps)= \left|\ba z & i s & 0\\
i \(s+\frac1s\) & z-\eps s^2R_{11}(\eps z,\eps s) & i s\sqrt{\frac23}-\eps s^2R_{41}(\eps z,\eps s)
\\
0 & i  s\sqrt{\frac23}-\eps s^2R_{14}(\eps z,\eps s) &z-\eps s^2R_{44}(\eps z,\eps s) \ea\right|. \label{T_3}
\ema

Then, by a direct computation and the  implicit function theorem, we can show

\begin{lem}\label{eigen_1}There are two small constants $r_0,r_1>0$ such that the equation $D_0(z,s)=0$ has a unique $C^\infty$ solution
$z=z(s)$ for $|s|\leq r_0$ and $|z|\le r_1$, which is a $C^{\infty}$ function of $s$ satisfying
\be z(0)=0,\quad  z'(0)=0,\quad z''(0)=2(L^{-1}(v_1\chi_2),v_1\chi_2). \label{z1}\ee
Moreover, $z(s)$ is an even, real function.
\end{lem}

We have the following result about  the solution of $D_1(z,s,\eps)=0$.

\begin{lem}\label{eigen_2}There are two small constants $r_0,r_1>0$ such that the equation
$D_1(z,s,\eps)=0$ has exactly three solutions
$z_j=z_j(s,\eps)$, $j=-1,0,1$ for $\eps| s|\le r_0$ and $|z_j-\eta_j(s)|\le r_1|s|$. They are $C^\infty$ functions of $s$ and $\eps$, which satisfy
 \be
z_j(s,0)=\eta_j(s),\quad \pt_{\eps}z_j(s,0)=-b_j(s), \label{z2}
\ee
where $\eta_j(s)=ji\sqrt{1+\frac53s^2}$, $j=-1,0,1$,  and
\be \label{z4a}
\left\{\bln
b_0(s)&=-\frac{3(s^2+s^4)}{3+5s^2}(L^{-1}P_1(v_1\chi_4),v_1\chi_4),\\
b_{\pm1}(s)&=-\frac12s^2(L^{-1}P_1(v_1\chi_1),v_1\chi_1)-\frac{s^4}{3+5s^2}(L^{-1}P_1(v_1\chi_4),v_1\chi_4).
\eln\right.
\ee
In particular, $z_j(s,\eps)$, $j=-1,0,1$ satisfy the following expansions
\be
z_j(s,\eps)=\eta_j(s)-\eps b_j(s)+O(1)(\eps^2 s^3).\label{z4}
\ee
\end{lem}
\begin{proof}  By \eqref{T_3}, we have
\bma
D_1(z,s,\eps)=&z^3-z^2\eps s^2(R_{11}+R_{44})+z\Big[1+\frac53s^2+i\eps\sqrt{\frac23} s^3\(R_{41}+R_{14}\)\nnm\\
&+\eps^2 s^4(R_{44}R_{11}-R_{14}R_{41})\Big]-\eps (s^2+s^4)R_{44},\label{A_13}
\ema
where $R_{jk}=R_{jk}(\eps z,\eps s)$, $j,k=1,2,4$ are defined by \eqref{rij}. It follows  that
\bq
D_1(z,s,0)=z\(z^2+1+\frac53s^2\)=0 \label{aaa}\eq
has three solutions $\eta_j(s)=ji\sqrt{1+\frac53s^2}$ for  $j=-1,0,1$.
Moreover, $D_1(z,s,\eps)$ is $C^\infty$ with respect to $(z,s,\eps)$ and satisfies
\bma
\partial_{\eps}D_1(z,s,\eps)=&-z^2s^2(R_{11}+R_{44})
-z^2\eps s^2 (\partial_{\eps}R_{11}+\partial_{\eps}R_{44})\nnm\\
&+z\Big[i\sqrt{\frac23}s^3 (R_{14}+R_{41}) +i\eps\sqrt{\frac23}s^3\partial_{\eps}  (R_{14}+R_{41}) +2\eps s^4R_{11}R_{44}\nnm\\
&+\eps^2 s^4\partial_{\eps}(R_{11}R_{44})-2\eps s^4R_{14}R_{41}
+\eps^2 s^4\partial_{\eps} (R_{14}R_{41})\Big]\nnm\\
&-(s^2+s^4)R_{44}-\eps(s^2+s^4)\partial_{\eps} R_{44},\label{p_0}
\\
\partial_{z}D_1(z,s,\eps)=&3z^2+1+\frac53 s^2-2z\eps s^2(R_{11}+R_{44})+z^2\eps s^2 (\partial_zR_{11}+\partial_zR_{44})\nnm\\
&+\Big[\eps^2 s^4(R_{11}R_{44}-R_{14}R_{41})-i\eps s^3\sqrt{\frac23}(R_{14}+R_{41})\Big]\nnm\\
&+z\Big[\eps^2 s^4\partial_z (R_{11}R_{44}-R_{14}R_{41})-i\eps s^3\sqrt{\frac23} (\partial_z R_{14}+\partial_z R_{41})\Big]\nnm\\
&-\eps(s^2+s^4)\partial_z R_{44}.\label{p_1}
\ema

For $j=-1,0,1$, we define
$$G_j(z,s,\eps)=z-\(3\eta_j(s)^2+1+\frac53s^2\)^{-1}D_1(z,s,\eps).$$
It is straightforward to verify that a solution of $D_1(z,s,\eps)=0$ for any fixed $s$ and $\eps$ is a fixed point of $G_j(z,s,\eps)$.

Since %it holds for any $(z,s,\eps)\in \mathbb{C}\times\R\times \R$ that
$$ |\partial_z R_{ij}(\eps z,\eps s)|\le C\eps ,\quad |\partial_\eps R_{ij}(\eps z,\eps s)|\le C(|z|+|s|),\quad i,j=1,2,4,$$
it follows from \eqref{p_0} and \eqref{p_1} that
\bmas |\partial_z G_j(z,s,\eps)|&=\bigg|1-\(3\eta_j(s)^2+1+\frac53s^2\)^{-1}\partial_{z}D_1(z,s,\eps)\bigg|\le Cr_1,\\
|\partial_{\eps} G_j(z,s,\eps)|&=\bigg|\(3\eta_j(s)^2+1+\frac53s^2\)^{-1}\partial_{\eps}D_1(z,s,\eps)\bigg|\le Cs^2,
\emas
for $|z-\eta_j(s)|\le r_1|s|$ and $\eps|s|\le r_0$ with $r_0,r_1>0$ sufficiently small.
This implies that for $|z-\eta_j(s)|\le r_1|s|$ and $\eps|s|\le r_0$ with $r_0,r_1\ll1$,
\bmas |G_j(z,s,\eps)-\eta_j(s)|&=|G_j(z,s,\eps)-G_j(\eta_j(s),s,0)|\\
&\le |G_j(z,s,\eps)-G_j(z,s,0)|+|G_j(z,s,0)-G_j(\eta_j(s),s,0)|\\
&\le |\partial_{\eps} G_j(z,s,\widetilde{\eps})||\eps|+|\partial_z G_j(\widetilde{z},s,0)||z-\eta_j(s)|\le r_1|s|,\\
|G_j(z_1,s,\eps)-G_j(z_2,s,\eps)|&\le |\partial_z G_j(\bar{z},s,\eps)||z_1-z_2|\le \frac12|z_1-z_2|,
\emas
where $\widetilde{\eps}$ is between $0$ and $\eps$, $\widetilde{z}$ is  between $z$ and $\eta_j(s)$, and $\bar{z}$ is  between $z_1$ and $z_2$.

Hence by the contraction mapping theorem, there exist exactly three functions $z_j(s,\eps)$, $j=-1,0,1$ for $\eps|s|\le r_0$ and $|z_j-\eta_j(s)|\le r_1|s|$ such that $G_j(z_j(s,\eps),s, \eps)=z_j(s,\eps)$ and $z_j(s,0)=\eta_j(s)$. This is equivalent to that $D_1(z_j(s,\eps),s,\eps)=0$.
Moreover, by \eqref{p_0}--\eqref{p_1} we have
\bma
 \partial_{\eps} z_{0}(s,0)&=-\frac{\partial_{\eps} D_1(0,s,0)}{\partial_z D_1(0,s,0)}=\frac{3(s^2+s^4)}{3+5s^2}(L^{-1}P_1(v_1\chi_4),v_1\chi_4), \label{bbb}\\
 \partial_{\eps} z_{\pm1}(s,0)&=-\frac{\partial_{\eps} D_1(\eta_{\pm1}(s),s,0)}{\partial_z D_1(\eta_{\pm1}(s),s,0)}\nnm\\
 &=\frac12s^2(L^{-1}P_1(v_1\chi_1),v_1\chi_1)+\frac{s^4}{3+5s^2}(L^{-1}P_1(v_1\chi_4),v_1\chi_4).\label{z3}
\ema
Combining \eqref{aaa}, \eqref{bbb} and \eqref{z3}, we obtain \eqref{z2} and \eqref{z4a}.

Finally, we deal with \eqref{z4}. By \eqref{p_0} and \eqref{p_1},  we obtain that for $|z-\eta_j(s)|\le r_1|s|$ and $\eps|s|\le r_0$,
$$
 -\frac{\pt_\eps D_1(z,s,\eps)}{\pt_z D_1(z,s,\eps)}=b_j(s)+o(1) |s|^2 ,
$$
which implies that for $\eps|s|\le r_0$,
\be
|z_j(s,\eps)-\eta_j(s)|= |\partial_\eps z_j(s,\tilde{\eps})||\eps|= \frac{|\pt_\eps D_1(z_j(s,\tilde{\eps}),s,\tilde{\eps})|}{|\pt_z D_1(z_j(s,\tilde{\eps}),s,\tilde{\eps})|}|\eps|\le C\eps|s|^2, \label{z5}
\ee
where $\widetilde{\eps}$ is between $0$ and $\eps$. Thus, it follows from \eqref{p_0}, \eqref{p_1} and \eqref{z5} that
$$
 \partial_{\eps} z_{j}(s,\eps)=-\frac{\pt_\eps D_1(z_j(s,\eps),s,\eps)}{\pt_z D_1(z_j(s,\eps),s,\eps)}=b_j(s)+O(1) \eps|s|^3, \quad \eps|s|\le r_0.
$$
 The above estimate and  \eqref{z2} lead to \eqref{z4}. The proof of the lemma is then completed.
\end{proof}

With the help of Lemmas \ref{eigen_1}--\ref{eigen_2}, we are able to
construct the eigenvalue $\beta_j(s,\eps)$ and the corresponding  eigenfunction $e_j(s,\eps)$  of $B_{\eps}(s)$ for $\eps |s|$ sufficiently small.

\begin{thm}\label{eigen_3}
There exists a constant $r_0>0$ such that the spectrum $\sigma(B_{\eps}(s))\cap \{\lambda\in\mathbb{C}\,|\,\mathrm{Re}\lambda\ge-\mu /2\}$  consists of five points $\{\beta_j(s,\eps),\ j=-1,0,1,2,3\}$ for $\eps |s|\leq r_0$. The eigenvalues $\beta_j(s,\eps)$ and the corresponding eigenfunction $e_j(s,\eps)$ are $C^\infty$ functions of $s$ and $\eps$. In particular, the eigenvalues $\beta_j(s,\eps)$, $j=-1,0,1,2,3$ admit the following asymptotic expansion
\be
\beta_{j}(s,\eps)=\eps\eta_j(s)-\eps^2 b_{j} (s)+O(\eps^3s^3),\quad \eps|s|\leq r_0, \label{eigen3}
\ee
where $\eta_j(s)$, $j=-1,0,1,2,3$ are defined by \eqref{eigen}, and
\be \label{eigen4}
\left\{\bln
b_0(s)&=-\frac{3(s^2+s^4)}{3+5s^2}(L^{-1}P_1(v_1\chi_4),v_1\chi_4),\\
b_{\pm1}(s)&=-\frac12s^2(L^{-1}P_1(v_1\chi_1),v_1\chi_1)-\frac{s^4}{3+5s^2}(L^{-1}P_1(v_1\chi_4),v_1\chi_4),\\
b_k(s)&=-s^2(L^{-1}P_1(v_1\chi_2),v_1\chi_2),\quad k=2,3.
\eln\right.
\ee

The corresponding eigenfunctions $e_j(s,\eps)=e_j(s,\eps,v)$, $j=-1,0,1,2,3$ satisfy
\be \label{eigenf1}
\left\{\bln
&(e_j ,\overline{e_k} )_s=(e_j,\overline{e_k})+s^{-2}(P_de_j,P_d\overline{e_k})=\delta_{jk},\,\,\, -1\le j,k\le 3,\\
& e_j(s,\eps)=P_0e_j(s,\eps)+P_1e_j(s,\eps), \\
&P_0 e_{j}(s,\eps)=g_{j}(s)+O(\eps s), \\
& P_1 e_{j}(s,\eps)= i\eps s L^{-1}P_1(v_1g_{j}(s))+O(\eps^2 s^2),\\
\eln\right.
\ee
where $g_j(s)$, $j=-1,0,1,2,3$ are defined by
\be
\left\{\bln
g_{0}(s)&=-\frac{\sqrt{2}s^2}{\sqrt{3+5s^2}\sqrt{1+s^2}} \chi_0+\frac{\sqrt{3+3s^2}}{\sqrt{3+5s^2}}\chi_4,\\
g_{\pm1}(s)&= \frac{\sqrt{3/2}s}{\sqrt{3+5s^2}} \chi_0\pm\sqrt{\frac12}v_1  \chi_0+ \frac{ s}{\sqrt{3+5s^2}}\chi_4,\\
g_{k}(s)&=v_k\chi_0,\quad k=2,3.
\eln\right.
\ee
\end{thm}
\begin{proof}
The eigenvalues $\beta_j(s,\eps)$ and the corresponding eigenfunctions $e_j(s,\eps)$ of $B_{\eps}(s)$ can be constructed as follows. For $j=2,3$,
we take $\beta_j(s,\eps)=z(\eps s)$ with $z(y)$ being a solution of $D_0(z,y)=0$  defined in Lemma \ref{eigen_1},  and take $W_0=W_1=W_4=0$ given in \eqref{A_6}--\eqref{A_8}. And the corresponding eigenfunctions $e_j(s,\eps)$, $j=2,3$ are defined by
\bq
e_j(s,\eps)=b_2(s,\eps)\chi_j+ib_2(s,\eps)\eps s (L-\beta_j (s,\eps)-i\eps sP_1v_1P_1)^{-1}P_1(v_1\chi_j), \label{C_2}
\eq
which are orthonormal, i.e.,  $(e_2(s,\eps),\overline{e_3(s,\eps)})_s=0$.

For $j=-1,0,1$, we choose
$\beta_j(s,\eps)=\eps z_j(s,\eps)$ with $z_j(s,\eps)$ being a solution of $D_1(z,s,\eps)=0$ defined in Lemma \ref{eigen_2}, and take $W_2=W_3=0$. We denote by $\{a_j,b_j,c_j\}=:\{W^j_0,\, W_1^j,\, W^j_4\}$  as  a solution of system \eqref{A_6}, \eqref{A_7} and \eqref{A_8} for $\beta=\beta_j(s,\eps)$. Then we can construct $e_j(s,\eps)$, $j=-1,0,1$ as
 \bq \label{C_3}
 \left\{\bln e_j(s,\eps)&=P_0e_j(s,\eps)+P_1e_j(s,\eps),\\
P_0e_j(s,\eps)&=a_j(s,\eps) \chi_0+b_j(s,\eps)\chi_1+c_j(s,\eps)\chi_4,\\
P_1e_j(s,\eps)&= i\eps s(L-\beta_j(s,\eps)- i\eps sP_1v_1P_1)^{-1}P_1\(v_1P_0e_j(s,\eps)\).
\eln\right.
\eq
Write
$$\(L-i\eps s v_1-i\eps \frac{v_1}s P_d\)e_j(s, \eps)=\beta_j(s,\eps)e_j(s, \eps), \quad j=-1,0,1,2,3.$$
By taking inner product
$(\cdot,\cdot)_{s}$ of above with $\overline{e_k}(s, \eps)$ and using
the fact that
\bmas
&\(L+i\eps s v_1+i\eps\frac{v_1}s P_d\)\overline{e_j}(s, \eps)=\overline{\beta_j}(s,\eps)\overline{e_j}(s, \eps),\\
&(B_{\eps}(s)f,g)_s=(f,B_{\eps}(-s)g)_s,\quad \forall f,g\in D(B_{\eps}(s)),
\emas
we have
\bq
(\beta_j(s,\eps)-\beta_{k}(s,\eps))(e_j(s, \eps),\overline{e_k}(s, \eps))_s=0,\quad j,k=-1,0,1,2,3.\eq
For $|\eps s|$ sufficiently small, $\beta_j(s, \eps)\neq\beta_{k}(s, \eps)$ for
$-1\le j\neq k \le 2$. Therefore,  we have the orthogonality relation
\bq
(e_j(s, \eps),\overline{e_k}(s, \eps))_s=0,\quad -1\leq j\neq k\leq 3.
\eq
We also normalize these eigenfunctions  by $(e_j(s, \eps),\overline{e_j}(s, \eps))_s=1$ for $-1\leq j\leq 3.$

The coefficient $b_2(s,\eps) $ in \eqref{C_2} is determined by the normalization condition  as
 \bq
 \label{C_4}  b_2(s,\eps)^2\(1 +\eps^2s^2D_2(s,\eps)\)=1
 \eq
with $D_2(s,\eps)=(R(\beta_0,\eps s)P_1( v_1\chi_2), R(\overline{\beta_0},-\eps s)P_1( v_1\chi_2)).$ It follows from \eqref{C_4} and \eqref{eigen3} that
$$ b_2(s,\eps)=1+O(\eps^2s^2),\quad \eps|s|\le r_0,$$
which together with \eqref{C_2} leads to \eqref{eigenf1} for $j=2,3$.

To obtain the expansion of $e_j(s,\eps)$ for $j=-1,0,1$ defined in
\eqref{C_3}, we consider its macroscopic part and microscopic part respectively.
By \eqref{A_6}, \eqref{A_7} and \eqref{A_8}, the macroscopic part $P_0(e_j(s, \eps))$ is determined in terms of  the coefficients $\{a_j(s,\eps),b_j(s,\eps), c_j(s,\eps)\}$ that satisfy
\be \label{expan2}
\left\{\bln
z_j(s,\eps) a_j(s,\eps)+isb_j(s,\eps)&=0,
\\
i\(s+\frac1s\)a_j(s,\eps)+\(z_j(s,\eps)-\eps s^2R_{11} \)b_j(s,\eps)&
\\
\quad+\bigg(is\sqrt{\frac23}+\eps s^2R_{41} \bigg)c_j(s,\eps)&=0,
\\
\bigg(is\sqrt{\frac23}-\eps s^2R_{14} \bigg)b_j(s,\eps)
+\(z_j(s,\eps)-\eps s^2R_{44} \)c_j(s,\eps)&=0,
\eln\right.
\ee
where $R_{kl}=R_{kl}(\eps z_j(s,\eps),\eps s)$ and $z_j(s,\eps)$ are given by \eqref{rij} and \eqref{z2} respectively.
Furthermore, we have the normalization condition:
\be 1=a_{j}^2(s,\eps)\(1+\frac1{s^2}\)+b_{j}^2(s,\eps)+c_{j}^2(s,\eps)+\eps^2s^2D_3(s,\eps),\quad  \eps|s|\le r_0, \label{expan3}\ee
 where $D_3(s,\eps)=(R(\beta_j,\eps s)P_1( v_1P_0e_j), R(\overline{\beta_j},-\eps s)P_1( v_1P_0e_j)).$

By \eqref{expan2} and \eqref{z4}, we can expand  $a_j(s,\eps),b_j(s,\eps), c_j(s,\eps)$ as
$$ a_{j}(s,\eps)= a_{j,0}(s)+ O(\eps s),\quad
b_{j}(s,\eps)=
b_{j,0}(s)+ O(\eps s),\quad c_{j}(s,\eps)=
c_{j,0}(s)+ O(\eps s).$$
Substituting above expansions and \eqref{z4} into \eqref{expan2} and \eqref{expan3}, we have the following expansion
\bq
\left\{\bal j\sqrt{1+\frac53s^2}a_{j,0}(s)+ s b_{j,0}(s)=0,\\
(s^2+1)a_{j,0}(s)+js\sqrt{1+\frac53s^2}b_{j,0}(s)+ \sqrt{\frac23}s^2 c_{j,0}(s)=0,\\
 s\sqrt{\frac23}b_{j,0}(s)+j\sqrt{1+\frac53s^2}c_{j,0}(s)=0,\\
 a_{j,0}^2(s)\(1+\frac1{s^2}\)+b_{j,0}^2(s)+c_{j,0}^2(s)=1.
 \ea\right.\label{C_5}
 \eq
By combining \eqref{C_5} and \eqref{C_3}, we can prove \eqref{eigenf1} for $j=-1,0,1$ by a straightforward computation. The proof is then completed.
\end{proof}

We now  consider the following 3-D eigenvalue problem:
\bq B_{\eps}(\xi)\psi=\(L- i \eps(v\cdot\xi)-i\eps\frac{ v\cdot\xi}{|\xi|^2}P_{d}\)\psi=\lambda \psi,\quad \xi\in \R^3.\label{L_3a}\eq

With the help of Lemma \ref{eigen_3}, we have the expansion of the eigenvalues $\lambda_j(|\xi|,\eps)$ and the corresponding eigenfunctions $\psi_j(\xi,\eps)$ of $B_{\eps}(\xi)$ for $\eps|\xi|\leq r_0$ as follows.

\begin{thm}\label{spect3}
 There exists a constant $r_0>0$ such that the spectrum $\sigma(B_{\eps}(\xi))\cap \{\lambda\in\mathbb{C}\,|\,\mathrm{Re}\lambda\ge-\mu /2\}$  consists of five points $\{\lambda_j(|\xi|,\eps),\ j=-1,0,1,2,3\}$ for $\eps |\xi|\leq r_0$. The eigenvalues $\lambda_j(|\xi|,\eps)$, $j=-1,0,1,2,3$  are $C^{\infty}$ functions of $(|\xi|, \eps)$, and satisfy the following expansions for $\eps|\xi|\leq r_0$:
 \be                                   \label{specr1a}
 \lambda_j(|\xi|,\eps)=\eps\eta_j(|\xi|)-\eps^2 b_{j} (|\xi|)+O(\eps^3|\xi|^3),
 \ee
 where $\eta_j(|\xi|)$ and $b_{j} (|\xi|)$ are defined by \eqref{eigen} and \eqref{eigen4} respectively.

The eigenfunctions $\psi_j(\xi,\eps)=\psi_j(\xi,\eps,v)$, $j=-1,0,1,2,3$ satisfy
\be \label{eigen1}
\left\{\bln
&(\psi_j ,\overline{\psi_k} )_{\xi}=(\psi_j,\overline{\psi_k})+|\xi|^{-2}(P_d\psi_j,P_d\overline{\psi_k})=\delta_{jk},\,\,\, -1\le j,k\le 3, \\
&\psi_j(\xi,\eps)=P_0\psi_j(\xi,\eps)+P_1\psi_j(\xi,\eps),\\
 &P_0 \psi_j(\xi,\eps)=h_{j}(\xi)+O(\eps |\xi|), \\
 &P_1 \psi_{j}(\xi,\eps)= i\eps  L^{-1}P_1[(v\cdot\xi)P_0h_{j}(\xi)]+O(\eps^2 |\xi|^2),
\eln\right.
\ee
where $h_j(\xi)\in N_0$ are defined by \eqref{eigen}.
\end{thm}
\begin{proof}
Let $\O$ be a rotational transformation  in $\R^3$ such that $\O: \frac{\xi}{|\xi|}\to(1,0,0)$. We have
\bq \O^{-1} \(L- i\eps (v\cdot\xi)-i\eps\frac{ v\cdot\xi}{|\xi|^2}P_{d}\)\O=L- i\eps v_1s-i\eps\frac{  v_1}{s}P_{d}.\eq
Thus, from Lemma \ref{eigen_3}, we have the following eigenvalues and eigenfunctions for \eqref{L_3a}:
\bgrs
\(L- i\eps (v\cdot\xi)-i\eps\frac{ v\cdot\xi}{|\xi|^2}P_{d}\)\psi_j(\xi,\eps)=\lambda_j(|\xi|,\eps) \psi_j(\xi,\eps),\\
\lambda_j(|\xi|,\eps)=\beta_j(|\xi|,\eps),\quad \psi_j(\xi,\eps)=\O e_j(|\xi|,\eps),\quad %{\rm for}\quad
j=-1,0,1,2,3.
\egrs
This proves the theorem.
\end{proof}

By virtue of Lemmas~\ref{LP_3}--\ref{spectrum} and Theorem \ref{spect3}, we can  analyze on the semigroup $S(t,\xi,\eps)=e^{\frac{t}{\eps^2}B_{\eps}(\xi)}$ precisely
by using an argument  similar to that of Theorem~3.4 in \cite{Li2}. Hence,  we
only state the result as follows and  omit the detail of the proof for brevity.

\begin{thm}\label{E_3}
The semigroup $S(t,\xi,\eps)=e^{\frac{t}{\eps^2} B_{\eps}(\xi)}$ has the following
decomposition:
\bq S(t,\xi,\eps)f=S_1(t,\xi,\eps)f+S_2(t,\xi,\eps)f,\quad f\in L^2_{\xi}(\R^3_v),\eq
where
\be
S_1(t,\xi,\eps)f=\sum^3_{j=-1}e^{\frac{t}{\eps^2}\lambda_j(|\xi|,\eps)}\(f,\overline{\psi_j(\xi,\eps)}\)_{\xi} \psi_j(\xi,\eps)1_{\{\eps|\xi|\leq r_0\}},\label{E_5}
\ee
with $(\lambda_j(|\xi|,\eps),\psi_j(\xi,\eps))$ being the eigenvalue and eigenfunction of the operator $B_{\eps}(\xi)$ given in Theorem \ref{spect3} for $\eps|\xi|\le r_0$,
and $S_2(t,\xi,\eps) $ satisfies for two constants $d>0$ and $C>0$ independent of $\xi$ and $\eps$ that
\be
\|S_2(t,\xi,\eps)f\|_{\xi}\leq Ce^{-\frac{dt}{\eps^2}}\|f\|_{\xi}.\label{B_3}
\ee
\end{thm}

\section{Fluid approximation of semigroup}\label{sect3}
\setcounter{equation}{0}
In this section,  we give the first and second order fluid approximations of the semigroup $e^{\frac{t}{\eps^2}B_\eps}$, which will be used to prove the convergence and establish the convergence rate of the solution to  the VPB system \eqref{VPB1}--\eqref{VPB2i} towards the solution to the NSPF system \eqref{NSP_1}--\eqref{compatible}.

Firstly,  introduce a function sapce  $ H^k_P\ (L^2_P=H^0_P)$
with the norm
\bmas\|f\|_{H^k_P}&=\(\intr (1+|\xi|^2)^k
\|\hat{f}\|^2_{\xi} d\xi \)^{1/2}\\
&=\(\intr (1+|\xi|^2)^k
     \(\|\hat{f}\|^2+\frac1{|\xi|^2}\lt|(\hat{f},\sqrt{M})\rt|^2\)d\xi
    \)^{1/2},
\emas
where $\hat{f}=\hat{f}(\xi,v)$ is the Fourier transformation of $f(x,v)$. Note that
\bq
\|f\|^2_{H^k_P}=\|f\|^2_{L^2_v(H^k_x)}+\| \nabla_x\Delta_x^{-1}
(f,\sqrt{M})\|^2_{H^k_x}.
\eq
Also, we define the norm
\bq \|f\|_{L^{\infty}_P}=\|f\|_{L^{\infty}_x(L^{2}_v)}+\| \nabla_x\Delta_x^{-1} (f,\sqrt{M})\|_{L^{\infty}_x}. \label{norm1}\eq
For any $f_0\in L^2(\R^3_{x}\times \R^3_v)$, set
 \be
  e^{\frac{t}{\eps^2}B_\eps}f_0=(\mathcal{F}^{-1}e^{\frac{t}{\eps^2}B_\eps(\xi)}\mathcal{F})f_0.
  \ee
By Lemma \ref{SG_1}, it holds that
 $$
 \|e^{\frac{t}{\eps^2}B_\eps}f_0\|_{H^k_P}=\intr (1+|\xi|^2)^k\|e^{\frac{t}{\eps^2}B_\eps(\xi)}\hat f_0\|^2_{\xi} d\xi\le \intr (1+|\xi|^2)^k\|\hat f_0\|^2_{\xi} d\xi
=\|f_0\|_{H^k_P}.
 $$
This means that the operator $\eps^{-2} B_\eps$ generates a strongly continuous contraction semigroup $e^{\frac{t}{\eps^2}B_\eps}$ in $H^k_P$, and therefore, $f(x,v,t)=e^{\frac{t}{\eps^2}B_\eps}f_0$ is a global solution to the linearized VPB system~\eqref{VPB} for any $f_0\in H^k_P$.

Consider the following linear Navier-Stokes-Poisson-Fourier (NSPF) system of $(n,m,q)(t,x)$:
\bgr
\Tdx\cdot m=0,\label{LNSP_1}\\
n+\sqrt{\frac23}q-\Delta_x^{-1}n=0,\label{LNSP_2}\\
\dt m+\kappa_0\Delta_x m+\Tdx p =H_1,\label{LNSP_3}\\
\dt \bigg(q-\sqrt{\frac23}n\bigg)+\kappa_1\Delta_x q=H_2,\label{LNSP_4}
\egr
where $H_1=(H^1_1,H^2_1,H^3_1)$ and $H_2$ are given functions, $p$ is the pressure satisfying $p=\Delta^{-1}_x\divx H_1$, and the initial data $(n,m,q)(0)$ satisfies   \eqref{NSP_5i} and \eqref{compatible}.

For any $U_0=U_0(x,v)\in N_0$, we define
\be V(t,\xi)\hat{U}_0=\sum_{j=0,2,3}e^{-b_j(|\xi|)t}\(\hat{U}_0,h_j(\xi)\)_{\xi}h_j(\xi),\label{v1}\ee
where $b_j(|\xi|)$ and $h_j(\xi)$, $j=0,2,3$ are defined by \eqref{eigen4} and \eqref{eigen} respectively.
Set
\be V(t)U_0=(\mathcal{F}^{-1}V(t,\xi)\mathcal{F})U_0. \label{v2}\ee

Then, we can represent the solution to the NSPF system \eqref{LNSP_1}--\eqref{LNSP_4} by the semigroup $V(t)$ as follows.

\begin{lem} \label{sem}
For any $f_0\in L^2_v(L^2_x)$ and $H_i\in L^1_t(L^2_x)$, $i=1,2,$ define
$$U(t,x,v)=V(t)P_0f_0+\intt V(t-s)H(s)ds,$$
 where
$$H(t,x,v)=H_1(t,x)\cdot v\chi_0+H_2(t,x)\chi_4.$$
Let $(n,m,q)=((U,\chi_0),(U,v\chi_0),(U,\chi_4))$. Then $(n,m,q)(t,x)\in L^\infty_t(L^2_x)$ is an unique global solution to the linear NSPF system \eqref{LNSP_1}--\eqref{LNSP_4} with the initial data $(n,m,q)(0)$ satisfies \eqref{NSP_5i}--\eqref{compatible}.
\end{lem}

\begin{proof}
By taking Fourier transform to \eqref{LNSP_1}--\eqref{LNSP_4}, we have
\bgr
i\xi\cdot \hat{m}=0,\label{LNSP_1a}\\
\hat{n}+\frac1{|\xi|^2}\hat{n}+\sqrt{\frac23}\hat{q}=0,\label{LNSP_2a}\\
\dt \hat{m}-\kappa_0|\xi|^2 \hat{m}+i\xi\hat{p} =\hat{H}_1,\label{LNSP_3a}\\
\dt \bigg(\hat{q}-\sqrt{\frac23}\hat{n}\bigg)-\kappa_1|\xi|^2 \hat{q}=\hat{H}_2,\label{LNSP_4a}
\egr
where the initial data $(\hat{n},\hat{m},\hat{q})(0)$ satisfies
\be
 \hat{m}(0)=\O_1\(P_0\hat{f}_0,v\chi_0\),\quad \hat{q}(0)-\sqrt{\frac23}\hat{n}(0)= \bigg(P_0\hat{f}_0,\chi_4-\sqrt{\frac23}\chi_0\bigg),\label{LNSP_5i}
\ee
with $\O_1=\O_1(\xi)$ being a projection defined by
$$\O_1y=y-\(y\cdot\frac{\xi}{|\xi|}\)\frac{\xi}{|\xi|},\quad \forall y\in \R^3.$$

By \eqref{LNSP_2a} and \eqref{LNSP_4a}, we obtain
\bma
&\hat{n}=-\sqrt{\frac23}\frac{|\xi|^2}{1+|\xi|^2}\hat{q}, \label{n1}\\
& \frac{3+5|\xi|^2}{3+3|\xi|^2}\dt\hat{q} -\kappa_1|\xi|^2 \hat{q}=\hat{H}_2. \label{n2}
\ema
It follows from \eqref{n2}, \eqref{LNSP_5i}, \eqref{eigen} and \eqref{eigen4} that
\bma
\hat{q}(\xi,t)&=e^{-b_0(|\xi|)t} \hat{q}(0)  +\intt e^{-b_0(|\xi|)(t-s)}\frac{3+3|\xi|^2}{3+5|\xi|^2}\hat{H}_2(s)ds\nnm\\
&=e^{-b_0(|\xi|)t} \(P_0\hat{f}_0,h_0(\xi)\)_{\xi}(h_0(\xi),\chi_4)\nnm\\
&\quad +\intt e^{-b_0(|\xi|)(t-s)} \(\hat{H}(s),h_0(\xi)\)_{\xi}(h_0(\xi),\chi_4)ds.\label{n3}
\ema
This and \eqref{n1} imply
\bma
\hat{n}(\xi,t)&=e^{-b_0(|\xi|)t} \(P_0\hat{f}_0,h_0(\xi)\)_{\xi}(h_0(\xi),\chi_0)\nnm\\
&\quad +\intt e^{-b_0(|\xi|)(t-s)} \(\hat{H}(s),h_0(\xi)\)_{\xi}(h_0(\xi),\chi_0)ds. \label{n4}
\ema
By \eqref{LNSP_1a} and \eqref{LNSP_3a}, we have
\bma
\hat{m}(\xi,t)&=e^{-b_2(|\xi|)t} \O_1\hat{m}(0)  +\intt e^{-b_2(|\xi|)(t-s)}\O_1\hat{H}_1(s)ds\nnm\\
&=\sum_{j=2,3}e^{-b_j(|\xi|)t} \(P_0\hat{f}_0,h_j(\xi)\)_{\xi}(h_j(\xi),v\chi_0) \nnm\\
&\quad+\sum_{j=2,3}\intt e^{-b_j(|\xi|)(t-s)} \(\hat{H}(s),h_j(\xi)\)_{\xi}(h_j(\xi),v\chi_0)ds. \label{n5}
\ema
Noting that  $(h_0(\xi),v\chi_0)=0$ and $(h_j(\xi),\chi_0)=(h_j(\xi),\chi_4)=0$, $j=2,3$, we can prove the lemma by using \eqref{n3}--\eqref{n5}.
\end{proof}

We have the time decay rates of the semigroups $e^{\frac{t}{\eps^2}B_\eps}$ and $V(t)$ as follows.

\begin{lem}\label{time}
For any $\eps\in (0,1)$, $\alpha\in \mathbb{N}^3$ and any $f_0\in L^2(\R^3_x\times \R^3_v)$, we have
\bma
 \left\|P_0\dxa e^{\frac{t}{\eps^2}B_\eps}f_0 \right\|_{L^{2}_{P} }
 &\le C (1+t)^{-\frac14-\frac{m}2}  \(\|\dxa f_0\|_{L^{2}_{x,v}}+\|\dx^{\alpha'}f_0\|_{L^{2}_v(L^1_x)} \), \label{time1}
\\
 \left\|P_1\dxa e^{\frac{t}{\eps^2}B_\eps}f_0 \right\|_{L^{2}_{x,v}}  &\le C\(\eps (1+t)^{-\frac34-\frac{m}2} + e^{-\frac{dt}{\eps^2}}\) \(\|\dxa f_0\|_{L^2_v(H^1_x)}+\|\dx^{\alpha'}f_0\|_{L^{2}_v(L^1_x)}\), \label{time2}
\ema
where $\alpha'\le \alpha$, $m=|\alpha-\alpha'|$, and $d,C>0$ are two constants independent of $\eps$.

Moreover, if $P_df_0=0$, then
\bma
 \left\|P_0\dxa e^{\frac{t}{\eps^2}B_\eps}f_0 \right\|_{L^{2}_{P} }
 &\le C (1+t)^{-\frac34-\frac{m}2}  \(\|\dxa f_0\|_{L^{2}_{x,v}}+\|\dx^{\alpha'}f_0\|_{L^{2}_v(L^1_x)} \), \label{time3}
\\
 \left\|P_1\dxa e^{\frac{t}{\eps^2}B_\eps}f_0 \right\|_{L^{2}_{x,v}}  &\le C\(\eps (1+t)^{-\frac54-\frac{m}2} + e^{-\frac{dt}{\eps^2}}\) \(\|\dxa f_0\|_{L^2_v(H^1_x)}+\|\dx^{\alpha'}f_0\|_{L^{2}_v(L^1_x)}\), \label{time4}
\ema
and if $P_0f_0=0$, then
\bma
 \left\|P_0\dxa e^{\frac{t}{\eps^2}B_\eps}f_0 \right\|_{L^{2}_{P} }
 &\le C\( \eps(1+t)^{-\frac54-\frac{m}2}+ e^{-\frac{dt}{\eps^2}}\)  \(\|\dxa f_0\|_{L^2_v(H^1_x)}+\|\dx^{\alpha'}f_0\|_{L^{2}_v(L^1_x)} \), \label{time5}
\\
 \left\|P_1\dxa e^{\frac{t}{\eps^2}B_\eps}f_0 \right\|_{L^{2}_{x,v}}  &\le C\(\eps^2 (1+t)^{-\frac74-\frac{m}2} + e^{-\frac{dt}{\eps^2}}\) \(\|\dxa f_0\|_{L^2_v(H^2_x)}+\|\dx^{\alpha'}f_0\|_{L^{2}_v(L^1_x)}\). \label{time6}
\ema
\end{lem}

\begin{proof}
By Theorem \ref{E_3}, we have
\bma
\left\|P_0\dxa e^{\frac{t}{\eps^2}B_\eps}f_0\right\|^2_{L^{2}_{P}}
=& \intr  \left\|P_0\xi^{\alpha}e^{\frac{t}{\eps^2}B_{\eps}(\xi)}\hat{f}_0 \right\|^2_{\xi}d\xi  \nnm\\
\le &  \int_{|\xi|\le \frac{r_0}{\eps}} \left\|\xi^{\alpha}P_0S_1(t,\xi,\eps)\hat{f}_0\right\|^2_{\xi}d\xi +\intr \left\|\xi^{\alpha}S_2(t,\xi,\eps)\hat{f}_0\right\|^2_{\xi}d\xi, \label{D1a}
\\
\left\|P_1\dxa e^{\frac{t}{\eps^2}B_\eps}f_0\right\|^2_{L^{2}_{x,v}}
=& \intr  \left\|P_1\xi^{\alpha}e^{\frac{t}{\eps^2}B_{\eps}(\xi)}\hat{f}_0 \right\|^2 d\xi  \nnm\\
\le &  \int_{|\xi|\le \frac{r_0}{\eps}} \left\|\xi^{\alpha}P_1S_1(t,\xi,\eps)\hat{f}_0\right\|^2 d\xi +\intr \left\|\xi^{\alpha}S_2(t,\xi,\eps)\hat{f}_0\right\|^2 d\xi. \label{D1b}
\ema
 By \eqref{B_3} and the fact that
 \bmas
  \intr \frac{(\xi^\alpha)^2}{|\xi|^2}\lt|(\hat{f_0},\chi_0)\rt|^2d\xi
&\leq
  \sup_{|\xi|\leq 1} \lt|\xi^{\alpha'}(\hat{f_0},\chi_0)\rt|^2\int_{|\xi|\leq1}\frac{1}{|\xi|^2}d\xi
 + \int_{|\xi|> 1} (\xi^\alpha)^2\lt|(\hat{f_0},\chi_0)\rt|^2d\xi
 \\
&\leq
 C\(\|\dx^{\alpha'}f_0\|^2_{L^{2}_v(L^1_x)}+\|\da_x f_0\|^2_{L^{2}_{x,v}}\),\quad \alpha'\le \alpha,
 \emas
 we can estimate the second term on the right hand side of \eqref{D1a}--\eqref{D1b} as follows:
\bma
\intr (\xi^\alpha)^2 \|S_2(t,\xi,\eps)\hat{f}_0\|^2_{\xi}d\xi&\le C e^{-2\frac{dt}{\eps^2}}\intr(\xi^\alpha)^2\|\hat{f}_0\|^2_{\xi}d\xi \nnm\\
&\le Ce^{-2\frac{dt}{\eps^2}}\(\|\dxa f_0\|^2_{L^{2}_{x,v}}+\|\dx^{\alpha'}f_0\|^2_{L^{2}_v(L^1_x)}\).\label{t3a}
\ema

By \eqref{E_5} and Theorem \ref{spect3}, we have
\bma
\int_{|\xi|\le \frac{r_0}{\eps}}\|\xi^{\alpha}P_0S_1(t,\xi,\eps)\hat{f}_0\|^2_{\xi}d\xi
&\le  C\int_{|\xi|\le \frac{r_0}{\eps}} e^{-c|\xi|^2t}|\xi|^{2|\alpha|}\|\hat{f}_0\|^2_{\xi}d\xi\nnm\\
&\le C(1+t)^{-\frac12-m}\(\|\dxa f_0\|^2_{L^{2}_{x,v}}+\|\dx^{\alpha'}f_0\|^2_{L^{2}_v(L^1_x)}\), \label{t1}
\\
\int_{|\xi|\le \frac{r_0}{\eps}}\|\xi^{\alpha}P_1S_1(t,\xi,\eps)\hat{f}_0\|^2d\xi
&\le  C\int_{|\xi|\le \frac{r_0}{\eps}} e^{-c|\xi|^2t} \eps^2|\xi|^{2+2|\alpha|}\|\hat{f}_0\|^2_{\xi}d\xi\nnm\\
&\le C\eps^2(1+t)^{-\frac32-m}\(\|\dxa f_0\|^2_{L^2_v(H^1_x)}+\|\dx^{\alpha'}f_0\|^2_{L^{2}_v(L^1_x)}\),\label{t2}
\ema
where $\alpha'\le \alpha$, $m=|\alpha-\alpha'|$ and $c>0$ is a constant. Combining \eqref{D1a}--\eqref{t2}, we prove \eqref{time1} and \eqref{time2}.

If $P_df_0=0$, then
\bma
\int_{|\xi|\le \frac{r_0}{\eps}}\|\xi^{\alpha}P_0S_1(t,\xi,\eps)\hat{f}_0\|^2_{\xi}d\xi
&\le  C\int_{|\xi|\le \frac{r_0}{\eps}} e^{-c|\xi|^2t}  |\xi|^{2|\alpha|}\|\hat{f}_0\|^2 d\xi\nnm\\
&\le C (1+t)^{-\frac32-m}\(\|\dxa f_0\|^2_{L^{2}_{x,v}}+\|\dx^{\alpha'}f_0\|^2_{L^{2}_v(L^1_x)}\),\label{t3}
\\
\int_{|\xi|\le \frac{r_0}{\eps}}\|\xi^{\alpha}P_1S_1(t,\xi,\eps)\hat{f}_0\|^2 d\xi
&\le  C\int_{|\xi|\le \frac{r_0}{\eps}} e^{-c|\xi|^2t} \eps^2 |\xi|^{2+2|\alpha|}\|\hat{f}_0\|^2 d\xi\nnm\\
&\le C\eps^2(1+t)^{-\frac52-m}\(\|\dxa f_0\|^2_{L^2_v(H^1_x)}+\|\dx^{\alpha'}f_0\|^2_{L^{2}_v(L^1_x)}\).\label{t4}
\ema
If $P_0f_0=0$, then
\bma
\int_{|\xi|\le \frac{r_0}{\eps}}\|\xi^{\alpha}P_0S_1(t,\xi,\eps)\hat{f}_0\|^2_{\xi} d\xi
&\le  C\int_{|\xi|\le \frac{r_0}{\eps}} e^{-c|\xi|^2t} \eps^2|\xi|^{2+2|\alpha|}\|\hat{f}_0\|^2 d\xi\nnm\\
&\le C\eps^2(1+t)^{-\frac52-m}\(\|\dxa f_0\|^2_{L^2_v(H^1_x)}+\|\dx^{\alpha'}f_0\|^2_{L^{2}_v(L^1_x)}\),\label{t5}
\\
\int_{|\xi|\le \frac{r_0}{\eps}}\|\xi^{\alpha}P_1S_1(t,\xi,\eps)\hat{f}_0\|^2  d\xi
&\le  C\int_{|\xi|\le \frac{r_0}{\eps}} e^{-c|\xi|^2t} \eps^4|\xi|^{4+2|\alpha|}\|\hat{f}_0\|^2 d\xi\nnm\\
&\le C\eps^4(1+t)^{-\frac72-m}\(\|\dxa f_0\|^2_{L^{2}_v(H^2_x)}+\|\dx^{\alpha'}f_0\|^2_{L^{2}_v(L^1_x)}\).\label{t6}
\ema
Combining \eqref{D1a}, \eqref{D1b} and \eqref{t3}--\eqref{t6}, we obtain \eqref{time3}--\eqref{time6}. The proof is then  completed.
\end{proof}

%By \eqref{v1}, \eqref{v2} and a similar argument as Lemma \ref{time}, we can show
\begin{lem} \label{timev}
For any  $\alpha\in \mathbb{N}^3$ and any $u_0\in N_0$, we have
\be
 \left\|\dxa V(t)u_0\right\|_{L^{2}_{P}} \le C (1+t)^{-\frac34-\frac{m}2}  \(\|\dxa u_0\|_{L^{2}_{x,v}}+\|\dx^{\alpha'}u_0\|_{L^{2}_v(L^1_x)}\), \label{time1b}
\ee
where $\alpha'\le \alpha$, $m=|\alpha-\alpha'|$, and $C>0$ is a constant. Moreover,
\be
 \left\|\dxa V(t)u_0\right\|_{L^{\infty}_P}
  \le C (1+t)^{-\frac34}t^{-\frac{m}2}   \(\|\dx^{\alpha'} u_0\|_{L^{\infty}_x(L^{2}_v)}+\|\dx^{\alpha'}u_0\|_{L^{2}_{x,v}}\). \label{time1a}
\ee
\end{lem}

\begin{proof}
By \eqref{v1}, we have
\bma
 &V(t,\xi)\hat{u}_0=e^{-b_0(|\xi|)t}R_0(\xi)\bigg(\hat{U}_0-\sqrt{\frac23}\hat{U}_4\bigg)+e^{-b_2(|\xi|)t}\sum^3_{j=1}R_j(\xi) \hat{U}_j, \label{v5}
 \\
 &\frac{\xi}{|\xi|^2}(V(t,\xi)\hat{u}_0,\chi_0) =e^{-b_0(|\xi|)t}R_4(\xi)\bigg(\hat{U}_0-\sqrt{\frac23}\hat{U}_4\bigg), \label{v6}
 \ema
where $U_j=(u_0,\chi_j)$, $j=0,1,2,3,4$, and
\bmas
R_0(\xi)&=\frac{2|\xi|^2}{3+5|\xi|^2}\chi_0 -\frac{\sqrt6(1+|\xi|^2)}{3+5|\xi|^2}\chi_4,\\
R_j(\xi)&=v_j\chi_0- \frac{(v\cdot\xi)}{|\xi|^2}\xi_j\chi_0,\quad j=1,2,3,\\
R_4(\xi)&=\frac{2\xi}{3+5|\xi|^2}.
\emas

 Thus, it follows from \eqref{v5} and \eqref{v6} that
\bma
 \left\|\dxa V(t)u_0\right\|^2_{L^{2}_{P}} &\le C \intr (\xi^{\alpha})^2e^{-2c|\xi|^2t}\|\hat{u}_0\|^2d\xi\nnm\\
 &\le \sup_{|\xi|\le 1}(\xi^{\alpha'})^2\|\hat{u}_0\|^2\int_{|\xi|\le 1} (\xi^{\alpha-\alpha'})^2e^{-2c|\xi|^2t}d\xi+\int_{|\xi|\ge 1} e^{-2ct}(\xi^{\alpha})^2\|\hat{u}_0\|^2d\xi\nnm\\
 &\le C(1+t)^{-3/2-m}  \(\|\dxa u_0\|^2_{L^{2}_{x,v}}+\|\dx^{\alpha'}u_0\|^2_{L^{2}_v(L^1_x)}\), \label{v8}
\ema
where $\alpha'\le \alpha$, $m=|\alpha-\alpha'|$ and $c>0$ is a constant. This proves \eqref{time1}

Then,  \eqref{time1a} can be proved as follows. %We only deal with the term $\| V(t)u_0\|_{L^{\infty}_x(L^2_v)}$.
For $t>1$, we have
\bma
\left\|\dxa V(t)u_0\right\|_{L^{\infty}_P}
&\le C\intr |\xi^{\alpha}|e^{-c|\xi|^2t}\|\hat{u}_0\|d\xi\nnm\\
&\le C\(\intr (\xi^{\alpha-\alpha'})^2 e^{-2c|\xi|^2t}d\xi\)^{1/2}\(\intr (\xi^{\alpha'})^2\|\hat{u}_0\|^2d\xi\)^{1/2}\nnm\\
&\le C(1+t)^{-3/4-m/2}\|\dx^{\alpha'}u_0\|_{L^{2}_{x,v}}. \label{v3}
\ema

For $t\le1$, we only estimate the term $\| V(t)u_0\|_{L^{\infty}_x(L^2_v)}$. Introduce the smooth cut-off function $\chi_1(\xi)$ with
$$ \chi_1(\xi)=1,\quad |\xi|\le 1;\quad \chi_1(\xi)=0,\quad |\xi|\ge 2,$$
and decompose $V(t)u_0$ into
\be
 \dxa V(t)u_0= \mathcal{F}^{-1}( \xi^{\alpha} V(t,\xi)\hat{u}_0\chi_1(\xi))+ \mathcal{F}^{-1}(\xi^{\alpha}V(t,\xi)\hat{u}_0\chi_2(\xi))=:I_1+I_2, \label{d1}
 \ee
 where $\chi_2(\xi)=1-\chi_1(\xi)$.
We can bound $I_1$ by
\be
 \|I_1\|\le C\int_{|\xi|\le 2} |\xi^{\alpha}|e^{-c|\xi|^2t} \|\hat{u}_0\|d\xi \le C(1+t)^{-3/4-m/2} \|\dx^{\alpha'}u_0\|_{L^{2}_{x,v}}. \label{v7}
\ee

Denote $G_j$, $j=0,1,2,3$ by their Fourier transforms:
$$\hat{G}_0(t,\xi)=e^{-b_0(|\xi|)t}R_0(\xi)\chi_2(\xi),\quad \hat{G}_j(t,\xi)=e^{-b_2(|\xi|)t}R_j(\xi)\chi_2(\xi),\,\,\,j=1,2,3.$$
We can represent $I_2$ as
\be
I_2=\dxa G_0(t)\ast \bigg( U_0-\sqrt{\frac23} U_4\bigg)+\sum^3_{j=1}\dxa G_j(t)\ast  U_j.\label{d10}
\ee

 Since
 \bmas
\left\|\pt^{\alpha}_{\xi}(\xi^{\beta}\hat{G}_j(t,\xi))\right\|&\le C\sum_{\alpha'\le \alpha}\pt^{\alpha'}_{\xi}e^{-b_i(|\xi|)t}\|\pt^{\alpha-\alpha'}_{\xi}(\xi^{\beta}R_j(\xi)\chi_2(\xi))\|\\
&\le C\sum_{\alpha'\le \alpha}t^{|\alpha'|/2}(1+|\xi|^2t)^{|\alpha'|/2}e^{-c|\xi|^2t}|\xi|^{|\beta|-|\alpha-\alpha'|}\\
&\le C|\xi|^{|\beta|}\(t^{|\alpha|/2}(1+|\xi|^2t)^{|\alpha|/2}+|\xi|^{-|\alpha|}\)e^{-c|\xi|^2t},\quad |\xi|\ge 1,
\emas
for $i=j=0$ and $i=2,j=1,2,3,$ it follows that
\be
 |x^{2\alpha}|\left\|\dxb G_j(t,x)\right\|\le C\int_{|\xi|\ge 1} \left\|\pt^{2\alpha}_{\xi}(\xi^{\beta}\hat{G}_j(t,\xi))\right\|d\xi\le Ct^{-\frac32-\frac{|\beta|}2}t^{|\alpha|},\quad t\le 1. \label{d2a}
\ee
For any $n\in \N$, we take $|\alpha| = 0$ for $|x|^2\le t$ and $|\alpha| = n$ for $|x|^2\ge t$, and obtain from \eqref{d2a},
\be
 \left\|\dxb G_j(t,x)\right\|\le Ct^{-\frac32-\frac{|\beta|}2}\(1+\frac{|x|^2}{t}\)^{-n}. \label{d2}
\ee

Thus, it follows from \eqref{d10} and \eqref{d2} that
\bma
 \|I_2\|%&\le \bigg|G_0\ast\bigg(\dxa U_0-\sqrt{\frac23}\dxa U_4\bigg)\bigg| +\sum^3_{j=1}|G_j\ast \dxa U_j|\\
 &\le C\|\dx^{\alpha-\alpha'}G_0(t)\|_{L^{2}_v(L^1_x)}\(\|\dx^{\alpha'} U_0\|_{L^{\infty}_x}+\|\dx^{\alpha'} U_4\|_{L^{\infty}_x}\)\nnm\\
 &\quad+ \|\dx^{\alpha-\alpha'}G_j(t)\|_{L^{2}_v(L^1_x)}\|\dx^{\alpha'} U_j\|_{L^{\infty}_x}\nnm\\
 &\le Ct^{-m/2}\|\dx^{\alpha'}u_0\|_{L^{\infty}_x(L^{2}_v)},\quad m=|\alpha-\alpha'|. \label{v2}
\ema
Combining \eqref{d1}, \eqref{v7} and \eqref{v2}, we obtain
\be \|\dxa V(t)u_0\|_{L^{\infty}_x(L^{2}_v)}\le Ct^{-m/2}\(\|\dx^{\alpha'} u_0\|_{L^{\infty}_x(L^{2}_v)}+\|\dx^{\alpha'}u_0\|_{L^{2}_{x,v}}\),\quad t\le 1. \label{v4} \ee
 By combining \eqref{v3}, \eqref{d1}, \eqref{v7} and \eqref{v4}, we obtain \eqref{time1a}. The proof is completed.
\end{proof}

We now prepare two  lemmas that  will be used to study the  fluid dynamical approximation of the semigroup $e^{\frac{t}{\eps^2}B_\eps}$.

\begin{lem} \label{a2}
Let $a>1$ be a constant. For any function $\theta(\xi)$ satisfying $|\pt^{\alpha}_{\xi}\theta(\xi)|\le C(1+|\xi|)^{-a-|\alpha|}$ for any $\alpha\in \N^3$, we have
\be
 \intr e^{ix\cdot\xi}e^{-b_j(|\xi|)t}\theta(\xi) d\xi\in L^1(\R^3_x), \label{a1}
\ee
where $b_j(|\xi|)$, $j=-1,0,1,2,3$ are defined by \eqref{eigen4}.
\end{lem}

\begin{proof}
For any $\alpha\in \mathbb{N}^3$ and $j=-1,0,1,2,3$, by \eqref{eigen4} we obtain
$$
 \left|\pt^{\alpha}_{\xi} e^{-b_j(|\xi|)t}\right| \le C t^{|\alpha|/2}(1+|\xi|^2t)^{|\alpha|/2}e^{-c|\xi|^2t},
$$
where $c,C>0$ are two constants.
This implies that
$$
\left|x^{\alpha}\intr e^{ix\cdot\xi}e^{-b_j(|\xi|)t} d\xi\right|\le C\intr \left|\pt^{\alpha}_{\xi}e^{-b_j(|\xi|)t}\right|d\xi\le C  t^{-3/2+|\alpha|/2},
$$
which gives
\be
 \left|\mathcal{F}^{-1} \(e^{-b_j(|\xi|)t} \)\right|\le Ct^{-\frac32}\(1+\frac{|x|^2}{t}\)^n,\quad \forall n\in \N. \label{a4}
\ee

Since $|\pt^{\alpha}_{\xi}\theta(\xi)|\le C(1+|\xi|)^{-a-|\alpha|}$ for any $\alpha\in \N^3$, we have
\be  \left|\mathcal{F}^{-1} \theta(\xi)\right|\le C\frac1{|x|^2}(1+|x|)^{-n},\quad \forall n\in \N.  \label{a5}\ee
Thus, it follows from \eqref{a4} and \eqref{a5} that
$$
 \left\|\mathcal{F}^{-1} \(e^{-b_j(|\xi|)t}\theta(\xi)\) \right\|_{L^1_x}
 \le  C \left\|\mathcal{F}^{-1} \(e^{-b_j(|\xi|)t} \)\right\|_{L^1_x}  \left\|\mathcal{F}^{-1} \theta(\xi)\right\|_{L^1_x}\le C,
$$
which proves \eqref{a1}.
\end{proof}

\begin{lem} \label{S2}
For any $f_0\in N_0$, we have
\be
\|S_{2}(t,\xi,\eps)f_0\|_{\xi}\le C\(\eps|\xi|1_{\{\eps|\xi|\le r_0\}}+1_{\{\eps|\xi|\ge r_0\}}\)e^{-\frac{dt}{\eps^2}}\| f_0\|_{\xi} .\label{S5}
\ee
\end{lem}

\begin{proof}
Define a projection $P_{\eps}(\xi)$ by
$$P_{\eps}(\xi)f=\sum^3_{j=-1}\(f,\overline{\psi_j(\xi,\eps)}\)_{\xi} \psi_j(\xi,\eps),\quad \forall f\in L^2(\R^3_v),$$
where $ \psi_j(\xi,\eps)$, $j=-1,0,1,2,3$ are the eigenfunctions of $B_{\eps}(\xi)$ defined by \eqref{eigen1} for $\eps|\xi|\le r_0$.

By Theorem \ref{E_3}, we have
\be
S_2(t,\xi,\eps)=S_{21}(t,\xi,\eps)+S_{22}(t,\xi,\eps),\label{S3}
\ee
where
\bma
S_{21}(t,\xi,\eps)&=S(t,\xi,\eps)1_{\{\eps|\xi|\le r_0\}}\(I- P_{\eps}(\xi)\),\\
S_{22}(t,\xi,\eps)&=S(t,\xi,\eps)1_{\{\eps|\xi|\ge r_0\}}.
\ema
It holds that
\be
\|S_{2j}(t,\xi,\eps)g\|_{\xi}\le Ce^{-\frac{dt}{\eps^2}}\| g\|_{\xi},\quad \forall g\in L^2(\R^3_v),\,\,\,j=1,2.
\ee

Since $h_j(\xi)$, $j=-1,0,1,2,3$ are the orthonormal basis of $N_0$, it follows that for any $f_0\in N_0$,
$$
f_0-P_{\eps}(\xi)f_0
=\sum^3_{j=-1}(f_0,h_j(\xi))_\xi h_j(\xi)-\sum^3_{j=-1}(f_0,\overline{\psi_j(\xi,\eps)})_\xi \psi_j(\xi,\eps),
$$
which together with \eqref{eigen1} gives rise to
\be
\|S_{21}(t,\xi,\eps)f_0\|_{\xi}\le C\eps|\xi|1_{\{\eps|\xi|\le r_0\}}e^{-\frac{dt}{\eps^2}}\|f_0\|_{\xi},\quad \forall f_0\in N_0.\label{S4}
\ee
By combining \eqref{S3}--\eqref{S4}, we obtain \eqref{S5}.
\end{proof}

Now we are going to estimate the first and second order expansions of the semigroup $e^{\frac{t}{\eps^2}B_\eps}$.

\begin{lem}\label{fl-1}
For any $\eps\in (0,1)$ and  any $f_0\in L^{2}_{v}(L^2_x)$, we have
\bma
 \left\|e^{\frac{t}{\eps^2}B_\eps}f_0-V(t)P_0f_0\right\|_{L^{\infty}_P}
 \le& C\bigg(\eps (1+t)^{-\frac32}+\(1+\frac{t}{\eps}\)^{-q}\bigg)\nnm\\
 &\times\(\|f_0\|_{L^2_v(H^4_x)}+\|f_0\|_{L^2_v(W^{4,1}_x)}+\|\Tdx\phi_0\|_{L^p_x}\), \label{limit1}
\ema
for any $1<p<2$ and $q=3/p-3/2$, where $ \phi_0= \Delta^{-1}_x(f_0,\chi_0)$. Moreover, if $f_0$ satisfies \eqref{initial}, then
\be
 \left\|e^{\frac{t}{\eps^2}B_\eps}f_0-V(t)P_0f_0\right\|_{L^{\infty}_P}
 \le C\eps (1+t)^{-\frac32}\(\|f_0\|_{L^2_v(H^3_x)}+\|f_0\|_{L^2_v(L^{1}_x)}\). \label{limit1a}
\ee
\end{lem}

\begin{proof}First, we prove \eqref{limit1} as follows.
 By Theorem \ref{E_3}, we have
\bma
\left\|e^{\frac{t}{\eps^2}B_\eps}f_0-V(t)P_0f_0\right\|=&C\left\| \intr e^{ix\cdot\xi} \(e^{\frac{t}{\eps^2}B_{\eps}(\xi)}\hat{f}_0 -V(t,\xi)P_0\hat{f}_0\)d\xi\right\|\nnm\\
\le & C\bigg\|\int_{|\xi|\le \frac{r_0}{\eps}}e^{ix\cdot\xi}\(S_1(t,\xi,\eps)\hat{f}_0-V(t,\xi)P_0\hat{f}_0\)d\xi\bigg\|\nnm\\
&+C\int_{|\xi|\ge \frac{r_0}{\eps}} \left\|V(t,\xi)P_0\hat{f}_0\right\|d\xi+C\intr \left\|S_2(t,\xi,\eps)\hat{f}_0\right\|d\xi\nnm\\
=:&I_1+I_2+I_3. \label{S_1}
\ema
We estimate $I_j$, $j=1,2,3$ as follows. By \eqref{E_5} and Theorem \ref{spect3}, we have
$$
S_1(t,\xi,\eps)\hat{f}_0 %&=\sum^3_{j=-1}e^{\frac{t}{\eps^2}\lambda_j(|\xi|,\eps)}\(\hat{f}_0,\overline{\psi_j(\xi,\eps)}\)_{\xi}\psi_j(\xi,\eps) \\
=\sum^3_{j=-1}e^{\frac{\eta_j(|\xi|)}{\eps}t-b_j(|\xi|)t+ O(\eps|\xi|^3)t}\[\(P_0\hat{f}_0,h_j(\xi)\)_{\xi}h_j(\xi)+O(\eps |\xi|)\],
$$
which leads to
\bma
I_1\le& \sum^3_{j=-1}\int_{|\xi|\le \frac{r_0}{\eps}}\bigg\|e^{\frac{\eta_j(|\xi|)}{\eps}t-b_j(|\xi|)t+ O(\eps|\xi|^3)t}\[\(P_0\hat{f}_0,h_j(\xi)\)_{\xi}h_j(\xi)+O(\eps |\xi|)\]\nnm\\
&\qquad-e^{ \frac{\eta_j(|\xi|)}{\eps}t-b_j(|\xi|)t}\(P_0\hat{f}_0,h_j(\xi)\)_{\xi}h_j(\xi)\bigg\|d\xi \nnm\\
&+\sum_{j=-1,1}\bigg\|\int_{|\xi|\le \frac{r_0}{\eps}}e^{ix\cdot\xi }e^{ \frac{\eta_j(|\xi|)}{\eps}t-b_j(|\xi|)t}\(P_0\hat{f}_0,h_j(\xi)\)_{\xi}h_j(\xi)d\xi\bigg\|\nnm\\
=:&I_{11}+I_{12}. \label{S_2}
\ema
For $I_{11}$, it follows from \eqref{eigen4} and \eqref{eigen1} that
\bma
I_{11}&\le C\eps\int_{|\xi|\le \frac{r_0}{\eps}}e^{-c|\xi|^2t}\(|\xi|^3t\|P_0\hat{f}_0\|_{\xi} +|\xi|\|\hat{f}_0\|_{\xi}\) d\xi \nnm\\
&\le C\eps (1+t)^{-\frac32}\(\|f_0\|_{L^2_v(H^3_x)}+\|f_0\|_{L^{2}_v(L^1_x)}\). \label{S_3}
\ema
Then, we estimate $I_{12}$ as follows.
By \eqref{eigen}, we have
\be
 \(P_0\hat{f}_0,h_j(\xi)\)_{\xi}h_j(\xi)
= \sum^4_{k=0}\hat{R}_{jk}  \hat{F}^1_{k}+ \frac{ \xi}{|\xi|^2}\bigg( R_{j5}(\xi) \hat{F}^1_{0} +R_{j6}(\xi)\sum^3_{l=1} \xi_l \hat{F}^1_{l}\bigg), \label{a7a}
\ee
where $j=-1,1$, and
\bmas
R_{j0}(\xi)&=\frac{1+|\xi|^2}{3+5|\xi|^2}\bigg(\frac32\chi_0
+\sqrt{\frac32}\chi_4\bigg)+ \frac{j\sqrt3}{2\sqrt{3+5|\xi|^2}}(v\cdot\xi)\chi_0 ,
\\
R_{jl}(\xi)&= \frac{j\sqrt3\xi_l}{2\sqrt{3+5|\xi|^2}}\bigg(\chi_0
+\sqrt{\frac23}\chi_4\bigg) ,\quad l=1,2,3,\\
R_{j4}(\xi)&= \frac{\sqrt{3}|\xi|^2}{3+5|\xi|^2}\bigg(\sqrt{\frac12}\chi_0
+ \sqrt{\frac13}\chi_4\bigg)+\frac{j\sqrt2}{2\sqrt{3+5|\xi|^2}}(v\cdot\xi)\chi_0 ,
\\
R_{j5}(\xi)&= \frac{j\sqrt3}{2\sqrt{3+5|\xi|^2}}v\chi_0 , \quad R_{j6}(\xi)= \frac12v \chi_0 ,
\\
\hat{F}^1_{k}(\xi)&= \(P_0\hat{f}_0,\chi_k\)1_{\{|\xi|\le \frac{r_0}{\eps}\}} ,\quad k=0,1,2,3,4.
\emas

Set
\be\label{a7b}
\left\{\bln
\hat{G}_{j}(t,\xi)&= e^{\frac{ji\sqrt{1+\frac53|\xi|^2}}{\eps}t}\(1+\frac53|\xi|^2\)^{-\frac54} ,\quad j=-1,1,
\\
\hat{H}_{jl}(t,\xi)&= e^{-b_j(|\xi|)t} \(1+\frac53|\xi|^2\)^{-\frac34}R_{jl}(\xi)  ,\,\,\,0\le l\le 6,
\\
\hat{F}^2_{k}(\xi)&= \(1+\frac53|\xi|^2\)^{2}\(P_0\hat{f}_0,\chi_k\)1_{\{|\xi|\le \frac{r_0}{\eps}\}} ,\,\,\,0\le k\le 4.
\eln\right.
\ee
By \eqref{a7a} and \eqref{a7b}, we have
\bma
& e^{ \frac{\eta_j(|\xi|)}{\eps}t-b_j(|\xi|)t}\(P_0\hat{f}_0,h_j(\xi)\)_{\xi}h_j(\xi) \nnm\\
=& \sum^4_{k=0}\hat{G}_j  \hat{H}_{jk}  \hat{F}^2_{k}+ \hat{G}_j \hat{H}_{j5} \frac{\xi}{|\xi|^2}\hat{F}^2_{0}+ \sum^3_{l=1}\hat{G}_j \hat{H}_{j6} \frac{\xi\xi_n}{|\xi|^2} \hat{F}^2_{l}. \label{a7}
\ema
We deal with the terms in the right side of \eqref{a7} as follows. From \cite{Nelson}, one has
\be
\|G_j(t)\|_{L^p_x}\le C\(\frac{t}{\eps}\)^{-\frac32+\frac{3}{p}},\quad 2\le p\le \infty. \label{a3}
\ee
Since $\|\pt^{\alpha}_{\xi}R_{jk}(\xi)\|\le C(1+|\xi|)^{-|\alpha|}$ for any $\alpha\in \N^3$ and $0\le k \le 6$, it follows from \eqref{a7b} and Lemma \ref{a2} that
\be
\|H_{jk}(t)\|_{L^{2}_v(L^1_x)}\le C,\quad 0\le k\le 6. \label{a6}
\ee
By  ellipticity, we have
\be
\| D_{ij}\Delta^{-1}_x g\|_{L^{q}}\le C\|g\|_{L^q},\quad \forall g\in L^q(\R^3),\,\,\, q\in (1,\infty). \label{a8}
\ee
This and \eqref{a7}, \eqref{a3}, \eqref{a6} imply that
\bma
I_{12}\le&  \|G_j(t)\|_{L^{r}_x}\|H_{jk}(t)\|_{L^{2}_v(L^1_x)}\| F^2_{k}\|_{L^{p}_x}+\|G_j(t)\|_{L^{r}_x}\|H_{j5}(t)\|_{L^{2}_v(L^1_x)}\|\Tdx\Delta^{-1}_xF^2_{0}\|_{L^{p}_x}\nnm\\
&+\|G_j(t)\|_{L^{r}_x}\|H_{j6}(t)\|_{L^{2}_v(L^1_x)}\| D^2_x\Delta^{-1}_xF^2_{l}\|_{L^{p}_x}\nnm\\
\le & C\(\frac{t}{\eps}\)^{-\frac32+\frac{3}{r}}\(\|f_0\|_{L^2_{v}(W^{4,p}_x)} +\|\Tdx\phi_0\|_{L^p_x }\), \label{S_4}
\ema
where
$1/p+1/r=1$ with $1<p<2.$ By combining \eqref{S_2}, \eqref{S_3} and \eqref{S_4}, we obtain
\be
I_{1}\le  C\bigg(\eps (1+t)^{-\frac32}+\(1+\frac{t}{\eps}\)^{-q}\bigg)\(\|f_0\|_{L^2_{v}(H^{4}_x)} +\|f_0\|_{L^2_{v}(W^{4,1}_x)} +\|\Tdx\phi_0\|_{L^p_x }\) \label{S_5a}
\ee
for any $1<p<2$ and $q=3/p-3/2$.

By \eqref{v5} and \eqref{B_3}, we have
\bma
I_2&\le  C  \(\int_{|\xi|\ge \frac{r_0}{\eps}}\frac1{|\xi|^4}d\xi\)^{1/2}\(\int_{|\xi|\ge \frac{r_0}{\eps}}e^{-2c|\xi|^2t}|\xi|^4\|P_0\hat f_0\|^2 d\xi\)^{1/2}\nnm\\
&\le Ce^{-\frac{cr_0^2t}{\eps^2}}\|f_0\|_{L^2_{v}(H^{2}_x)},\label{S_5}
\\
I_3&\le  C \(\intr \frac1{(1+|\xi|^2)^2}d\xi\)^{1/2}\(\intr e^{-2\frac{dt}{\eps^2}}(1+|\xi|^2)^2\|\hat{f}_0\|^2_{\xi} d\xi\)^{1/2}\nnm\\
 &\le Ce^{-\frac{dt}{\eps^2}}\(\|f_0\|_{L^2_v(H^2_x)}+\|f_0\|_{L^2_v(L^1_{x})}\). \label{S_6}
\ema
Therefore, it follows from \eqref{S_1} and \eqref{S_5a}--\eqref{S_6} that
\bma
\left\|e^{\frac{t}{\eps^2}B_\eps}f_0-V(t)P_0f_0\right\|_{L^{\infty}_x(L^{2}_v)}
\le& C\bigg(\eps (1+t)^{-\frac32}+\(1+\frac{t}{\eps}\)^{-q}\bigg)\nnm\\
 &\times\(\|f_0\|_{L^2_v(H^4_x)}+\|f_0\|_{L^2_v(W^{4,1}_x)}+\|\Tdx\phi_0\|_{L^p_x}\)
\ema
for any $1<p<2$ and $q=3/p-3/2$.

By a similar argument as above, we can prove
\bma
&\left\|\Tdx\Delta^{-1}_x\(e^{\frac{t}{\eps^2}B_\eps}f_0-V(t)P_0f_0,\chi_0\)\right\|_{L^{\infty}_x}\nnm\\
\le& C\bigg(\eps (1+t)^{-\frac32}+\(1+\frac{t}{\eps}\)^{-q}\bigg)\(\|f_0\|_{L^2_v(H^4_x)}+\|f_0\|_{L^2_v(W^{4,1}_x)}+\|\Tdx\phi_0\|_{L^p_x}\).
\ema

We now turn to \eqref{limit1a}. If $f_0$ satisfies \eqref{initial}, we have
$$
\(P_0\hat{f}_0,h_j(\xi)\)_{\xi}=0,\quad j=-1,1,
$$
which implies that $I_{12}=0$. The term $I_{11}$ satisfies \eqref{S_3}.
It follows from \eqref{v5} that
$$
I_2\le  C \eps \int_{|\xi|\ge \frac{r_0}{\eps}}e^{-c|\xi|^2t}|\xi|\|P_0\hat f_0\| d\xi\le C\eps e^{-\frac{cr_0^2t}{\eps^2}}\|f_0\|_{L^2_{v}(H^{3}_x)}.
$$
By Lemma \ref{S2}, we have
\bmas
I_3&\le C\int_{|\xi|\le \frac{r_0}{\eps}}\eps |\xi|e^{-\frac{dt}{\eps^2}}\| \hat{f}_0\|_{\xi}d\xi+C\int_{|\xi|\ge \frac{r_0}{\eps}}e^{-\frac{dt}{\eps^2}}\| \hat{f}_0\|_{\xi}d\xi\\
&\le C\eps e^{-\frac{dt}{\eps^2}}\(\|f_0\|_{L^2_v(H^3_x)}+\|f_0\|_{L^2_v(L^1_{x})}\).
\emas
Thus we obtain \eqref{limit1a}.
This proves the lemma.
\end{proof}

\begin{remark}
From Lemma \ref{fl-1}, we have
\be
\|e^{\frac{t}{\eps^2}B_{\eps}}P_0f_0-V(t)P_0f_0-u^{osc}_{\eps}(t)\|_{L^{\infty}_P} \le C\eps (1+t)^{-\frac34}\(\|f_0\|_{L^2_v(H^3_x)}+\|f_0\|_{L^2_v(L^{1}_x)}\), \label{limit7}
\ee
where $u^{osc}_{\eps}(t)=u^{osc}_{\eps}(t,x,v)$ is defined by \eqref{osc}.
\end{remark}

\begin{lem}\label{fl-2}
For any $\eps\in(0,1)$ and  any $f_0\in L^{2}_{v}(L^2_x)$ satisfying $P_0f_0=0$, we have
\bma
&\bigg\|\frac1{\eps}e^{\frac{t}{\eps^2}B_\eps}f_0+V(t)P_0(v\cdot\Tdx L^{-1}f_0)\bigg\|_{L^{\infty}_P} \nnm\\
\le&  C\bigg(\eps (1+t)^{-\frac52}+\(1+\frac{t}{\eps}\)^{-p}+\frac1{\eps}e^{-\frac{dt}{\eps^2}}\bigg)\(\|f_0\|_{L^2_v(H^5_x)}+\|f_0\|_{L^2_v(W^{5,1}_x)}\), \label{limit2}
\ema
for any $0<p<3/2$.
\end{lem}

\begin{proof}
By Theorem \ref{E_3}, we have
\bma
&\left\|\frac1{\eps}e^{\frac{t}{\eps^2}B_\eps}f_0+V(t)P_0(v\cdot\Tdx L^{-1}f_0)\right\| \nnm\\
\le &C\left\|\int_{|\xi|\le \frac{r_0}{\eps}}e^{ix\cdot\xi}\(\frac1{\eps}S_1(t,\xi,\eps)\hat{f}_0+V(t,\xi)P_0(iv\cdot\xi L^{-1}\hat{f}_0)\)d\xi\right\|\nnm\\
&+C\int_{|\xi|\ge \frac{r_0}{\eps}} \left\|V(t,\xi)P_0(v\cdot\xi L^{-1}\hat{f}_0)\right\|d\xi+C\intr \left\|\frac1{\eps}S_2(t,\xi,\eps)\hat{f}_0\right\|d\xi\nnm\\
=:&I_1+I_2+I_3. \label{S_1a}
\ema
We estimate $I_j$, $j=1,2,3$ as follows. By \eqref{E_5} and Theorem \ref{spect3}, for any $f_0\in L^{2}_{v}(L^2_x)$ satisfying $P_0f_0=0$, we have
$$
S_1(t,\xi,\eps)\hat{f}_0=i\eps\sum^3_{j=-1}e^{\frac{\eta_j(|\xi|)}{\eps}t-b_j(|\xi|)t+ O(\eps|\xi|^3)t}\[\(v\cdot\xi L^{-1}\hat{f}_0,h_j(\xi)\)h_j(\xi)+O(\eps|\xi|^2)\],
$$
which leads to
\bma
I_1\le &\sum^3_{j=-1}\int_{|\xi|\le \frac{r_0}{\eps}}\bigg\|e^{\frac{\eta_j(|\xi|)}{\eps}t-b_j(|\xi|)t+ O(\eps|\xi|^3)t}\[\(v\cdot\xi L^{-1}\hat{f}_0,h_j(\xi)\)h_j(\xi)+O(\eps|\xi|^2)\]\nnm\\
&\qquad-e^{\frac{\eta_j(|\xi|)}{\eps}t-b_j(|\xi|)t}\(v\cdot\xi L^{-1}\hat{f}_0,h_j(\xi)\)h_j(\xi)\bigg\|d\xi\nnm\\
&+\sum_{j=-1,1}\bigg\|\int_{|\xi|\le \frac{r_0}{\eps}}e^{ix\cdot\xi }e^{ \frac{\eta_j(|\xi|)}{\eps}t-b_j(|\xi|)t}\(v\cdot\xi L^{-1}\hat{f}_0,h_j(\xi)\) h_j(\xi)d\xi\bigg\|\nnm\\
=:&I_{11}+I_{12}. \label{S_2c}
\ema
For $I_{11}$, it holds that
\bma
I_{11}\le&  C\eps\int_{|\xi|\le \frac{r_0}{\eps}}  e^{-c|\xi|^2t}\( |\xi|^3t\|P_0(v\cdot\xi L^{-1}\hat{f}_0)\|
+|\xi|^2\|\hat{f}_0\|\) d\xi\nnm\\
\le& C\eps(1+t)^{-\frac52} \(\| f_0\|_{L^2_{v}(H^{4}_x)}+\| f_0\|_{L^{2}_v(L^1_x)}\). \label{S_3a}
\ema
Then, we estimate $I_{12}$ as follows. Let
$$
\hat{U}_k(\xi)=\(1+\frac53|\xi|^2\)^2(v\cdot\xi L^{-1}\hat{f}_0,\chi_k)1_{\{|\xi|\le \frac{r_0}{\eps}\}},\quad k=1,2,3,4.
$$
By \eqref{a7}, we have
\be
 e^{ \frac{\eta_j(|\xi|)}{\eps}t-b_j(|\xi|)t}\(v\cdot\xi L^{-1}\hat{f}_0,h_j(\xi)\) h_j(\xi)
= \sum^4_{k=1}\hat{G}_j  \hat{H}_{jk}  \hat{U}_k+ \sum^3_{l=1}\hat{G}_j \hat{H}_{j6} \frac{\xi\xi_l}{|\xi|^2} \hat{U}_l, \label{b1}
\ee
where $\hat{G}_j(t,\xi)$, $\hat{H}_{jk}(t,\xi)$ are defined by \eqref{a7b}.
It follows from \eqref{b1}, \eqref{a3}, \eqref{a6} and \eqref{a8} that
\bma
I_{12}\le& \|G_j(t)\|_{L^r_x}\|H_{jk}(t)\|_{L^{2}_v(L^1_x)} \| U_k\|_{L^{q}_x}+\|G_j(t)\|_{L^{r}_x}\|H_{j6}(t)\|_{L^{2}_v(L^1_x)}\|D^2_x\Delta^{-1}_x U_l\|_{L^{q}_x}\nnm\\
\le & C\(\frac{t}{\eps}\)^{-\frac32+\frac{3}{r}}\|f_0\|_{L^2_v(W^{5,q}_x)}, \label{S_4a}
\ema
where $1/r+1/q=1$ with $1<q<2.$ By \eqref{S_2c}, \eqref{S_3a} and \eqref{S_4a}, we obtain
\be
I_{1}\le  C\bigg(\eps (1+t)^{-\frac52}+\(1+\frac{t}{\eps}\)^{-p}\bigg)\(\|f_0\|_{L^2_{v}(H^{5}_x)} +\|f_0\|_{L^2_{v}(W^{5,1}_x)} \) \label{S_5c}
\ee
for any $0<p<3/2$.

By \eqref{B_3} and \eqref{v5}, we have
\bma
I_2%&\le C \int_{|\xi|\ge \frac{r_0}{\eps}}e^{-c|\xi|^2t}|\xi|\|P_1\hat{f}_0\| d\xi\nnm\\
&\le  C  \(\int_{|\xi|\ge \frac{r_0}{\eps}}\frac1{|\xi|^4}d\xi\)^{1/2}\(\int_{|\xi|\ge \frac{r_0}{\eps}}e^{-2c|\xi|^2t}|\xi|^6\|P_1\hat f_0\|^2 d\xi\)^{1/2}\nnm\\
&\le Ce^{-\frac{cr_0^2t}{\eps^2}}\|f_0\|_{L^2_v(H^3_x)},
\\
I_3 &\le  C \(\intr \frac1{(1+|\xi|^2)^2}d\xi\)^{1/2}\(\intr\frac1{\eps^2} e^{-2\frac{dt}{\eps^2}}(1+|\xi|^2)^2\|\hat{f}_0\|^2_{\xi} d\xi\)^{1/2}\nnm\\
%&\le C\(\intr\frac1{\eps^2} e^{-2\frac{dt}{\eps^2}}(1+|\xi|^2)^2\|\hat{f}_0\|^2_{\xi} d\xi\)^{1/2} \nnm\\
&\le C\frac1{\eps}e^{-\frac{dt}{\eps^2}}\(\|f_0\|_{L^2_v(H^2_x)}+\|f_0\|_{L^{2}_v(L^1_x)}\).\label{S_6a}
\ema
Therefore, it follows from \eqref{S_1a} and \eqref{S_5c}--\eqref{S_6a} that
\bma
&\bigg\|\frac1{\eps}e^{\frac{t}{\eps^2}B_\eps}f_0+V(t)P_0(v\cdot\Tdx L^{-1}f_0)\bigg\|_{L^{\infty}_x(L^{2}_v)}\nnm\\
\le&  C\bigg(\eps (1+t)^{-\frac52}+\(1+\frac{t}{\eps}\)^{-p}+\frac1{\eps}e^{-\frac{dt}{\eps^2}}\bigg)\(\|f_0\|_{L^2_v(H^5_x)}+\|f_0\|_{L^2_v(W^{5,1}_x)} \)
\ema
for any $0<p<3/2$.

By a similar argument as above, we can prove
\bmas
&\left\|\Tdx \Delta^{-1}_x\(\frac1{\eps}e^{\frac{t}{\eps^2}B_\eps}f_0+V(t)P_0(v\cdot\Tdx L^{-1}f_0),\chi_0\)\right\|_{L^{\infty}_x}\\
%\le &\left|\int_{|\xi|\le \frac{r_0}{\eps}}e^{ix\cdot\xi}\frac{\xi}{|\xi|^2}\(\frac1{\eps}S_1(t,\xi,\eps)\hat{f}_0+V(t,\xi)P_0(iv\cdot\xi L^{-1}\hat{f}_0),\chi_0\)d\xi\right|\\
%&+\int_{|\xi|\ge \frac{r_0}{\eps}} \left|\frac{\xi}{|\xi|^2}(V(t,\xi)P_0(iv\cdot\xi L^{-1}\hat{f}_0),\chi_0)\right|d\xi\\
%&+\intr \left|\frac{\xi}{\eps|\xi|^2}(S_2(t,\xi,\eps)\hat{f}_0,\chi_0)\right|d\xi\\
\le &C\bigg(\eps (1+t)^{-\frac52}+\(1+\frac{t}{\eps}\)^{-p}+\frac1{\eps}e^{-\frac{dt}{\eps^2}}\bigg)\(\|f_0\|_{L^2_v(H^5_x)}+\|f_0\|_{L^2_v(W^{5,1}_x)} \).
\emas
This proves the lemma.
\end{proof}

\section{Diffusion limit}\setcounter{equation}{0}
\label{sect4}
In this section, we study the diffusion limit of the solution to the nonlinear VPB system \eqref{VPB1}--\eqref{VPB2i} based on the fluid approximations of the semigroup given in Section 3.

Since the operator $B_{\eps}$ generates a contraction semigroup in $H^k_P$ $(k\ge 0)$, the solution $f_{\eps}(t)=f_{\eps}(t,x,v)$ to the VPB system \eqref{VPB1}--\eqref{VPB2i} can be represented by
\be
f_{\eps}(t)=e^{\frac{t}{\eps^2}B_{\eps}}f_0+\intt e^{\frac{t-s}{\eps^2}B_{\eps}}\[G_1(f_{\eps})+\frac1{\eps}G_2(f_{\eps})\](s)ds, \label{fe}
\ee
where the nonlinear terms $G_1(f_{\eps})$ and $G_2(f_{\eps})$ are defined in \eqref{gamma}.

Let $(n,m,q)(t,x)$ be the global solution to the  Navier-Stokes-Poisson-Fourier system \eqref{NSP_1}--\eqref{compatible}. Then by Lemma \ref{sem},  $u(t)=u(t,x,v)=n(t,x)\chi_0+m(t,x)\cdot v\chi_0+q(t,x)\chi_4$  can be represented by
\be
 u(t)=V(t)P_0f_0+\intt V(t-s)\[Z_1(u)+\divx Z_2(u)\](s)ds, \label{ue}
\ee
where
\bma
Z_1(u)&=(n\Tdx\phi)\cdot v\chi_0+\sqrt{\frac23}(m\cdot\Tdx\phi)\chi_4, \label{Z_1}\\
Z_2(u)&=- (m\otimes m)\cdot v\chi_0-\frac53 (qm)\chi_4. \label{Z_2}
\ema

Firstly, we establish the existence and energy estimates of the solutions $f_{\eps}(t)$ and $u(t)$ as follows.
By taking the inner product between $\{\chi_0$, $v\chi_0$, $\chi_4\}$ and \eqref{VPB1}, we obtain
\bma
&\dt n_{\eps}+\frac{1}{\eps}\divx m_{\eps}=0, \label{G_3a}\\
&\dt m_{\eps}+\frac{1}{\eps}\Tdx n_{\eps}+\frac{1}{\eps}\sqrt{\frac23}\Tdx q_{\eps}-\frac{1}{\eps}\Tdx \phi_{\eps}= n_{\eps}\Tdx\phi_{\eps}-\frac{1}{\eps}( v\cdot\Tdx P_1f_{\eps}, v\chi_0), \label{G_3b}\\
&\dt q_{\eps}+\frac{1}{\eps}\sqrt{\frac23}\divx m_{\eps}= \sqrt{\frac{2}{3}}m_{\eps}\cdot\Tdx\phi_{\eps}-\frac{1}{\eps}( v\cdot\Tdx P_1f_{\eps}, \chi_4), \label{G_3c}
\ema
where
$$n_{\eps}=(f_{\eps},\chi_0), \quad m_{\eps}=(f_{\eps},v\chi_0),\quad q_{\eps}=(f_{\eps},\chi_4).$$
Taking the microscopic projection $ P_1$ to \eqref{VPB1}, we have
\bma
&\dt( P_1f_{\eps})+ \frac{1}{\eps}P_1(v\cdot\Tdx  P_1f_{\eps})-\frac{1}{\eps^2}L( P_1f_{\eps})\nnm\\
=&- \frac{1}{\eps}P_1(v\cdot\Tdx  P_0f_{\eps})+ P_1 G_1(f_{\eps})+\frac{1}{\eps}P_1 G_2(f_{\eps}).\label{GG2}
\ema
By \eqref{GG2}, we can express the microscopic part $ P_1f_{\eps}$ as
\bma   P_1f_{\eps}=& \eps L^{-1}[\eps\dt(P_1f_{\eps})+ P_1(v\cdot\Tdx  P_1f_{\eps})- \eps P_1 G_1-P_1 G_2 ]\nnm\\
&-\eps L^{-1}P_1(v\cdot\Tdx  P_0f_{\eps}). \label{p_c}
\ema
Substituting \eqref{p_c} into \eqref{G_3a}--\eqref{G_3c}, we obtain
\bma
&\dt n_{\eps}+\frac{1}{\eps}\divx m_{\eps}=0, \label{G_4a}\\
&\dt m_{\eps}+\eps\dt R_1 +\frac{1}{\eps}\Tdx n_{\eps}+\frac{1}{\eps}\sqrt{\frac23}\Tdx q_{\eps}-\frac{1}{\eps}\Tdx \phi_{\eps}\nnm\\
&\quad=\kappa_0 \(\Delta_xm_{\eps}+\frac13\Tdx{\rm div}_xm_{\eps}\)+ n_{\eps}\Tdx\phi_{\eps}+R_3, \label{G_4b}\\
&\dt q_{\eps}+\eps\dt R_2 +\frac{1}{\eps}\sqrt{\frac23}\divx m_{\eps}=\kappa_1 \Delta_x q_{\eps}+ m_{\eps}\cdot\Tdx\phi_{\eps}-R_4, \label{G_4c}
\ema
where the coefficients $\kappa_0,\kappa_1>0$ are defined by \eqref{coe} and the remainder terms
 \bmas
 %\kappa_0&=-(L^{-1} P_1(v_1\chi_2),v_1\chi_2),\quad \kappa_1=-(L^{-1} P_1(v_1\chi_4),v_1\chi_4),\\
 R_1&=\( v\cdot\Tdx L^{-1}( P_1f_{\eps}),v\chi_0\),\quad R_2=\( v\cdot\Tdx L^{-1}( P_1f_{\eps}),\chi_4\),\\
 R_3&=\( v\cdot\Tdx L^{-1}[ P_1(v\cdot\Tdx  P_1f_{\eps})-\eps P_1 G_1-P_1G_2],v\chi_0\),\\
 R_4&=\( v\cdot\Tdx L^{-1}[ P_1(v\cdot\Tdx  P_1f_{\eps})-\eps P_1 G_1-P_1G_2],\chi_4\).
 \emas
 Let $N\ge 4$. For any $k\ge 0$ and any $f_{\eps}\in L^{2}(\R^3_x\times \R^3_v)$, we define
 \bmas
 E_k(f_{\eps})&=\sum_{|\alpha|+|\beta|\le N}\|\nu^{k}\dxa\dvb f_{\eps}\|^2_{L^{2}_{x,v}}+\sum_{|\alpha|\le N}\|\dxa \Tdx\phi_{\eps}\|^2_{L^2_x},\\
 H_k(f_{\eps})&=\sum_{|\alpha|+|\beta|\le N} \| \nu^{k}\dxa\dvb P_1f_{\eps}\|^2_{L^{2}_{x,v}}+\sum_{|\alpha|\le N-1}\(\|\dxa\Tdx P_0f_{\eps}\|^2_{L^{2}_{x,v}}+\|\dxa n_{\eps}\|^2_{L^2_x}\),\\
 D_k(f_{\eps})&=\sum_{|\alpha|+|\beta|\le N}\frac{1}{\eps^2}\|\nu^{k+1/2}\dxa\dvb P_1f_{\eps}\|^2_{L^{2}_{x,v}}\\
 &\quad+\sum_{|\alpha|\le N-1}\(\|\dxa\Tdx P_0f_{\eps}\|^2_{L^{2}_{x,v}}+\|\dxa n_{\eps}\|^2_{L^2_x}\),
 \emas
 where $\phi_{\eps}=\Delta^{-1}_x (f_{\eps},\chi_0)$. For simplicity, we denote $E(f_{\eps})=E_0(f_{\eps})$, $H(f_{\eps})=H_0(f_{\eps})$ and $D(f_{\eps})=D_0(f_{\eps})$.

First, we have the existence and the energy estimate for the solution $f_{\eps}$ to the VPB system \eqref{VPB1}--\eqref{VPB2i}.

\begin{lem}\label{energy1}For any $\eps\in (0,1)$, there exists a small constant $\delta_0>0$ such that if $E(f_0)\le \delta_0^2$, then the  system \eqref{VPB1}--\eqref{VPB2i} admits a unique global solution  $f_{\eps}(t)= f_{\eps}(t,x,v)$ satisfying the following  energy estimate:
\be
   E(f_{\eps}(t))+ d_1\intt  D(f_{\eps}(s))ds\le C\delta_0^2, \label{G_1}
\ee
where $d_1,C>0$ are two constants independent of $\eps$.
\end{lem}

\begin{proof}

First, we establish the macroscopic energy estimate of the solution $f_{\eps}$. Taking the inner product between $\dxa m_{\eps}$ and $\dxa\eqref{G_4b}$ with $|\alpha|\le N-1$, we have
\bma
 &\quad\frac12\Dt \( \|\dxa m_{\eps}\|^2_{L^2_x} + \|\dxa n_{\eps} \|^2_{L^2_x} + \|\dxa\Tdx\phi_{\eps}\|^2_{L^2_x}\)
 +\eps\Dt \intr \dxa R_1\dxa m_{\eps}dx
\nnm\\
&\quad+\frac{1}{\eps}\sqrt{\frac23}\intr \dxa\Tdx q_{\eps}\dxa m_{\eps}dx+\kappa_0\(\|\dxa\Tdx  m_{\eps}\|^2_{L^2_x}+\frac13\|\dxa\divx  m_{\eps}\|^2_{L^2_x}\)
\nnm\\
&=\intr \dxa(n_{\eps}\Tdx \phi_{\eps}) \dxa m_{\eps}dx+\intr\dxa R_3\dxa m_{\eps}dx+\eps\intr\dxa R_1\dt\dxa m_{\eps}dx\nnm\\
&=:I_1+I_2+I_3 .\label{en_1}
   \ema

   From \cite{Duan2,Guo2}, for any $\beta\in \mathbb{N}^3$ and $k\in[0,1]$ one has
\be
\|\nu^{-k}\dvb \Gamma(f,g)\|_{L^2_{v }}\le C\sum_{\beta_1+\beta_2\le  \beta}\(\|\nu^{1-k}\dv^{\beta_1}f\|_{L^2_{v }}\|\dv^{\beta_2}g\|_{L^2_{v }}+\|\dv^{\beta_1}f\|_{L^2_{v }}\|\nu^{1-k}\dv^{\beta_2}g\|_{L^2_{v }}\). \label{p2}
\ee
By Cauchy inequality and \eqref{p2}, we have
\bma
I_1&\le C \sqrt{E(f_{\eps})}D(f_{\eps})  , \label{en_2}\\
I_2&\le\frac{C}{\delta}\|\dxa\Tdx P_1f_{\eps}\|^2_{L^{2}_{x,v}}+\delta\|\dxa\Tdx m_{\eps}\|^2_{L^2_x} +C\sqrt{E(f_{\eps})}D(f_{\eps}), \label{en_3}
\ema
where $\delta>0$ is a small constant to be determined later.
By \eqref{G_3b}, we can bound $I_3$ by
\bma
I_3&\le \frac{C}{\delta}\(\|\dxa\Tdx P_1f_{\eps}\|^2_{L^{2}_{x,v}}+\|\dxa P_1f_{\eps}\|^2_{L^{2}_{x,v}}\)+C\eps \sqrt{E(f_{\eps})}D(f_{\eps})\nnm\\
&\quad+\delta\(\|\dxa\Tdx n_{\eps}\|^2_{L^2_x} +\|\dxa n_{\eps}\|^2_{L^2_x} +\|\dxa\Tdx q_{\eps}\|^2_{L^2_x}\).\label{en_4}
\ema

Therefore, it follows from \eqref{en_1}, \eqref{en_2}, \eqref{en_3} and \eqref{en_4} that
\bma &\frac12\Dt \(\| (n_{\eps}, m_{\eps})\|^2_{H^{N-1}_x}+\| \Tdx\phi_{\eps}\|^2_{H^{N-1}_x}\)+\eps\Dt\sum_{|\alpha|\le N-1} \intr \dxa R_1\dxa m_{\eps}dx\nnm\\
&\quad+\sqrt{\frac23}\sum_{|\alpha|\le N-1}\intr \dxa\Tdx q_{\eps}\dxa m_{\eps}dx+\frac{\kappa_0}2 \(\| \Tdx  m_{\eps}\|^2_{H^{N-1}_x}+\frac13\|\divx  m_{\eps}\|^2_{H^{N-1}_x}\)\nnm\\
&\le C \sqrt{E(f_{\eps})}D(f_{\eps})
+\frac{C}{\delta}\| \Tdx P_1f_{\eps}\|^2_{L^2_v(H^N_x)}+\delta\(\|\Tdx n_{\eps}\|^2_{H^N_x}+\|\Tdx q_{\eps}\|^2_{H^{N-1}_x}\).\label{m_1}
\ema

Similarly, taking the inner product between $\dxa q_{\eps}$ and $\dxa\eqref{G_4c}$ with $|\alpha|\le N-1$, we have
\bma
&\frac12\Dt \|\dxa q_{\eps}\|^2_{L^2_x}+\eps\Dt \intr \dxa R_2\dxa q_{\eps}dx
 + \frac{\kappa_1}2 \|\dxa\Tdx q\|^2_{L^2_x}\nnm\\
 &\quad+\sqrt{\frac23}\intr \dxa\divx m_{\eps}\dxa q_{\eps}dx\nnm\\
 &=\sqrt{\frac23}\intr \dxa(m_{\eps}\cdot\Tdx \phi_{\eps}) \dxa q_{\eps}dx+\intr\dxa R_4\dxa q_{\eps}dx+\eps\intr\dxa R_3\dt\dxa q_{\eps}dx \nnm\\
&=:I_4+I_5+I_6.\label{q_1a}
\ema
We estimate $I_j$, $j=4,5,6$ as follows. For $|\alpha|\ge 1$, it holds that
\be
I_4\le C \sqrt{E(f_{\eps})}D(f_{\eps}),\label{en_4a}
\ee
and for $|\alpha|=0,$ it follows from \eqref{G_3b} and \eqref{G_3c} that
\bma
I_4&=\sqrt{\frac23}\intr  \bigg[m_{\eps}\cdot\bigg(\eps\dt m_{\eps}+\Tdx n_{\eps}+\sqrt{\frac23}\Tdx q_{\eps}-\eps n_{\eps}\Tdx\phi_{\eps}\nnm\\
&\qquad\qquad-(v\cdot\Tdx P_1f_{\eps},v\chi_0)\bigg)\bigg]  q_{\eps}dx\nnm\\
&\le \eps\sqrt{\frac16}\Dt\intr  (m_{\eps})^2 q_{\eps}dx+C\(\sqrt{E(f_{\eps})}+ E(f_{\eps})\)D(f_{\eps}).\label{en_5}
\ema

By Cauchy inequality and \eqref{p2}, we have
\be
I_5\le\frac{C}{\delta}\|\dxa\Tdx P_1f_{\eps}\|^2_{L^{2}_{x,v}}+ \delta\|\dxa\Tdx q_{\eps}\|^2_{L^2_x}+C\sqrt{E(f_{\eps})}D(f_{\eps}) . \label{en_6}
\ee
By \eqref{G_3c}, we can bound $I_6$ by
\bma
I_6&\le \frac{C}{\delta}\(\|\dxa\Tdx P_1f_{\eps}\|^2_{L^{2}_{x,v}}+\|\dxa P_1f_{\eps}\|^2_{L^{2}_{x,v}}\)+C\eps \sqrt{E(f_{\eps})}D(f_{\eps})\nnm\\
&\quad+\delta\(\|\dxa\Tdx n_{\eps}\|^2_{L^2_x} +\|\dxa n_{\eps}\|^2_{L^2_x} +\|\dxa\Tdx q_{\eps}\|^2_{L^2_x}\).\label{en_7}
\ema
Thus, it follows from \eqref{q_1a}, \eqref{en_4a}--\eqref{en_7} that
\bma
&\quad\frac12\Dt \| q_{\eps}\|^2_{H^{N-1}_x}+\eps\Dt \sum_{|\alpha|\le N-1}\intr \dxa R_2\dxa q_{\eps}dx-\eps\sqrt{\frac16}\Dt\intr  (m_{\eps})^2 q_{\eps}dx\nnm\\
 &\quad+\sqrt{\frac23}\sum_{|\alpha|\le N-1}\intr \dxa\divx m_{\eps}\dxa q_{\eps}dx+ \frac12\kappa_1 \| \Tdx q_{\eps}\|^2_{H^{N-1}_x}\nnm\\
&\le C\(\sqrt{E(f_{\eps})}+ E(f_{\eps})\)D(f_{\eps})+C\| \Tdx P_1f_{\eps}\|^2_{L^2_v(H^{N-1}_x)}+\delta\|\Tdx m_{\eps}\|^2_{H^{N-1}_x}.\label{q_1}
\ema
Agian, taking the inner product between $\dxa\Tdx n_{\eps}$ and $\dxa\eqref{G_3b}$ with $|\alpha|\le N-1$, one has
 \bma
&\eps\Dt\intr \dxa m_{\eps} \dxa\Tdx n_{\eps}dx
 +  \frac12\|\dxa\Tdx n_{\eps}\|^2_{L^2_x}+\|\dxa n_{\eps}\|^2_{L^2_x}
  \nnm\\
\le&
 C \sqrt{E(f_{\eps})}D(f_{\eps})+ \|\dxa\divx  m_{\eps}\|^2_{L^2_x}+\|\dxa \Tdx q_{\eps}\|^2_{L^2_x} + C\|\dxa\Tdx (P_1f_{\eps})\|^2_{L^{2}_{x,v}}.    \label{abc}
  \ema
Thus, we take the summation of $\eqref{m_1}+\eqref{q_1}+C_0\eqref{abc}$ with $C_0>0 $  small enough to obtain the  macroscopic energy estimate:
\bma
&\quad\frac12\Dt \(\| P_0f_{\eps}\|^2_{L^2_v(H^{N-1}_x)}+\|\Tdx\phi_{\eps}\|^2_{H^{N-1}_x} \)+\eps\sqrt{\frac16}\Dt\intr  (m_{\eps})^2 q_{\eps}dx\nnm\\
&\quad+\eps\Dt \sum_{|\alpha|\le N-1}\(\intr \dxa R_1\dxa m_{\eps}dx+\intr \dxa R_2\dxa q_{\eps}dx+C_0\intr \dxa m_{\eps} \dxa\Tdx n_{\eps}dx\)\nnm\\
 &\quad+  \frac{C_0}{2} \(\| \Tdx P_0f_{\eps}\|^2_{L^2_v(H^{N-1}_x)}+\|n_{\eps}\|^2_{H^{N-1}_x}\)\nnm\\
&\le C\(\sqrt{E(f_{\eps})}+ E(f_{\eps})\)D(f_{\eps}) +C\| P_1f_{\eps}\|^2_{L^2_v(H^N_x)}.\label{macro}
\ema

Next, we deal with the microscopic energy estimate. Taking inner product between $\dxa  f_{\eps}$ and $\dxa\eqref{VPB1}$ with $|\alpha|\le 4$, we have
\bma
&\quad\frac12\Dt \(\|\dxa f_{\eps}\|^2_{L^{2}_{x,v}}+\|\dxa\Tdx\phi_{\eps}\|^2_{L^2_x}\)-\frac{1}{\eps^2}\intr (L\dxa f_{\eps})\dxa f_{\eps}dxdv
\nnm\\
&=\frac12\intrr  \dxa (v\cdot\Tdx \phi_{\eps} f_{\eps})\dxa f_{\eps}dxdv-\intrr \dxa (\Tdx \phi_{\eps}\cdot \Tdv f_{\eps})\dxa f_{\eps}dxdv\nnm\\
&\quad+\frac{1}{\eps}\intrr  \dxa \Gamma( f_{\eps},f_{\eps})\dxa P_1f_{\eps}dxdv\nnm\\
&= :I_7+I_8+I_9. \label{G_0}
\ema
We estimate $I_j$, $j=7,8,9$ as follows. For $|\alpha|=0$, it holds that
\be I_8=0,\quad I_9\le C\sqrt{E(f_{\eps})}D(f_{\eps}). \label{b2}\ee
To estimate $I_7$ with $|\alpha|=0$, we decompose
\bmas
I_7&=\frac12\intrr   (v\cdot\Tdx \phi_{\eps} P_0f_{\eps})P_0 f_{\eps}dxdv+\frac12\intrr   (v\cdot\Tdx \phi_{\eps} P_1f_{\eps})P_0 f_{\eps}dxdv\\
&\quad+\frac12\intrr   (v\cdot\Tdx \phi_{\eps} P_1f_{\eps})P_1 f_{\eps}dxdv\\
&=:I_{71}+I_{72}+I_{73}.
\emas
It holds that
$
I_{72},I_{73}\le C\sqrt{E(f_{\eps})}D(f_{\eps}),
$
and by \eqref{en_5},
\bma
I_{71}&=\intr (n_{\eps}\cdot\Tdx\phi_{\eps})m_{\eps}dx+\sqrt{\frac{2}{3}}\intr (m_{\eps}\cdot\Tdx\phi_{\eps})q_{\eps}dx \nnm\\
&\le \eps\sqrt{\frac16}\Dt\intr  (m_{\eps})^2 q_{\eps}dx+C\(\sqrt{E(f_{\eps})}+ E(f_{\eps})\)D(f_{\eps}).\label{b3}
\ema

For $|\alpha|\ge 1$, we have
\be
I_{7},I_{8},I_{9}\le C\sqrt{E(f_{\eps})}D(f_{\eps}). \label{b4}
\ee
Thus, it follows from \eqref{G_0}--\eqref{b4} that
\bma
 &\Dt \(\|f_{\eps}\|^2_{L^2_v(H^N_x)}+\|\Tdx\phi_{\eps}\|^2_{H^N_x}\)+\eps\sqrt{\frac23}\Dt\intr  (m_{\eps})^2 q_{\eps}dx+\frac{\mu_1}{\eps^2}\|\nu^{1/2} P_1f_{\eps}\|^2_{L^2_v(H^N_x)}\nnm\\
\le& C\(\sqrt{E(f_{\eps})}+ E(f_{\eps})\)D(f_{\eps}).\label{G_00}
\ema

Finally, we  estimate the mixed derivative terms $\dxa\dvb P_1f_{\eps}$ with $|\alpha|+|\beta|\le N$. For $|\beta|=0$, we take the inner product between $P_1 f_{\eps}$ and \eqref{VPB1} and using Cauchy inequality to obtain
\be
\Dt \| P_1f_{\eps}\|^2_{L^{2}_{x,v}}+\frac{\mu_1}{\eps^2}\|\nu^{1/2
}P_1 f_{\eps}\|^2_{L^{2}_{x,v}}\le C\|\Tdx P_0 f_{\eps}\|^2_{L^{2}_{x,v}}+ C\sqrt{E(f_{\eps})}D(f_{\eps}).\label{E_4}
\ee
To estimate  $\dxa\dvb P_1f_{\eps}$ with $|\beta|\ge 1$, we rewrite \eqref{VPB1} as
 \bma
&\partial_t(P_1f_{\eps})+\frac{1}{\eps}v\cdot\Tdx P_1f_{\eps} -\frac{1}{\eps^2}\nu(v)P_1f_{\eps}\nnm\\
=&-\frac{1}{\eps^2}K(P_1f_{\eps})-\frac{1}{\eps}P_0(v\cdot\Tdx P_1f_{\eps})-\frac{1}{\eps}P_1(v\cdot\Tdx P_0f_{\eps})\nnm\\
&+\frac{1}{\eps}\Gamma(f_{\eps},f_{\eps})+P_1\(\frac12v\cdot\Tdx\phi_{\eps}  f_{\eps}-\Tdx\phi_{\eps}\cdot\Tdv f_{\eps}\).\label{G_2}
\ema

Let $1\le k\le N$, and choose $\alpha,\beta $ with $|\beta|=k$ and $|\alpha|+|\beta|\le N$. Taking the inner product between $\dxa\dvb P_1 f_{\eps}$ and $\dxa\dvb\eqref{G_2}$ and summing the resulted equations,   we obtain  after a tedious calculation
\bma
& \sum_{|\beta|=k \atop |\alpha|+|\beta|\le N}\Dt\|\dxa\dvb P_1f_{\eps}\|^2_{L^{2}_{x,v}}
+\frac{\mu_1}{\eps^2}\sum_{|\beta|=k \atop |\alpha|+|\beta|\le N}\|\nu^{1/2}\dxa\dvb P_1f_{\eps}\|^2_{L^{2}_{x,v}}\nnm\\
&\le C\sum_{|\alpha|\le 4-k}\(\|\dxa\Tdx P_0f_{\eps}\|^2_{L^{2}_{x,v}}+\|\dxa\Tdx P_1f_{\eps}\|^2_{L^{2}_{x,v}}\)+\frac{C_k}{\eps^2}\sum_{ |\beta|\le k-1\atop |\alpha|+|\beta|\le N}\|\dxa\dvb P_1f_{\eps}\|^2_{L^{2}_{x,v}}\nnm\\
&\quad+C\sqrt{E(f_{\eps})}D(f_{\eps}).\label{aa}
\ema
Taking the summation $\sum_{1\le k\le N}p_k\eqref{aa}$ with
$$\mu_1 p_k\ge 2\sum_{1\le j\le N-k}p_{k+j}C_{k+j},\,\,\, 1\le k\le N-1,\quad p_N=1,$$
we obtain
 \bma
&\Dt\sum_{1\le k\le N}p_k\sum_{|\beta|=k \atop |\alpha|+|\beta|\le N} \|\dxa\dvb P_1f_{\eps}\|^2_{L^{2}_{x,v}}
+\frac{\mu_1}{2\eps^2}\sum_{1\le k\le N}p_k\sum_{|\beta|=k \atop |\alpha|+|\beta|\le N}\|\nu^{1/2}\dxa\dvb P_1f_{\eps}\|^2_{L^{2}_{x,v}}\nnm\\
&\le C\sum_{|\alpha|\le N-1}\|\dxa\Tdx P_0f_{\eps}\|^2_{L^{2}_{x,v}}+C\sum_{|\alpha|\le N}\|\dxa P_1f_{\eps}\|^2_{L^{2}_{x,v}}+C\sqrt{E(f_{\eps})}D(f_{\eps}). \label{E_5b}
\ema

Assume that $E(f_{\eps})\le C\delta_0^2$. Taking the summation of $A_1\eqref{macro}+A_2\eqref{G_00}+\eqref{E_5b}$ with $A_2>C_0A_1>0$ being large enough and integrate in time $t$, we prove \eqref{G_1}.  Based on the uniform energy estimates \eqref{G_1}, we can prove the existence of the global solution by the standard continuity argument. The detail are omitted for brevity.
This completes the proof of the lemma.
\end{proof}

By a similar argument as for Lemma \ref{energy1}, we have the following result with the detailed proof  omitted.

\begin{lem}\label{energy2}For any $\eps\in (0,1)$, there exists a small constant $\delta_0>0$ such that if $E_1(f_0)\le \delta_0^2$, then
there are two energy functionals $\mathcal{E}_1(f_{\eps}(t))\sim E_1(f_{\eps}(t))$ and $\mathcal{H}_1(f_{\eps}(t))\sim H_1(f_{\eps}(t))$ such that
\bma
   \Dt \mathcal{E}_1(f_{\eps}(t))+  D_1(f_{\eps}(t))&\le 0, \label{G_1b}\\
   \Dt \mathcal{H}_1(f_{\eps}(t))+  D_1(f_{\eps}(t))&\le C\|\Tdx P_0f_{\eps}(t)\|^2_{L^{2}_{x,v}}, \label{G_1c}
\ema
where $C>0$ is a constant independent of $\eps$.
\end{lem}

With the help of Lemmas \ref{time}, \ref{energy1} and \ref{energy2}, we have the time decay rate of $f_{\eps}$ as follows.

\begin{lem}\label{time7}For any $\eps\in(0,1)$, there exists a small constant $\delta_0>0$ such that if $E_1(f_0)+\|f_0\|^2_{L^{2}_v(L^1_x)}\le \delta_0^2$, then the solution
   $f_{\eps}(t)=f_{\eps}(t,x,v)$ to the system \eqref{VPB1}--\eqref{VPB2i} has the following time-decay rate estimates:
\bma
   E_1(f_{\eps}(t)) \le C\delta_0^2 (1+t)^{-\frac12}, \label{G_11}\\
   H_1(f_{\eps}(t)) \le C\delta_0^2 (1+t)^{-\frac32}, \label{G_12}
\ema
where $C>0$ is a constant independent of $\eps$. In particular, we have
\be
\|P_1f_{\eps}(t)\|_{L^2_v(H^{N-3}_x)}\le C\delta_0\(\eps(1+t)^{-\frac34}+e^{-\frac{dt}{4\eps^2}}\),\label{G_5}
\ee
where $d,C>0$ are two constants independent of $\eps$.

Moreover, if $f_0$ satisfies \eqref{initial} and $E_1(f_0)+\|f_0\|^2_{L^{2}_v(L^1_x)}\le \delta_0^2$, then
\bma
   E_1(f_{\eps}(t)) \le C\delta_0^2 (1+t)^{-\frac32}, \label{G_11a}\\
   H_1(f_{\eps}(t)) \le C\delta_0^2 (1+t)^{-2}. \label{G_12a}
\ema
\end{lem}
\begin{proof}\iffalse Recalling \eqref{fe}, we have
\be
f_{\eps}(t)=e^{\frac{t}{\eps^2}B_{\eps}}f_0+\intt e^{\frac{t-s}{\eps^2}B_{\eps}}\[G_1(f_{\eps})+\frac1{\eps}G_2(f_{\eps})\](s)ds. \label{fe1}
\ee \fi
Define
$$
Q_{\eps}(t)=\sup_{0\le s\le t}\left\{(1+s)^{1/4}E_1(f_{\eps}(s))^{1/2}+(1+s)^{3/4}H_1(f_{\eps}(s))^{1/2}\right\}.
$$
We claim that
  \be
 Q_{\eps}(t)\le C\delta_0.  \label{assume}
 \ee
It is straightforward  to verify that the estimates  \eqref{G_11} and \eqref{G_12} follow  from \eqref{assume}.

By \eqref{p2} and \eqref{gamma}, we have
\bma
\| G_j(f_{\eps}(s))\|_{L^{2}_{x,v}}&\le CE_1(f_{\eps})^{1/2}H_1(f_{\eps})^{1/2}\le CQ_{\eps}(t)^2 (1+s)^{-1}  ,  \label{h1}\\
\| G_j(f_{\eps}(s))\|_{L^{2}_v(L^1_x)}&\le CE_1(f_{\eps})^{1/2}E_1(f_{\eps})^{1/2}
\le CQ_{\eps}(t)^2 (1+s)^{-\frac12} ,  \label{h2}
\ema
and
\bma
\|\dxa G_j(f_{\eps}(s))\|_{L^{2}_{x,v}}&\le CH_1(f_{\eps})^{1/2}H_1(f_{\eps})^{1/2}\le CQ_{\eps}(t)^2 (1+s)^{-\frac32}  ,  \label{h3}\\
\|\dxa G_j(f_{\eps}(s))\|_{L^{2}_v(L^1_x)}&\le CE_1(f_{\eps})^{1/2}H_1(f_{\eps})^{1/2}
\le CQ_{\eps}(t)^2 (1+s)^{-1} ,  \label{h4}
\ema
where $j=1,2,$ $1\le|\alpha|\le N-1$, $0\le s\le t$ and $C>0$ is a constant.

Since $P_dG_1(f_{\eps} )=0$ and $P_0G_2(f_{\eps} )=0$, it follows from Lemma \ref{time}, \eqref{fe} and \eqref{h1}--\eqref{h4} that
\iffalse \bma
\|P_0f_{\eps} (t)\|_{L^2_{P} }&\le
C(1+t)^{-\frac14}(\|f_0\|_{L^{2}_{x,v}}+\|f_0\|_{L^{2}_v(L^1_x)})\nnm\\
&\quad+C\sum^2_{j=1}\int^{t}_0 \( (1+t-s)^{-\frac34}+\eps^{-1}e^{-\frac{d(t-s)}{\eps^2}}\)\nnm\\
&\qquad\qquad\times(\| G_j(s)\|_{L^{2}_{x,v}}+\|G_j(s)\|_{L^{2}_v(L^1_x)})ds\nnm\\
&\le C\delta_0(1+t)^{-\frac14}+CQ_{\eps}(t)^2\int^{t}_0
(1+t-s)^{-\frac34}(1+s)^{-\frac12} ds\nnm\\
&\le
C\delta_0(1+t)^{-\frac14}+CQ_{\eps}(t)^2(1+t)^{-\frac14}, \label{density_1}
 \ema
and \fi
 \bma
 \|\dxa P_0f_{\eps} (t)\|_{L^2_{P}}
 &\le
 C(1+t)^{-\frac14-\frac{|\alpha|}2}\(\|\dxa f_0\|_{L^{2}_{x,v}}+\|f_0\|_{L^{2}_v(L^1_x)}\)
 \nnm\\
&\quad+C\sum^2_{j=1}\int^{t/2}_0 \((1+t-s)^{-\frac34-\frac{|\alpha|}2}+\frac1{\eps}e^{-\frac{d(t-s)}{\eps^2}}\)\nnm\\
&\qquad\qquad\times\(\|\dxa G_j(s)\|_{L^{2}_{x,v}}
  +\| G_j(s)\|_{L^{2}_v(L^1_x)}\)ds
  \nnm\\
&\quad+C\sum^2_{j=1}\int^{t}_{t/2} \((1+t-s)^{-\frac34}+\frac1{\eps}e^{-\frac{d(t-s)}{\eps^2}}\)\nnm\\
&\qquad\qquad\times\(\|\dxa G_j(s)\|_{L^{2}_{x,v}}
  +\|\dxa  G_j(s)\|_{L^{2}_v(L^1_x)}\)ds
 \nnm\\
&\le
 C\delta_0(1+t)^{-\frac14-\frac{|\alpha|}2}+CQ_{\eps}(t)^2 (1+t)^{-\frac14-\frac{|\alpha|}2},\label{density_1}
 \ema
for $|\alpha|=0,1$.
By \eqref{G_1c}, \eqref{density_1} and noting that $D_1(f_{\eps})\ge c\mathcal{H}_1(f_{\eps})$ for $c>0$ a constant, we obtain
 \bma
  H_1(f_{\eps}(t))&\le Ce^{-c t}H_1(f_0)
  +\intt e^{-c(t-s)}\|\Tdx  P_0f_{\eps}(s)\|^2_{L^{2}_{x,v}}ds
  \nnm\\
&\le
 C \delta_0^2e^{-c t}
  +C(\delta_0+Q_{\eps}(t)^2 )^2\intt e^{-c(t-s)}(1+s)^{-\frac32} ds
  \nnm\\
&\le
  C(\delta_0+Q_{\eps}(t)^2 )^2(1+t)^{-\frac32} .  \label{Hf2}
\ema
Combining \eqref{density_1}--\eqref{Hf2}, we obtain
$$
Q_{\eps}(t)\le C\delta_0+CQ_{\eps}(t)^2 ,
$$
which leads to \eqref{assume} provided $\delta_0>0$ sufficiently small.

Then, we deal with \eqref{G_5}. For any $f_0,g_0\in L^2_v(L^{2}_{x})$ satisfying $P_df_0=0$ and $P_0g_0=0$, by \eqref{D1b}, \eqref{t4} and \eqref{t6} we can obtain
\bma
\| P_1e^{\frac{t}{\eps^2}B_{\eps}}f_0\|_{L^{2}_{x,v}}&\le C\(\eps(1+t)^{-\frac34}+e^{-\frac{dt}{\eps^2}}\)(\|f_0\|_{L^2_v(H^1_x)}+\|\Tdx f_0\|_{L^{2}_v(L^1_x)}), \label{t7}\\
\| P_1e^{\frac{t}{\eps^2}B_{\eps}}g_0\|_{L^{2}_{x,v}}&\le C\(\eps^2(1+t)^{-\frac54}+e^{-\frac{dt}{\eps^2}}\)(\|g_0\|_{L^2_v(H^2_x)}+\| \Tdx g_0\|_{L^{2}_v(L^1_x)}).\label{t8}
\ema
Thus, it follows from Lemma \ref{time}, \eqref{t7} and \eqref{t8}  that
 \bma
\|\dxa P_1f_{\eps} (t)\|_{L^{2}_{x,v}}&\le
C\(\eps(1+t)^{-\frac34-\frac{|\alpha|}2 }+e^{-\frac{dt}{\eps^2}}\)\(\|\dxa f_0\|_{L^2_v(H^1_x)}+\|f_0\|_{L^{2}_v(L^1_x)}\)\nnm\\
&\quad+C\sum^2_{j=1}\int^{t/2}_0 \(\eps(1+t-s)^{-\frac54-\frac{|\alpha|}2 }+\frac1{\eps}e^{-\frac{d(t-s)}{\eps^2}}\)\nnm\\
&\qquad\qquad\times\(\|\dxa G_j(s)\|_{L^2_v(H^2_x)}+\|G_j(s)\|_{L^{2}_v(L^1_x)}\)ds\nnm\\
&\quad+C\sum^2_{j=1}\int^{t}_{t/2} \(\eps(1+t-s)^{-\frac34 }+\frac1{\eps}e^{-\frac{d(t-s)}{\eps^2}}\)\nnm\\
&\qquad\qquad\times\(\|\dxa G_j(s)\|_{L^2_v(H^2_x)}+\|\dxb G_j(s)\|_{L^{2}_v(L^1_x)}\)ds\nnm\\
&\le
(C\delta_0 +CQ_{\eps}(t)^2)\(\eps(1+t)^{-\frac34}+e^{-\frac{dt}{4\eps^2}}\), \label{micro_1}
 \ema
where $|\beta|=1$ for $\alpha=0$, $\beta=\alpha$ for  $1\le |\alpha|\le N-3$. By \eqref{micro_1}, we can obtain \eqref{G_5}.

If $f_0$ satisfies \eqref{initial}, we have
$$
\(P_0\hat{f}_0,h_j(\xi)\)_{\xi}=0,\quad j=-1,1.
$$
This and Theorems \ref{spect3}, \ref{E_3} imply that
\be
\|\dxa e^{\frac{t}{\eps^2}B_{\eps}}f_0\|_{L^2_P}\le C(1+t)^{-\frac34-\frac{m}2}(\|\dxa f_0\|_{L^2_{x,v}}+\|\dx^{\alpha'}f_0\|_{L^2_v(L^1_x)}), \label{time9}
\ee
where $\alpha'\le \alpha$, $m=|\alpha-\alpha'|$. By \eqref{time9}, Lemma \ref{time} and a similar argument as above, we can prove \eqref{G_11a}--\eqref{G_12a}.
The proof of the lemma is completed.
\end{proof}

By \eqref{ue} and Lemma \ref{timev}, we have

\begin{lem}\label{energy3} Let $N\ge 4$. There exists a small constant $\delta_0>0$ such that if $\|f_{0}\|_{L^2_v(H^N_x)}+\|f_0\|_{L^{2}_v(L^1_x)}\le \delta_0$, then
the NSPF system \eqref{NSP_1}--\eqref{compatible} admits a unique global solution $(n,m,q)(t,x)\in L^{\infty}_t(H^N_x)$.  Moreover,  $u(t,x,v)=n(t,x)\chi_0+m(t,x)\cdot v\chi_0+q(t,x)\chi_4$ has the following time-decay rate:
\be
   \|u(t)\|_{L^2_v(H^N_x)}+\|\Tdx\phi(t)\|_{H^N_x } \le C\delta_0 (1+t)^{-\frac34}, \label{time8}
\ee
where $\phi(t,x)=\Delta^{-1}_xn(t,x)$ and $C>0$ is a constant.
\end{lem}
\begin{proof}
\iffalse By \eqref{ue}, we have
\be
u(t)=V(t)P_0f_0+\intt V(t-s)\[Z_1(u)+\divx Z_2(u)\]ds. \label{ue1}
\ee
\fi

Define
$$
Q(t)=\sup_{0\le s\le t}\left\{(1+s)^{3/4}\(\|u(s)\|_{L^2_v(H^N_x)}+\|\Tdx\phi(s)\|_{H^N_x }\)\right\}.
$$
It holds that
\be
\| Z_j(u(s))\|_{L^{2}_{x}(H^N_x)}+\|Z_j(u(s))\|_{L^{2}_v(L^1_x)}\le CQ (t)^2 (1+s)^{-\frac32}  ,  \label{h1a}
\ee
for $j=1,2,$  $0\le s\le t$ and $C>0$  a constant. By \eqref{v8}, we obtain that for any vector $U_0\in \R^3$,
\be
\|\dxa V(t)\divx U_0\|_{L^2_{P}}\le C\((1+t)^{-\frac54-\frac{m}2}+t^{-\frac 12}e^{-ct}\)\(\|\dxa U_0\|_{L^{2}_{x,v}}+\|\dx^{\alpha'}U_0\|_{L^{2}_v(L^1_x)}\), \label{ccc}
\ee
where $\alpha'\le \alpha$, $m=|\alpha-\alpha'|$ and $c>0$ is a constant.

Then, it follows from Lemma \ref{timev}, \eqref{ue}, \eqref{h1a} and \eqref{ccc} that
 \bma
 \|\dxa u (t)\|_{L^2_{P}}
 &\le
 C(1+t)^{-\frac34-\frac{|\alpha|}2}(\|\dxa P_0f_0\|_{L^{2}_{x,v}}+\|P_0f_0\|_{L^{2}_v(L^1_x)})
 \nnm\\
&\quad+C\sum_{j=1,2}\intt \((1+t-s)^{-\frac34-\frac{|\alpha|}2}+(t-s)^{-\frac12}e^{-c(t-s)}\)\nnm\\
&\qquad\qquad\times(\|\dxa Z_j(s)\|_{L^{2}_{x,v}}
  +\| Z_j(s)\|_{L^{2}_v(L^1_x)})ds
  \nnm\\
&\le
 C\delta_0(1+t)^{-\frac34}+CQ(t)^2 (1+t)^{-\frac34},\label{density_2}
 \ema
for $0\le |\alpha|\le N$.
By \eqref{density_2}, we can obtain
$$
Q(t)\le C\delta_0 ,
$$
 provided $\delta_0>0$ sufficiently small. This proves \eqref{time8}.
The existence of the solution can be proved by the contraction mapping theorem, the details are omitted.
 The proof of the lemma is completed.
\end{proof}

\begin{lem}\label{gamma1}For any $i,j=1,2,3,$ it holds that
\bma
\Gamma(v_i\chi_0,v_j\chi_0) &=-\frac12LP_1(v_iv_j\chi_0), \\
\Gamma(v_i\chi_0,|v|^2\chi_0) &=-\frac12LP_1(v_i|v|^2\chi_0), \\
\Gamma(|v|^2\chi_0,|v|^2\chi_0)&=-\frac12LP_1(|v|^4\chi_0).
\ema
\end{lem}

\begin{proof}
Let $u=v_*$, $u'=v'_*$.
Since
$$ v_i+u_i=v'_i+u'_i,\quad i=1,2,3,\quad |v|^2+|u|^2=|v'|^2+|u'|^2, $$
it follows that
\be \label{a}
\left\{\bln
&(v_i+u_i)(v_j+u_j)=(v'_i+u'_i)(v'_j+u'_j),\quad i,j=1,2,3,\\
&(v_i+u_i)(|v|^2+|u|^2)=(v'_i+u'_i)(|v'|^2+|u'|^2),\quad i=1,2,3, \\
&(|v|^2+|u|^2)(|v|^2+|u|^2)= (|v'|^2+|u'|^2)(|v'|^2+|u'|^2).
\eln\right.
\ee
By \eqref{Q}, \eqref{L_1} and \eqref{a}, we have
\bmas
\Gamma(v_i\chi_0,v_j\chi_0)=&\frac12\sqrt{M}\intr\ints B(|v-u|,\omega)(u'_iv'_j+v'_iu'_j-u_iv_j-v_iu_j)M(u) d\omega du\\
=&-\frac12\sqrt{M}\intr\ints B(|v-u|,\omega)(u'_iu'_j+v'_iv'_j-v_iv_j-u_iu_j)M(u)d\omega du\\
=&-\frac12LP_1(v_iv_j\chi_0),
\emas
and
\bmas
\Gamma(v_i\chi_0,|v|^2\chi_0)
=&\frac12\sqrt{M}\intr \ints B(|v-u|,\omega)(u'_i|v'|^2+|v'|^2u'_i-u_i|v|^2-|u|^2v_i)M(u)d\omega du\\
=&-\frac12\sqrt{M}\intr \ints B(|v-u|,\omega)(u'_i|u'|^2+v'_i|v'|^2-v_i|v|^2-u_i|u|^2)M(u)d\omega du\\
=&-\frac12LP_1(v_i|v|^2\chi_0),
\emas
and
\bmas
\Gamma(|v|^2\chi_0,|v|^2\chi_0)
=& \sqrt{M}\intr \ints B(|v-u|,\omega)(|u'|^2|v'|^2  - |v|^2|u|^2 )M(u)d\omega du\\
=&-\frac12\sqrt{M}\intr \ints B(|v-u|,\omega)(|u'|^4 +|v'|^4 - |v|^4-|u|^4 )M(u)d\omega du\\
=&-\frac12LP_1(|v|^4\chi_0).
\emas
This proves the lemma.
\end{proof}

Theorem \ref{thm1.1}  follows from Lemmas \ref{time7} and \ref{energy3}.

\begin{proof}[\underline{\bf Proof of Theorem \ref{thm1.2}}]
Define
\be
Q_{\eps}(t)=\sup_{0\le s\le t}\bigg(\eps^a (1+s)^{-\frac12}+ \bigg(1+\frac{s}{\eps}\bigg)^{-b}\bigg)^{-1}
\|f_{\eps}(s)-u(s) \|_{ L^{\infty}_P},
\ee
where the norm $\|\cdot\|_{L^\infty_P}$ is defined by \eqref{norm2}.

We claim that
\be
Q_{\eps}(t) \le C\delta_0 ,\quad \forall t>0. \label{limit6}
\ee
It is easy to verify that the estimate  \eqref{limit_1} follows  from \eqref{limit6}.

 By \eqref{fe} and \eqref{ue}, we have
\bma
&\|f_{\eps}(t)-u(t) \|_{ L^{\infty}_P}\nnm\\
\le& \left\|e^{\frac{t}{\eps^2}B_{\eps}}f_0-V(t)P_0f_0\right\|_{ L^{\infty}_P}+\intt \left\| e^{\frac{t-s}{\eps^2}B_{\eps}}G_1(f_{\eps})-V(t-s)Z_1(u)\right\|_{ L^{\infty}_P}ds\nnm\\
&+\intt \bigg\|\frac1{\eps}e^{\frac{t-s}{\eps^2}B_{\eps}}G_2(f_{\eps})-V(t-s) \divx Z_2(u)\bigg\|_{ L^{\infty}_P}ds\nnm\\
=:&I_1+I_2+I_3. \label{c1}
\ema
By \eqref{limit1}, we can bound $I_1$ by
\bma
I_1&\le C \bigg(\eps (1+t)^{-\frac32}+ \bigg(1+\frac{t}{\eps}\bigg)^{-r}\bigg)\(\|f_0\|_{L^2_v(H^4_x)}+\|f_0\|_{L^2_v(W^{4,1}_x)}+\|\Tdx\phi_0\|_{L^p_x}\)\nnm\\
&\le C\delta_0\bigg(\eps (1+t)^{-\frac32}+ \bigg(1+\frac{t}{\eps}\bigg)^{-r}\bigg),
\ema
where $r=3/p-3/2\in (0,3/2)$ and $p\in (1,2)$.

To estimate $I_2$, we decompose
\bma
I_2\le& \intt \left\| e^{\frac{t-s}{\eps^2}B_{\eps}}G_1(f_{\eps})-V(t-s)P_0G_1(f_{\eps})\right\|_{ L^{\infty}_P}ds\nnm\\
&+\intt\left\|V(t-s)P_0G_1(f_{\eps})-V(t-s)Z_1(u)\right\|_{ L^{\infty}_P}ds\nnm\\
=:&I_{21}+I_{22}.
\ema

By \eqref{limit1}, \eqref{h1}--\eqref{h4}, and noting that $(G_1(f_{\eps}),\chi_0)=0$, we have
\bma
I_{21}\le &C \intt  \bigg(\eps (1+t-s)^{-\frac32}+ \bigg(1+\frac{t-s}{\eps}\bigg)^{-q}\bigg)\nnm\\
&\qquad\times\(\|G_1(f_{\eps})\|_{L^2_v(H^4_x)}+\|G_1(f_{\eps})\|_{L^2_v(W^{4,1}_x)}\)ds\nnm\\
\le &C \delta_0^2\intt \bigg(\eps (1+t-s)^{-\frac32}+ \bigg(1+\frac{t-s}{\eps}\bigg)^{-q}\bigg)(1+s)^{-\frac12}ds \nnm\\
\le &C \delta_0^2\bigg(\eps (1+t)^{-\frac12}+ \bigg(1+\frac{t}{\eps}\bigg)^{-q}\bigg), \label{I21}
\ema
for any $q\in(1,3/2)$.

Since
$$P_0G_1(f_{\eps})=\( n_{\eps}\Tdx\phi_{\eps}\)\cdot v\chi_0+\sqrt{\frac23}\(m_{\eps}\cdot\Tdx\phi_{\eps}\)\chi_4,$$
it follows from \eqref{time1a} that
\bma
I_{22}\le &C\intt (1+t-s)^{-\frac34}\nnm\\
&\quad\times\Big(\|f_{\eps}-u\|_{L^{\infty}_x(L^2_v)}\|\Tdx\phi_{\eps}\|_{H^{2}_x}+\| u\|_{L^2_v( H^{2}_x)}\|\Tdx\phi_{\eps}-\Tdx\phi\|_{L^{\infty}_x}\Big)ds\nnm\\
\le &C \delta_0Q_{\eps}(t)\intt  (1+t-s)^{-\frac34}\bigg(\eps^a (1+s)^{-\frac12}+ \bigg(1+\frac{s}{\eps}\bigg)^{-b}\bigg)(1+s)^{-\frac14} ds. \label{I22}
\ema
Denote
\be J_0=\intt  (1+t-s)^{-\frac34}  \bigg(1+\frac{s}{\eps}\bigg)^{-b} (1+s)^{-\frac14}  ds,\quad b\in (0,1]. \label{I23}\ee
We estimate $J_0$ as follows.
\bma
J_0 &=\(\int^{t/2}_0+\int^{t}_{t/2}\)  (1+t-s)^{-\frac34}  \bigg(1+\frac{s}{\eps}\bigg)^{-1}(1+s)^{-\frac14}  ds\nnm\\
&\le C\eps|\ln\eps| (1+t)^{-\frac12}+C\bigg(1+\frac{t}{\eps}\bigg)^{-1},\quad b=1,
\ema
and
\bma
J_0&=\(\int^{t/2}_0+\int^{t}_{t/2}\)  (1+t-s)^{-\frac34}  \bigg(1+\frac{s}{\eps}\bigg)^{-b}(1+s)^{-\frac14}  ds\nnm\\
&\le C\eps^b (1+t)^{-\frac34}\int^{t/2}_0 s^{-b}(1+s)^{-\frac14}ds+C\bigg(1+\frac{t}{\eps}\bigg)^{-b}\nnm\\
&\le C\eps^{b}(1+t)^{-r(b)} +C\bigg(1+\frac{t}{\eps}\bigg)^{-b},\quad 0< b<1, \label{I24}
\ema
where $r(b)=\frac14-b$ when $b\in[3/4,1)$, $r(b)=b$ when $b\in(0,3/4)$.
Thus, it follows from \eqref{I22}--\eqref{I24} that
\be
I_{22}\le C \delta_0 Q_{\eps}(t) \bigg(\eps^a (1+t)^{-\frac12}+ \bigg(1+\frac{t}{\eps}\bigg)^{-b}\bigg).\label{I25}
\ee

To estimate $I_3$, we decompose
\bma
I_3\le& \intt \left\| \frac1{\eps}e^{\frac{t-s}{\eps^2}B_{\eps}}G_2(f_{\eps})-V(t-s)P_0\(v\cdot\Tdx L^{-1}G_2(f_{\eps})\)\right\|_{ L^{\infty}_P}ds\nnm\\
&+\intt\left\|V(t-s)P_0\(v\cdot\Tdx L^{-1}G_2(f_{\eps})\)-V(t-s)\divx Z_2\right\|_{ L^{\infty}_P}ds\nnm\\
=:&I_{31}+I_{32}. \label{I3}
\ema
By \eqref{limit2}, \eqref{h1}--\eqref{h4}, and noting that $P_0G_2(f_{\eps})=0$, we have
\bma
I_{31}\le &C \intt  \bigg(\eps (1+t-s)^{-\frac52}+ \bigg(1+\frac{t-s}{\eps}\bigg)^{-q}+\frac1{\eps}e^{-\frac{d(t-s)}{\eps^2}}\bigg)\nnm\\
&\qquad\times\(\|G_2(f_{\eps})\|_{L^2_v(H^5_x)}+\|G_2(f_{\eps})\|_{L^2_v(W^{5,1}_x)}\)ds\nnm\\
\le &C \delta_0^2\intt \bigg(\eps (1+t-s)^{-\frac52}+ \bigg(1+\frac{t-s}{\eps}\bigg)^{-q}+\frac1{\eps}e^{-\frac{d(t-s)}{\eps^2}}\bigg)(1+s)^{-\frac12}ds\nnm \\
\le &C \delta_0^2 \bigg(\eps(1+t)^{-\frac12}  + \bigg(1+\frac{t}{\eps}\bigg)^{-q}\bigg),\label{I31}
\ema
for any $q\in(1,3/2)$.

To estimate $I_{32}$, we decompose
\bmas
&P_0\(v\cdot\Tdx L^{-1}G_2(f_{\eps})\)\nnm\\
=&P_0(v\cdot\Tdx L^{-1}\Gamma(P_0f_{\eps},P_0f_{\eps}))+2P_0(v\cdot\Tdx L^{-1}\Gamma(P_0f_{\eps},P_1f_{\eps}))\\
&+P_0(v\cdot\Tdx L^{-1}\Gamma(P_1f_{\eps},P_1f_{\eps}))
=:J_1+J_2+J_3.
\emas
By Lemma \ref{gamma1}, we have
\bmas
%&P_0(v\cdot\Tdx L^{-1}\Gamma(P_0f_{\eps},P_0f_{\eps}))\\
J_1=&-\frac12\sum^3_{i,j,k,l=1}\(\pt_k(m^i_{\eps}m^j_{\eps}P_1(v_iv_j\chi_0)),v_kv_l\chi_0\)v_l\chi_0 \\ &-\frac1{12}\sum^3_{j=1}\(\pt_j(q_{\eps}^2P_1(|v|^4\chi_0)),v_j^2\chi_0\)v_j\chi_0 -\sum^3_{j=1}\(\pt_j(m^j_{\eps}q_{\eps}P_1(v_j\chi_4)),v_j\chi_4\)\chi_4.
\emas
Since
\bmas
&P_1(v_iv_j\chi_0)=(v_iv_j-\delta_{ij}|v|^2/3)\chi_0,\quad i,j=1,2,3,\\
&P_1(v_j\chi_4)=v_j(\chi_4-\sqrt{2/3}\chi_0),\quad j=1,2,3,\\
&P_1(|v|^4\chi_0)=|v|^4 \chi_4-15\chi_0-10\sqrt6\chi_4,
\emas
it follows that
\bmas
&(P_1(v_iv_j\chi_0),v_iv_j\chi_0)=1,\quad i\ne j,\\
&(P_1(v_i^2\chi_0),v_i^2\chi_0)=4/3,\quad (P_1(v_i^2\chi_0),v_j^2\chi_0)=-2/3,\quad j\ne i,\\
&(P_1(v_j\chi_4),v_j\chi_4)=5/3,\quad (P_1(|v|^4\chi_0),v_j^2\chi_0)=0,\quad j=1,2,3.
\emas
Thus
$$
J_1=-\sum^3_{i, j=1}\pt_i(m^i_{\eps} m^j_{\eps})v_j\chi_0+\frac13\sum^3_{i,j=1}\pt_j( m^i_{\eps})^2 v_j\chi_0 -\frac53\sum^3_{j=1}\pt_j(m^j_{\eps}q_{\eps}) \chi_4.
$$
This  and \eqref{v1} give
\be
V(t)J_1=-V(t)\divx\[(m_{\eps}\otimes m_{\eps})\cdot v\chi_0+\frac53 (q_{\eps}m_{\eps})\chi_4\]=:V(t)\divx J_4. \label{eee1}
\ee
By \eqref{time1a}, \eqref{G_11}, \eqref{G_5}, \eqref{ccc} and \eqref{eee1}, we have
\bma
I_{32}\le &\intt  \left\|V(t-s)\divx (J_4-Z_2)\right\|_{ L^{\infty}_P}ds+\sum^3_{k=2}\intt  \left\|V(t-s)J_k\right\|_{ L^{\infty}_P}ds\nnm\\
\le  &C\intt (1+t-s)^{-\frac34}(t-s)^{-\frac12} \|f_{\eps}-u\|_{L^{\infty}_x(L^2_v)}\(\|f_{\eps}\|_{ L^2_v(H^{2}_x)}+\| u\|_{ L^2_v(H^{2}_x)}\)ds\nnm\\
  &+C \intt (1+t-s)^{-\frac34}\|P_1f_{\eps}\|_{ L^2_v(H^{3}_x)}\(\|P_0f_{\eps}\|_{ L^2_v(H^{3}_x)}+\|P_1f_{\eps}\|_{ L^2_v(H^{3}_x)}\)ds\nnm\\
\le &C \delta_0Q_{\eps}(t)  \intt  (1+t-s)^{-\frac34}(t-s)^{-\frac12}\bigg(\eps^a (1+s)^{-\frac12}+ \bigg(1+\frac{s}{\eps}\bigg)^{-b}\bigg)(1+s)^{-\frac14} ds \nnm\\
 &+C \delta_0^2 \intt  (1+t-s)^{-\frac34}\(\eps (1+s)^{-\frac34}+ e^{-\frac{ds}{\eps^2}}\)(1+s)^{-\frac14} ds. \label{I32}
\ema
Denote
\be
J_5= \intt  (1+t-s)^{-\frac34}(t-s)^{-\frac12}  \bigg(1+\frac{s}{\eps}\bigg)^{-b} (1+s)^{-\frac14} ds,\quad b\in(0,1].
\ee
We estimate $J_5$ as follows.
\bma
J_5&= \(\int^{t/2}_0+\int^{t}_{t/2}\)(1+t-s)^{-\frac34}(t-s)^{-\frac12}  \bigg(1+\frac{s}{\eps}\bigg)^{-1}(1+s)^{-\frac14}  ds\nnm\\
&\le C\eps|\ln\eps| t^{-\frac12}(1+t)^{-\frac12}+C\bigg(1+\frac{t}{\eps}\bigg)^{-1}\nnm\\
&\le C\eps|\ln\eps|^2  (1+t)^{-1}+C\bigg(1+\frac{t}{\eps}\bigg)^{-1},\quad b=1,\,\, t\ge \eps,
\ema
and
\bma
J_5&= \(\int^{t/2}_0+\int^{t}_{t/2}\)(1+t-s)^{-\frac34}(t-s)^{-\frac12}  \bigg(1+\frac{s}{\eps}\bigg)^{-b}(1+s)^{-\frac14}  ds\nnm\\
&\le C\eps^bt^{-\frac12} (1+t)^{-\frac34}\int^{t/2}_0 s^{-b}(1+s)^{-\frac14}ds+C\bigg(1+\frac{t}{\eps}\bigg)^{-b}\nnm\\
&\le C\bigg(1+\frac{t}{\eps}\bigg)^{-b},\quad 0< b<1,\,\, t\ge \eps.  \label{I33}
\ema
Thus, it follows from \eqref{I32}--\eqref{I33} that
\be
I_{32}\le C(\delta_0^2+ \delta_0Q_{\eps}(t)) \bigg(\eps^a (1+t)^{-\frac12}+ \bigg(1+\frac{t}{\eps}\bigg)^{-b}\bigg).\label{limit3}
\ee

By combining \eqref{c1}--\eqref{I21}, \eqref{I25}--\eqref{I31} and \eqref{limit3}, we obtain
\be
Q_{\eps}(t)\le  C\delta_0+C\delta_0^2+C\delta_0Q_{\eps}(t),
\ee
where $C>0$ is a constant independent of  $\eps$.
By taking $\delta_0>0$  small enough, we can obtain \eqref{limit6}, which proves \eqref{limit_1}. By \eqref{limit1a}, \eqref{G_11a}, \eqref{G_12a} and a similar argument as above, we can prove \eqref{limit_1a}. The proof is completed.
\end{proof}

\medskip
\noindent {\bf Acknowledgements:} The authors would like to thank Yan Guo for
some insightful discussions. The research of the first author was supported partially by the National Science Fund for Distinguished Young Scholars No. 11225102, the National Natural Science Foundation of China Nos. 11871047, 11671384 and 11461161007, and the Capacity Building for Sci-Tech Innovation-Fundamental Scientific Research Funds 00719530050166.
The research of the second author was supported by the General Research Fund of Hong Kong, CityU 11302518, and the  Fundamental Research Funds for the Central Universities No.2019CDJCYJ001.
The research of the third author is supported by the National Natural Science Foundation of China No. 11671100, the National Science Fund for Excellent Young Scholars No. 11922107, and Guangxi Natural Science Foundation No. 2018GXNSFAA138210.

\end{document}